%% file: W1TypeMetrics.tex
\documentclass[a4paper,11pt]{article}
\usepackage{LatexDefinitions}

\renewcommand{\d}{{\mathrm{d}}}

\renewcommand{\diam}{{\mathrm{diam}}}

\newcommand{\alphamin}{\alpha_{\min}}
\newcommand{\betamin}{\beta_{\min}}
\newcommand{\tildealphamin}{\tilde{\alpha}_{\min}}
\newcommand{\tildebetamin}{\tilde{\beta}_{\min}}

\newcommand{\meas}{\mc{M}}
\newcommand{\measp}{\mc{M}_+}

\newcommand{\dom}{\mathrm{dom}}
\renewcommand{\div}{\mathrm{div}}
\newcommand{\prox}{\mathrm{prox}}
\newcommand{\RadNik}[2]{{\tfrac{\d#1}{\d#2}}}
\newcommand{\preconj}[1]{\vphantom{#1}^\ast#1}
\newcommand{\cl}{\mathrm{cl}}
\newcommand{\measnrm}[1]{#1(\Omega)}

\newcommand{\h}{h_0}
\newcommand{\g}{h_1}

\newcommand{\hB}{h_{01}}
\newcommand{\gB}{h_{10}}
\newcommand{\Cl}{c_0}
\newcommand{\Cm}{c_{01}}
\newcommand{\Cr}{c_1}
\newcommand{\Dl}{{D_0}}
\newcommand{\Dm}{{D_{01}}}
\newcommand{\Dr}{{D_1}}

\newcommand{\SimFR}{D^{\tn{H}}}
\newcommand{\SimTV}{D^{\tn{TV}}}
\newcommand{\SimDisc}{D^{\tn{d}}}
\newcommand{\SimJS}{D^{\tn{JS}}}
\newcommand{\SimChi}{D^{\chi^2}}
\newcommand{\SimE}[1]{D^{\tn{E},#1}}
\newcommand{\SimEp}{\SimE{p}}
\DeclareMathOperator{\TVX}{TVX}

\title{A Framework for Wasserstein-$1$-Type Metrics}
\author{Bernhard Schmitzer \and Benedikt Wirth}
\date{}

\begin{document}

\maketitle
\begin{abstract}
We propose a unifying framework for generalising the Wasserstein-1 metric to a discrepancy measure between nonnegative measures of different mass.
This generalization inherits the convexity and computational efficiency from the Wasserstein-1 metric,
and it includes several previous approaches from the literature as special cases.
For various specific instances of the generalized Wasserstein-1 metric we furthermore demonstrate their usefulness in applications by numerical experiments.
\end{abstract}
\tableofcontents

\input{01-intro}
\input{02-w1-extension}
\input{03-examples}
\input{04-algorithm}

\input{05-numerics}
\input{06-conclusion}

\bibliography{references}{}
\bibliographystyle{plain}

\end{document}

%% file: 01-intro.tex

\section{Introduction}
\subsection{Motivation}
Optimal transport and Wasserstein metrics are becoming increasingly popular tools in many applied fields.
In particular, they provide a robust and intuitive measure of discrepancy for histograms and mass distributions,
which can for instance be exploited in image processing for image interpolation \cite{AngenentOTFlow2003}, deformation analysis \cite{OptimalTransportTangent2012}, colour transfer \cite{Ferradans-ICIP14}, cartoon-texture decomposition \cite{LellmannKantorovichRubinstein2014}, or shape averaging \cite{Solomon-siggraph-2015}.

A significant practical limitation of Wasserstein metrics is their high naive computational complexity. Various algorithms and methods for computational acceleration and efficient approximation have been proposed: a flow formulation \cite{BenamouBrenier2000}, entropic regularization \cite{Cuturi2013,BenamouIterativeBregman2015} and adaptive sparse solvers \cite{SchmitzerShortCuts2015} among others.
The Wasserstein-1 metric $W_1$ represents an exception, since it can be reformulated as a minimal cost flow problem with local constraints, leading to a significant reduction of the problem size.

Plain Wasserstein metrics can only compare nonnegative measures of equal mass, which often does not reflect the requirements of applications.
Therefore, early on, ad hoc extensions of optimal transport to unbalanced measures have been proposed (e.\,g.\ \cite{RubnerEMD-IJCV2000,Benamou-Unbalanced-2003,PeleECCV2008}). An extension of $W_1$ that retains its computational efficiency can be found in \cite{LellmannKantorovichRubinstein2014}.
More recently, unbalanced optimal transport has been studied from a dynamic \cite{DNSTransportDistances09} and geometric perspective (\cite{KMV-OTFisherRao-2015,ChizatOTFR2015,LieroMielkeSavare-HellingerKantorovich-2015a}, see also \cite{ChizatDynamicStatic2018}).

Figure\,\ref{fig:transportBio} illustrates the need for unbalanced versions of $W_1$.
It displays two consecutive frames from a microscopy video of biological cells showing the temporally evolving distribution of a specific molecule.
The optimal mass flow with respect to the $W_1$ distance would be a natural estimate of the actual molecule flux,
however, due to inhomogeneous brightness changes, image acquisition noise, and other artefacts the resulting mass flow is useless.
This is remedied if the mass is allowed to slightly change (to accommodate the above artefacts), using an unbalanced version of $W_1$, so that the resulting mass flow indeed correctly identifies the motion of several distinct biological structures.

This article provides a unified framework for unbalanced extensions of $W_1$,
thereby covering several previous approaches as special cases and providing a new family of efficient transport-based discrepancy measures.

\begin{figure}
  \centering
  \includegraphics{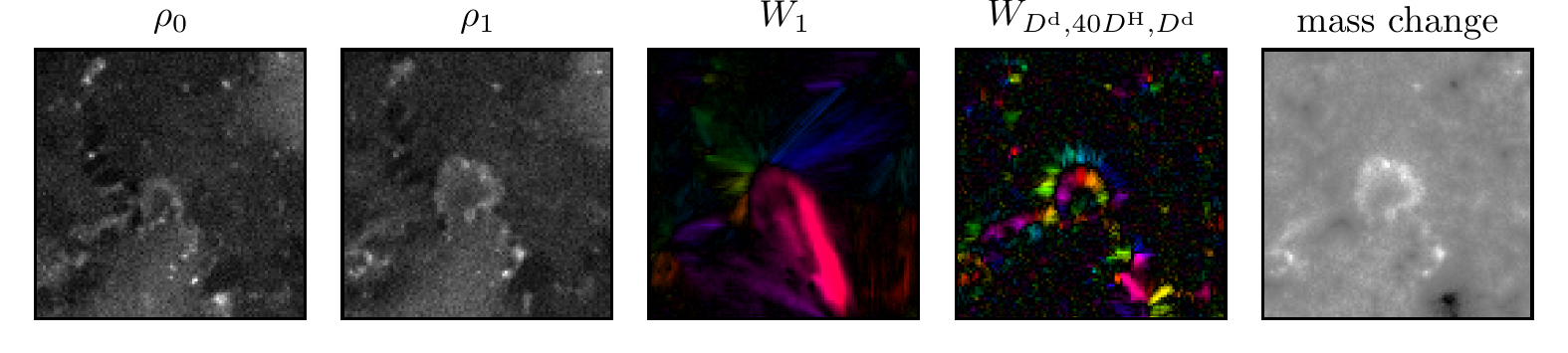}%
  \setlength\unitlength\linewidth%
  \begin{picture}(0,0)
  \put(-.587,.02){\includegraphics[width=.04\unitlength,trim=118 90 23 15,clip]{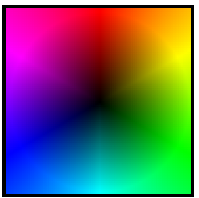}}
  \end{picture}%
  \caption[]{Extracting flows from video sequences via an unbalanced extension of $W_1$. %
  \textit{From left to right:}
  Two consecutive video frames $\rho_0$ and $\rho_1$ show the distribution of a fluorescent molecule at the boundary between two biological cells (data courtesy Hans-Joachim Schnittler, M\"unster). %
  Flow extracted from $W_1$ transport between normalized images (colour-coding according to inset). %
  Flow extracted using a particular unbalanced transport discrepancy $W_{\SimDisc,40 \SimFR, \SimDisc}$. %
  Change of mass during unbalanced transport (bright means increase). %
  }
  \label{fig:transportBio}
\end{figure}

\subsection{Overview: Old and new Wasserstein-1 metrics}
\label{sec:IntroOverview}
The Wasserstein-1 metric between two nonnegative measures $\rho_0$ and $\rho_1$ on a domain $\Omega$
(which has to fulfill certain properties but can simply be thought of as a subset of $\R^n$ for the time being)
is given by
\begin{equation}\label{eqn:W1primal}
W_1(\rho_0,\rho_1) = \inf_{\pi \in \Gamma(\rho_0,\rho_1)}\int_{\Omega \times \Omega} d(x,y)\,\d \pi(x,y)\,,
\end{equation}
where $d$ denotes the distance measure between points in $\Omega$
and $\Gamma(\rho_0,\rho_1)$ is the set of so-called couplings, that is,
nonnegative measures $\pi$ on $\Omega\times\Omega$ satisfying
\begin{equation*}
\pi(B\times\Omega)=\rho_0(B)\quad\text{ and }\quad\pi(\Omega\times B)=\rho_1(B)
\end{equation*}
for all measurable $B\subset\Omega$.
By convention, the infimum is infinite for empty $\Gamma(\rho_0,\rho_1)$, which is equivalent to $\rho_0(\Omega)\neq\rho_1(\Omega)$.
The metric $W_1(\rho_0,\rho_1)$ can be interpreted as the cost of transporting the mass $\rho_0$ to the final distribution $\rho_1$,
where the cost contribution of each mass particle is its transport distance, and $\pi(x,y)$ can be interpreted as the mass transported from $x$ to $y$.
Other optimal transport problems are obtained when replacing the distance $d$ by other nonnegative cost functions.

The above formulation is based on a variable $\pi$ that lives on the high-dimensional space $\Omega\times\Omega$,
thus making a direct implementation infeasible except for special cases.
However, $W_1$ has the particular property that it can also be expressed by the Kantorovich--Rubinstein formula \cite[Rk.~5.16]{Villani-OptimalTransport-09}
\begin{multline}\label{eqn:W1predual}
	W_1(\rho_0,\rho_1) = \sup\left\{\int_\Omega \alpha\,\d\rho_0 + \int_\Omega \beta\,\d\rho_1\right| \left. \vphantom{\int_\Omega} \alpha,\beta\text{ Lipschitz with constant 1, } \right. \\
	\left. \vphantom{\int_{\Omega}} \alpha(x) + \beta(x) \leq 0 \,\forall\, x \in \Omega\right\}\,.
\end{multline}
Typically one function is eliminated by setting $\beta=-\alpha$. Note that this represents a (computationally feasible) convex optimization problem in the variables $\alpha$ and $\beta$ that only live in $\Omega$ and only satisfy local constraints
(for the latter one needs to work in a path metric space such as $\R^n$ since then the Lipschitz constraint just amounts to ensuring $|\nabla\alpha|,|\nabla\beta|\leq1$ almost everywhere).
Such a reduction of the high-dimensional problem domain $\Omega\times\Omega$ to just $\Omega$ is only possible for $W_1$ transport
and represents a substantial decrease in memory requirements (and typically also computation time due to the smaller number of variables and constraints).
For other transport problems, the strongest dimension reduction can be achieved by considering a dynamic formulation \cite{BenamouBrenier2000}
which in addition to $\Omega$ only requires a single further (time) dimension.
In the special case that the measure $\rho_0$ is Lebesgue-continuous, a Monge map $T:\Omega\to\Omega$ exists describing the transport
so that one may try to identify $T$ instead,
however, even though the problem dimension then is the same as for $W_1$, the formulation in $T$ is highly nonlinear and nonconvex preventing a global optimization.

A natural question to ask is to what extent \eqref{eqn:W1predual} is distinguished among optimal transport metrics or whether it is just one special instance in a larger family of transport discrepancies. Put another way, one can ask how to generalize the above problem while maintaining convexity and locality of the constraints, the two requirements for an efficient implementation.
We suggest the following simple extension:
\begin{multline}\label{eqn:predualExtension}
W_{\h,\g,B}(\rho_0,\rho_1)=\sup\left\{\int_\Omega \h(\alpha)\,\d\rho_0 + \int_\Omega \g(\beta)\,\d\rho_1\right| \alpha,\beta\text{ Lipschitz with constant 1,} \\
	\left.\vphantom{\int_\Omega} (\alpha(x),\beta(x))\in B \,\forall\, x \in \Omega\right\}\,,
\end{multline}
where $\h,\g:\R\to[-\infty,\infty]$ are concave functions and $B\subset\R^2$ is a convex set
(a few more natural conditions on $\h,\g,B$ will be stated in detail later).
A specific instance of this class turns out to be given by the Kantorovich--Rubinstein distance examined in \cite{LellmannKantorovichRubinstein2014}, also known as the flat metric in geometric measure theory,
\begin{equation*}
W_{\text{KR},\lambda}(\rho_0,\rho_1)
=\sup\left\{\int_\Omega\alpha\,\d(\rho_0-\rho_1)\,\middle|\,\alpha\text{ Lipschitz with constant 1, }|\alpha|\leq\lambda\right\}\,.
\end{equation*}
This metric can for example be recovered by the choice $B=\{(\alpha,\beta)\,|\,\alpha+\beta\leq0\}$ and $\h(\alpha)=\g(\alpha)=\alpha$ if $\alpha\geq-\lambda$ and $-\infty$ else
(where the variable $\beta$ may be eliminated via $\beta=-\alpha$).

An alternative way to generalize $W_1$ is motivated by the desire to extend it to unbalanced measures.
To this end one can introduce pointwise discrepancy functionals that locally penalize growth or shrinkage of mass and are of the form
\begin{equation*}
D(\rho_0,\rho_1)=\int_\Omega c(\RadNik{\rho_0}\mu,\RadNik{\rho_1}\mu)\,\d\mu\,,
\end{equation*}
where $\mu\gg\rho_0,\rho_1$ is an arbitrary measure, $\RadNik{\rho_i}{\mu}$ is the Radon--Nikodym derivative, and $c(m_0,m_1)\geq0$ is the cost of changing the mass of a particle from $m_0$ to $m_1$
(assumed one-homogeneous in $(m_0,m_1)$ so that $D$ is independent of the choice of $\mu$).
Using such a discrepancy $D$ one can extend $W_1$ to allow for unbalanced measures $\rho_0$ and $\rho_1$, for example via $\inf_{\rho',\rho''}W_1(\rho_0,\rho')+D(\rho',\rho'')+W_1(\rho'',\rho_1)$
or $\inf_{\rho',\rho''}D(\rho_0,\rho')+W_1(\rho',\rho'')+D(\rho'',\rho_1)$ or in the more general case via
\begin{equation}\label{eqn:infConvExtension}
W_{\Dl,\Dm,\Dr}(\rho_0,\rho_1)=\inf_{\rho_0',\rho_0'',\rho_1'',\rho_1'}\Dl(\rho_0,\rho_0')+W_1(\rho_0',\rho_0'')+\Dm(\rho_0'',\rho_1'')+W_1(\rho_1'',\rho_1')+\Dr(\rho_1',\rho_1)
\end{equation}
with potentially different discrepancy functionals $\Dl$, $\Dm$, and $\Dr$.

Such `inf-convolutions' can also be constructed for more general transport distances beyond $W_1$. The models used in \cite{RubnerEMD-IJCV2000,PeleECCV2008} can be interpreted as instances of this framework, where the discrepancies $D_i$ were chosen sufficiently simple, such that \eqref{eqn:infConvExtension} remains a linear program.
Indeed, assuming without loss of generality $\rho_1(\Omega)\geq\rho_0(\Omega)$, those models can (in the continuous setting) be formulated as
\begin{equation*}
W_{\text{EMD}}(\rho_0,\rho_1)
=\inf_{\substack{\pi\geq0,\,\pi(\Omega\times\Omega)\geq\rho_0(\Omega)\\\pi(B\times\Omega)\leq\rho_0(B),\,\pi(\Omega\times B)\leq\rho_1(B)\,\forall B\subset\Omega}}
\int_{\Omega\times\Omega}d(x,y)\,\d \pi(x,y)+\lambda\mathrm{diam}\Omega|\rho_1(\Omega)-\rho_0(\Omega)|
\end{equation*}
with $\alpha\geq0$ ($\alpha=0$ in \cite{RubnerEMD-IJCV2000}).
It turns out that for $\alpha\notin(0,1)$ this can be recovered by picking $\Dl=\Dm$ to be infinite unless their arguments are identical
and $\Dr(\rho_0,\rho_1)=\int_\Omega c(\RadNik{\rho_0}\mu,\RadNik{\rho_1}\mu)\,\d\mu$ with $c(m_0,m_1)=\alpha\mathrm{diam}|m_1-m_0|$ if $\alpha\geq1$ or, if $\alpha=0$, $c(m_0,m_1)=0$ for $m_1\geq m_0$ and $\infty$ else.
In \cite{LieroMielkeSavare-HellingerKantorovich-2015a} entropy functionals were considered as local discrepancies and alternative dynamic and lifted reformulations of \eqref{eqn:infConvExtension} were given for particular choices of discrepancies and transport costs.

In this article we exploit the particular structure of $W_1$ to show that the two approaches \eqref{eqn:predualExtension} and \eqref{eqn:infConvExtension} are actually equivalent in the sense that for any admissible $(\h,\g,B)$ one can find $(\Dl,\Dm,\Dr)$ and vice versa such that
\begin{equation*}
W_{\h,\g,B}=W_{\Dl,\Dm,\Dr}\,.
\end{equation*}
\subsection{Contributions}
The main contributions of this article are the following.
\begin{itemize}
\item
We introduce and analyse the simple, but very broad and flexible new class \eqref{eqn:predualExtension} of Wasserstein-1-type discrepancies, including existence of optimizers.
In a sense it represents the most general extension of the Kantorovich--Rubinstein formula for Wasserstein-1 transport that is still convex without nonlocal constraints.
\item
We furthermore introduce and analyse the alternative generalization \eqref{eqn:infConvExtension} of the Wasserstein-1 distance,
which on the one hand represents a more complex optimization problem in more variables,
but which on the other hand allows a more straightforward interpretation of the different involved terms.
This formulation represents a general framework for Wasserstein-1-type discrepancies since it covers various unbalanced transport models from the literature.
\item
We show equivalence of both above generalizations and detail how both formulations relate to each other.
\item
We propose a discretization and feasible numerical implementation for the simpler first formulation
and perform a number of experiments showcasing the versatility of the Wasserstein-1-type discrepancy framework by producing different effects with different discrepancy variants.
This should raise awareness of the fact that the way optimal transport is extended to unbalanced measures has to be carefully chosen depending on the application.
\end{itemize}

The article is organized as follows.
The mathematical setting and notation are fixed in Section~\ref{sec:notation}.
In Section~\ref{sec:W1Extensions} we introduce in detail the $W_1$ generalizations \eqref{eqn:predualExtension} and \eqref{eqn:infConvExtension}
and prove their equivalence via convex duality in Section~\ref{sec:modelEquivalence}.
Examples and special cases are described in Section~\ref{sec:Examples}. In particular, different existing approaches fit into our framework such as the Kantorovich--Rubinstein norm from \cite{LellmannKantorovichRubinstein2014}
and a $W_1$-type variant of the Hellinger--Kantorovich distance from \cite{LieroMielkeSavare-HellingerKantorovich-2015a}.
Finally, a numerical discretization and optimization algorithm are provided in Section~\ref{sec:discretization},
followed by a number of numerical experiments in Section~\ref{sec:experiments}.
We conclude in Section~\ref{sec:conclusion}.

\subsection{Setting and notation}
\label{sec:notation}
Throughout the article we will use the following conventions.
\begin{itemize}
\item We will work in a compact path metric space $(\Omega,d)$, that is, the metric $d(x,y)$ represents the infimal length of continuous paths $\gamma:[0,1]\to\Omega$ connecting $x$ and $y$.
For instance, $\Omega$ could be a closed bounded convex subset of a normed vector space such as $\R^n$ or a compact geodesically convex subset of a Riemannian manifold.
\item For a topological space $X$ we denote by $C(X)$ the set of continuous bounded real functions over $X$ equipped with the $\sup$-norm.
\item For a metric space $X$ we write $\Lip(X)\subset C(X)$ for the set of 1-Lipschitz functions on $X$.
\item For a compact space $X$ we identify the topological dual of $C(X)^m$ with the space of $\R^m$-valued Radon measures on $X$, denoted by $\meas(X)^m$, and we equip it with the weak-* topology arising from this pairing (see e.\,g.\ \cite[Def.\ 1.58]{AFP00}).
\item We write $\measp(X)$ for the nonnegative Radon measures on $X$.
\item The Dirac measure $\delta_x\in\measp(\Omega)$ at a point $x\in\Omega$ is defined for all measurable $A\subset\Omega$ via $\delta_x(A)=1$ if $x\in A$ and $\delta_x(A)=0$ else.
\item For two measures $\mu,\nu\in\meas(\Omega)$, $\mu\ll\nu$ expresses that $\mu$ is absolutely continuous with respect to $\nu$, and its Radon--Nikodym derivative will be denoted $\RadNik\mu\nu$.
\item For two measurable spaces $X$, $Y$, some $\mu \in \meas(X)$, and a measurable map $f : X \to Y$ we denote by $f_\sharp \mu \in \meas(Y)$ the push-forward of $\mu$ under $f$ defined by $f_\sharp \mu(A) = \mu(f^{-1}(A))$ for any measurable $A \subset Y$.
\item For $i=0,1$ we define the canonical projections $\Proj_{i}:\Omega\times\Omega\to\Omega$, $(x_0,x_1)\mapsto x_i$.
\item For two $\rho_0$, $\rho_1 \in \measp(\Omega)$ the set of couplings is given by
\begin{equation*}
\Gamma(\rho_0,\rho_1) = \left\{ \pi \in \measp(\Omega\times\Omega)\,\vphantom{{\Proj_0}_\sharp}\right|\left.\,{\Proj_0}_\sharp \pi = \rho_0 \text{ and } {\Proj_1}_\sharp \pi = \rho_1 \right\}\,.
\end{equation*}
Note that $\Gamma(\rho_0,\rho_1)$ is empty when $\rho_0(\Omega) \neq \rho_1(\Omega)$.
\item Let $X$ be a normed vector space with dual space $X'$. For a function $f:X\to\R\cup\{\infty\}$ the Legendre--Fenchel conjugate is given by
\begin{equation*}
f^\ast(x')=\sup_{x\in X}\langle x,x'\rangle-f(x)\,,
\end{equation*}
where $\langle\cdot,\cdot\rangle$ denotes the dual pairing.
The preconjugate of a function $f:X'\to\R\cup\{\infty\}$ will be denoted
\begin{equation*}
\vphantom{f}^\ast f(x)=\sup_{x'\in X'}\langle x,x'\rangle-f(x')\,.
\end{equation*}
\item For two functions $f,g:X\to\R\cup\{\infty\}$ on a vector space $X$, their infimal convolution is denoted $$f\,\square\,g:X\to\R\cup\{\infty\}\,,\quad(f\,\square\,g)(x)=\inf\{f(x-y)+g(y)\,|\,y\in X\}\,.$$
\item The indicator function $\iota_C:X\to\R\cup\{\infty\}$ of a set $C\subset X$ is defined as $\iota_C(x)=0$ if $x\in C$ and $\iota_C(x)=\infty$ else.
\end{itemize}

\begin{remark}[Choice of metric space]
The property $\alpha\in\Lip(\Omega)$ of a function $\alpha:\Omega\to\R$ is usually a global one;
to verify it one needs to assert $\alpha(x)-\alpha(y)\leq d(x,y)$ for all pairs $(x,y)\in\Omega\times\Omega$.
However, in path metric spaces $\Omega$, the 1-Lipschitz property reduces to a local condition on $\alpha$ that only needs to be checked for all $x\in\Omega$
(on $\R^n$, for instance, $\alpha\in\Lip(\R^n)$ is equivalent to $|\nabla\alpha(x)|\leq1$ for almost every $x\in\R^n$).
This reduction of a pointwise constraint on $\Omega\times\Omega$ to a pointwise constraint only on $\Omega$ underlies the computational efficiency of $W_1$,
and for this reason we here restrict to path metric spaces $\Omega$.

On the other hand, the choice of a compact space is rather technically motivated.
Indeed, for compact $\Omega$, a predual space to $\meas(\Omega)$ is simply $C(\Omega)$ and thus contains $\Lip(\Omega)$.
Therefore, formulations as in \eqref{eqn:W1predual} can be obtained via standard duality arguments.
On noncompact $\Omega$, by contrast, considering the space $C_0(\Omega)$ of continuous functions that vanish on the boundary and at infinity as the predual to $\meas(\Omega)$,
we have $\Lip(\Omega)\not\subset C_0(\Omega)$ so that additional approximation arguments are required (see e.\,g.\ the discussion in \cite[p.99]{Villani-OptimalTransport-09} or the more direct application of the Hahn--Banach Theorem \cite{Ed10}).
\end{remark}

%% file: 02-w1-extension.tex

\section{Static formulations of Wasserstein-1-type discrepancies}
\label{sec:W1Extensions}
In this section we consider two different extensions of the classical $W_1$ metric and subsequently demonstrate their equivalence.

\subsection{Wasserstein-1-Type discrepancies}
One may ask how the $W_1$ metric \eqref{eqn:W1predual} can be generalized without giving up its efficiency,
particularly the convexity of the problem and its low dimensional variables and constraints.
This motivates the following definition of a discrepancy between two nonnegative measures $\rho_0$ and $\rho_1$.

\begin{definition}[$W_{\h,\g,B}$-discrepancy]
  \label{def:GeneralizedPrimalProblem}
  Consider a convex set $B\subset{\R^2}$ and concave, upper semi-continuous functions $\h,\g:\R\to\R\cup\{-\infty\}$.
  For $\rho_0,\rho_1 \in \measp(\Omega)$ we define
  \begin{align}
  E_{\h,\g,B}^{\rho_0,\rho_1}(\alpha,\beta)&=\begin{cases}\int_\Omega \h(\alpha(x))\,\d\rho_0(x)\\ \quad+ \int_\Omega \g(\beta(x))\,\d\rho_1(x)&\text{if }\alpha, \beta \in \Lip(\Omega)\text{ with }(\alpha(x),\beta(x)) \in B \ \forall\, x \in\Omega\,,\\-\infty&\text{else,}\end{cases}\nonumber\\
  W_{\h,\g,B}(\rho_0,\rho_1)&=\sup_{\alpha,\beta\in C(\Omega)} E_{\h,\g,B}^{\rho_0,\rho_1}(\alpha,\beta)\,.\label{eq:GeneralizedPrimalProblem}
  \end{align}
\end{definition}

\begin{remark}[Inhomogeneous versions]
In principle one could also allow $\h$, $\g$, and $B$ to have spatially varying, inhomogeneous forms $\h,\g:\R\times\Omega\to\R\cup\{-\infty\}$ and $B:\Omega\to2^{\R\times\R}$.
To avoid technicalities we shall for now only consider the spatially homogeneous case.
The generalization to the inhomogeneous case will be discussed in Section~\ref{sec:Inhomogeneous}.
\end{remark}

\begin{remark}[Complexity]
Just like \eqref{eqn:W1predual}, definition \eqref{eq:GeneralizedPrimalProblem} represents a convex optimization problem
whose variables $\alpha$ and $\beta$ are functions on the low-dimensional domain $\Omega$
and satisfy three local constraints everywhere on $\Omega$, two Lipschitz constraints as well as $(\alpha(x),\beta(x))\in B$.
\end{remark}

\begin{remark}[Convexity]
As the pointwise supremum of linear functionals in $\rho_0$ and $\rho_1$, the discrepancy $W_{\h,\g,B}(\rho_0,\rho_1)$ is jointly convex in $\rho_0$ and $\rho_1$.
\end{remark}

\begin{remark}[Reduction to $W_1$ case]\label{rem:W1case}
The Wasserstein-1-distance is obviously retrieved for the choice
\begin{equation*}
\h(\alpha)=\alpha\,,\qquad
\g(\beta)=\beta\,,\qquad
B=\{(\alpha,\beta)\in\R^2\,|\,\alpha+\beta\leq0\}\,.
\end{equation*}
\end{remark}

\begin{definition}[Admissible $(\h,\g,B)$]\label{def:admissibility}
We call $(\h,\g,B)$ admissible if there exist functions $\hB,\gB:\R\to\R\cup\{-\infty\}$ such that
\begin{align}
\label{eqn:BIntersection}
B & = B_{01} \cap (B_0 \times B_1)\qquad\text{with} \\
\label{eqn:BHypograph}
B_{01} & =\left\{(\alpha,\beta)\in\R^2\,\right|\left.\alpha\leq\hB(-\beta)\right\}
	=\left\{(\alpha,\beta)\in\R^2\,\right|\left.\beta\leq\gB(-\alpha)\right\},\hspace*{-\linewidth}\\
\label{eqn:AlphaBetaMinA}
B_0 & = [\alphamin,+\infty), & \alphamin & = \inf \{ \alpha \in \R : \h(\alpha) > -\infty \}, \\
\label{eqn:AlphaBetaMinB}
B_1 & = [\betamin,+\infty), & \betamin & = \inf \{ \beta \in \R : \g(\beta) > -\infty \},
\end{align}
and $\h$, $\g$, $\hB$ or equivalently $\gB$ satisfy the following conditions (where we drop the indices):
\begin{enumerate}[1.]
\item\label{enm:convexity} $h$ is concave,
\item\label{enm:wellposedness} $h$ is upper semi-continuous,
\item\label{enm:positivity1} $h(s)\leq s$ for all $s\in\R$ and $h(0)=0$,
\item\label{enm:positivity} $h$ is differentiable at $0$ and $h'(0)=1$,
\item\label{enm:negativeMeasures} $h$ is monotonically increasing.
\end{enumerate}
Note that on their respective domains, $\hB=-\gB^{-1}(-\cdot)$ and $\gB=-\hB^{-1}(-\cdot)$.
\end{definition}

\begin{remark}[On the conditions]
The admissibility conditions are chosen as to make $W_{\h,\g,B}$ a reasonable discrepancy on $\measp(\Omega)$,
especially if two of $\h$, $\g$, and $B$ are taken as in Remark~\ref{rem:W1case}.
In particular, we ask for the following properties.
\begin{enumerate}[a.]
\item\label{enm:propUSC} $E_{\h,\g,B}^{\rho_0,\rho_1}$ should be upper semi-continuous (a natural requirement for well-posedness of optimization problem\,\eqref{eq:GeneralizedPrimalProblem}),
\item\label{enm:propNonNeg} $W_{\h,\g,B}(\rho_0,\rho_1)\geq0$ for all $\rho_0,\rho_1\in\measp(\Omega)$ and $W_{\h,\g,B}(\rho_0,\rho_1)=0$ if $\rho_0=\rho_1$,
\item\label{enm:propPos} $W_{\h,\g,B}(\rho_0,\rho_1)>0$ for $\rho_0\neq\rho_1$,
\item\label{enm:propNonNegMeas} $W_{\h,\g,B}(\rho_0,\rho_1)=\infty$ whenever $\rho_0$ or $\rho_1$ are negative,
\item\label{enm:propSWLSC} $W_{\h,\g,B}(\rho_0,\rho_1)$ should be sequentially weakly-* lower semi-continuous in $(\rho_0,\rho_1)$.
\end{enumerate}
Now to obtain corresponding conditions on $B_{01}$ we first consider the case $\h=\g=\id$ (then $B = B_{01}$).
Property~\ref{enm:propUSC} requires closedness of $B$, while property~\ref{enm:propNonNegMeas} implies $(-\infty,a]^2\subset B$ for some finite $a \in \R$.
Together with the convexity of $B$ it follows that $B_{01}$ can be expressed in the form \eqref{eqn:BHypograph}
for an upper semi-continuous, concave, monotonically increasing $\hB$.

Next, we set any two of $\h$, $\g$, and $\hB$ to the identity.
For the remaining one it is not difficult to see that
condition~\ref{enm:convexity} is equivalent to the convexity of optimization problem\,\eqref{eq:GeneralizedPrimalProblem}.
Likewise, condition~\ref{enm:wellposedness} is necessary for property~\ref{enm:propUSC} (it is also needed to make sense of the integrals in $E_{\h,\g,B}^{\rho_0,\rho_1}$)
and will in the proof of Proposition~\ref{prop:upperSemicontinuity} turn out to be also sufficient.
It is furthermore a simple exercise to show the equivalence between condition~\ref{enm:positivity1} and property~\ref{enm:propNonNeg}.
Indeed, assume $\g=\hB=\id$ (the other cases follow analogously),
then condition~\ref{enm:positivity1} implies $W_{\h,\g,B}(\rho_0,\rho_1)\geq E_{\h,\g,B}^{\rho_0,\rho_1}(0,0)=0$
as well as $W_{\h,\g,B}(\rho,\rho)=\sup_{\alpha\in C(\Omega)}\int_\Omega\h(\alpha)-\alpha\,\d\rho\leq0$ for $\rho\in\measp(\Omega)$.
Vice versa, taking $\rho=\delta_x$ for some $x\in\Omega$, $W_{\h,\g,B}(\rho,\rho)=0$ implies $\sup_{\alpha\in\R}\h(\alpha)-\alpha=0$ and thus in particular $\h\leq\id$.
Furthermore, for a contradiction assume $\h(0)<0$, then by virtue of the hyperplane separation theorem, the concavity and upper semi-continuity of $\h$ there exists $s\in\R$ and $\veps>0$ with $\h(\alpha)<s\alpha-\veps$ for all $\alpha\in\R$.
Due to $\h\leq\id$ we may assume $s\geq0$ and thus obtain $0\leq W_{\h,\g,B}(\rho,s\rho)=\sup_{\alpha\in\R}\h(\alpha)-s\alpha<0$.
Assuming now conditions~\ref{enm:convexity} to \ref{enm:positivity1} one can easily derive the equivalence of condition~\ref{enm:positivity} and property~\ref{enm:propPos}.
Indeed, taking again $\g=\hB=\id$, condition~\ref{enm:positivity} implies
that $E_{\h,\g,B}^{\rho_0,\rho_1}$ is differentiable at $(\alpha,\beta)=(0,0)$ in any direction $(\varphi,-\varphi)\in \Lip(\Omega)^2$
with $\partial_{(\alpha,\beta)}E_{\h,\g,B}^{\rho_0,\rho_1}(\alpha,\beta)(\varphi,-\varphi)=\int_\Omega\varphi\,\d(\rho_0-\rho_1)$.
Thus, $E_{\h,\g,B}^{\rho_0,\rho_1}(0,0)=0$ can only be a maximum if $\rho_0 = \rho_1$.
On the other hand, $-\h$ is convex with subgradient $\partial(-\h)(0)\supset\{-1\}$ due to condition~\ref{enm:positivity1}.
Now taking $\rho_0\in\measp(\Omega)$ and $\rho_1=s\rho_0$ for some $s\geq0$ with $s\neq1$,
property~\ref{enm:propPos} implies the existence of some $\alpha\in \Lip(\Omega)$
with $0<E_{\h,\g,B}^{\rho_0,\rho_1}(\alpha,-\alpha)\leq-\int_\Omega(\partial(-\h)(0)+s)\alpha\,\d\rho_0$.
Thus, $-s\notin\partial(-\h)(0)$ and therefore $\partial(-\h)(0)=\{-1\}$ from which condition~\ref{enm:positivity1} follows.
Note further that condition~\ref{enm:positivity1} automatically implies property~\ref{enm:propNonNegMeas} since $\h$ and $\g$ are unbounded from below,
while condition~\ref{enm:negativeMeasures} may simply be assumed for $\h$ and $\g$ without loss of generality:
Indeed, suppose for instance $\h$ to be nonmonotone,
then conditions~\ref{enm:convexity} and \ref{enm:positivity1} imply the existence of a unique maximum value $\h(\bar\alpha)\geq0$.
Therefore, $E_{\h,\g,B}^{\rho_0,\rho_1}(\alpha,\beta)\leq E_{\h,\g,B}^{\rho_0,\rho_1}(\min(\alpha,\bar\alpha),\beta)=E_{h,\g,B}^{\rho_0,\rho_1}(\alpha,\beta)$
and thus $W_{\h,\g,B}=W_{h,\g,B}$ for the monotonically increasing $h(\alpha)=\h(\min(\alpha,\bar\alpha))$.

Finally, the structure \eqref{eqn:BIntersection} of the set $B$ is necessary for property~\ref{enm:propSWLSC} (that construction~\eqref{eqn:BIntersection} actually implies property~\ref{enm:propSWLSC} will later follow from Corollary \ref{cor:equivalenceStatic}).
For instance, take $\g=\id$ and let $\rho_1=\delta_x$ and $\rho_0^n=\frac1n\rho_1\to0$ as $n\to\infty$.
If $\gB(-\bar\alpha)>\gB(-\alphamin)$ for some $\bar\alpha < \alphamin$, then
\begin{multline*}
W_{\h,\g,B_{01}}(0,\rho_1)
\geq E_{\h,\g,B_{01}}^{0,\rho_1}(\bar\alpha,\gB(-\bar\alpha))
=\gB(-\bar\alpha)\\
>\gB(-\alphamin)
\geq\liminf_{n\to\infty}\sup_{\alpha\geq\alphamin}\tfrac{\h(\alpha)}n+\gB(-\alpha)
=\liminf_{n\to\infty}W_{\h,\g,B_{01}}(\rho_0^n,\rho_1)\,.
\end{multline*}
\end{remark}

Without explicit mention we will in the following always assume $(\h,\g,B)$ to be admissible.
The class of $W_{\h,\g,B}$-discrepancies is natural to consider and allows to extend the classical $W_1$ distance to unbalanced measures as we will see.
Several previously introduced extensions of the $W_1$-distance (as well as $W_1$ itself) can be shown to fall into this category.
In Section \ref{sec:Examples} some examples, both already well-known and new variants, will be discussed in more detail.

\begin{remark}[Discrepancy bounds]
The conditions on $\h$, $\g$, and $B$ imply $W_{\h,\g,B}(\rho_0,\rho_1)\leq W_1(\rho_0,\rho_1)$ for all $\rho_0,\rho_1\in\measp(\Omega)$.
\end{remark}

\begin{remark}[Non-existence of optimizers]
Unfortunately, maximizers of $E_{\h,\g,B}^{\rho_0,\rho_1}$ do not exist in general.
For instance, consider the relevant special case of $\h(\alpha)=\frac{\alpha}{1+\alpha}$ for $\alpha>-1$ and $\h(\alpha)=-\infty$ else (see the Hellinger distance in Section~\ref{sec:Examples})
and set $\g=\hB=\id$ for simplicity.
For $\rho_0=\delta_x$ and $\rho_1=0$ it is easily seen that $W_{\h,\g,B}(\rho_0,\rho_1)=1$ but that $E_{\h,\g,B}^{\rho_0,\rho_1}(\alpha,\beta)<1$ for all $\alpha,\beta\in C(\Omega)$.
\end{remark}

Due to the potential non-existence of maximizers we will later also examine a dual problem formulation. However, non-existence happens rather for certain special cases.
As shown in Proposition~\ref{prop:existencePrimal}, those cases can be characterized by equations which are simple to check.

\begin{proposition}[Upper semi-continuity]
	\label{prop:upperSemicontinuity}
	Energy $E_{\h,\g,B}^{\rho_0,\rho_1}$ is upper semi-continuous on $C(\Omega)^2$.
\end{proposition}
\begin{proof}
Let $(\alpha_n,\beta_n)\to(\alpha,\beta)$ in $C(\Omega)^2$ with $E_{\h,\g,B}^{\rho_0,\rho_1}(\alpha_n,\beta_n)>-\infty$ for all $n$ sufficiently large (else there is nothing to show).
Due to the closedness of $\Lip(\Omega)$ and $B$ we have $(\alpha,\beta)\in\Lip(\Omega)^2$ as well as $(\alpha(x),\beta(x))\in B$ for all $x\in\Omega$.
Finally
\begin{multline*}
\limsup_{n\to\infty}\int_\Omega \h(\alpha_n)\,\d\rho_0
=\limsup_{n\to\infty}\int_\Omega \h(\alpha_n)-\alpha_n\,\d\rho_0+\int_\Omega\alpha_n\,\d\rho_0\\
\leq\int_\Omega\limsup_{n\to\infty}\h(\alpha_n)-\alpha_n\,\d\rho_0+\lim_{n\to\infty}\int_\Omega\alpha_n\,\d\rho_0
\leq\int_\Omega \h(\alpha)-\alpha\,\d\rho_0+\int_\Omega\alpha\,\d\rho_0
=\int_\Omega \h(\alpha)\,\d\rho_0\,,
\end{multline*}
where we have used Fatou's lemma (noting that $\h(\alpha_n)-\alpha_n\leq0$), the upper semi-continuity of $\h$, as well as the continuity of the dual pairing between $\measp(\Omega)$ and $C(\Omega)$.
Analogously, $\limsup_{n\to\infty}\int_\Omega \g(\beta_n)\,\d\rho_1\leq\int_\Omega \g(\beta)\,\d\rho_1$, concluding the proof.
\end{proof}

\begin{proposition}[Existence of optimizers]
	\label{prop:existencePrimal}
	Let $\rho_0,\rho_1\in\measp(\Omega)$.
	\begin{itemize}
	\item $W_{\h,\g,B}(\rho_0,\rho_1)=\infty$ if and only if $\sup_{(\alpha,\beta)\in B}\measnrm{\rho_0}\,\h(\alpha)+\measnrm{\rho_1}\,\g(\beta)=\infty$ (in which case there are no maximizers of $E_{\h,\g,B}^{\rho_0,\rho_1}$).
	\item A maximizer of $E_{\h,\g,B}^{\rho_0,\rho_1}$ exists
	if and only if there exist $\alpha^*,\beta^*\in(-\infty,0]$ with
	\begin{align*}
	\measnrm{\rho_0}\,(-\h\circ\hB)'(-\beta^*)&\geq\measnrm{\rho_1}\,(-\g)'(\beta^*)\,\\
	\text{or}\quad
	\measnrm{\rho_1}\,(-\g\circ\gB)'(-\alpha^*)&\geq\measnrm{\rho_0}\,(-\h)'(\alpha^*)\,,
	\end{align*}
	where the prime refers to a favourably chosen element of the subgradient.
	\item A maximizer of $E_{\h,\g,B}^{\rho_0,\rho_1}$ exists if and only if $\sup_{(\alpha,\beta)\in B}\h(\alpha)\measnrm{\rho_0}+\g(\beta)\measnrm{\rho_1}$ has a maximizer.
	\end{itemize}
\end{proposition}
\begin{proof}
For a function $f \in C(\Omega)$ we abbreviate $\hat f=\max_{x\in\Omega} f(x)$, $\check f=\min_{x\in\Omega} f(x)$.
Throughout the following, let $\alpha_n,\beta_n\in C(\Omega)$ denote a maximising sequence
and assume without loss of generality that $\hat\alpha_n\geq\hat\beta_n$ for $n$ large enough (else we may simply swap the roles of $\h,\alpha$ and $\g,\beta$).

As for the first statement, let us show that $W_{\h,\g,B}(\rho_0,\rho_1)=\infty$ implies the divergence of $\sup_{\alpha,\beta\text{ constant}}E_{\h,\g,B}^{\rho_0,\rho_1}(\alpha,\beta)$ (the converse implication is trivial).
Indeed, we must have $\h(\hat \alpha_n)\to\infty$ and thus $\alpha_n(x)\to\infty$ for all $x \in \Omega$ due to the Lipschitz constraint. Then for $n$ sufficiently large,
\begin{multline*}
\int_\Omega \h(\alpha_n(x))\,\d\rho_0(x)+\int_\Omega \g(\beta_n(x))\,\d\rho_1(x) \\
\leq \h(\check\alpha_n)\measnrm{\rho_0}+\g(\hat\beta_n)\measnrm{\rho_1}+(\h(\hat\alpha_n)-\h(\check\alpha_n))\measnrm{\rho_0}\,.
\end{multline*}
Using $(\check\alpha_n,\hat\beta_n)\in B$ and $\h(\hat\alpha_n)-\h(\check\alpha_n)\leq\hat\alpha_n-\check\alpha_n\leq\diam\Omega$ (note that $\h$ is a contraction on $[0,\infty)$) we indeed obtain
$E_{\h,\g,B}^{\rho_0,\rho_1}(\check\alpha_n,\hat\beta_n)\geq E_{\h,\g,B}^{\rho_0,\rho_1}(\alpha_n,\beta_n)-\diam\Omega\measnrm{\rho_0}\to\infty$.

As for the second statement, note first that $-\h \circ \hB$ is convex and that $-(-\h)'(\hB(\beta)) \cdot (-\hB)'(\beta) \in \partial (-\h \circ \hB)(\beta)$, where the prime indicates an arbitrary subgradient element.
Assume now the existence of a suitable $\beta^*$.
For a contradiction, suppose $\check\alpha_n\to\infty$ and $\hat\beta_n\to-\infty$ (both are equivalent due to $\alpha+\beta\leq0$ for all $(\alpha,\beta)\in B$).
Then for any $\Delta>0$ and $n$ large enough, $\hat\beta_n<\beta^*-\Delta$.
Now define $\tilde\beta_n(x)=\beta_n(x)+\Delta$ and $\tilde\alpha_n(x)=\hB(-\tilde\beta_n(x))$.
Note that $\tilde\alpha_n\in\Lip(\Omega)$ (since $\tilde\beta_n < \beta^* \leq 0$ and $\hB$ is a contraction on $[0,\infty)$) with $\tilde\alpha_n\geq\hB(-\beta_n)+\Delta(-\hB)'(-\tilde\beta_n)\geq\alpha_n+\Delta(-\hB)'(-\beta^*)$. We have
\begin{align*}
E_{\h,\g,B}^{\rho_0,\rho_1}(\alpha_n,\beta_n)
&=\int_\Omega \h(\alpha_n(x))\,\d\rho_0(x)+\int_\Omega \g(\beta_n(x))\,\d\rho_1(x)\\
&\leq\int_\Omega \h(\tilde\alpha_n)+(-\h)'(\tilde\alpha_n)\Delta(-\hB)'(-\beta^*)\,\d\rho_0
+\int_\Omega \g(\tilde\beta_n)+(-\g)'(\tilde\beta_n)\Delta\,\d\rho_1\\
&\leq\int_\Omega \h(\tilde\alpha_n)\,\d\rho_0+(-\h)'(\hB(-\beta^*))\Delta(-\hB)'(-\beta^*)\measnrm{\rho_0}\\
&\qquad+\int_\Omega \g(\tilde\beta_n)\,\d\rho_1+(-\g)'(\beta^*)\Delta\measnrm{\rho_1}\\
&=\int_\Omega \h(\tilde\alpha_n)\,\d\rho_0+\int_\Omega \g(\tilde\beta_n)\,\d\rho_1 \\
& \qquad -\Delta\left[(-\h\circ\gB)'(-\beta^*)\measnrm{\rho_0}-(-\g)'(\beta^*)\measnrm{\rho_1}\right]\\
&\leq E_{\h,\g,B}^{\rho_0,\rho_1}(\tilde\alpha_n,\tilde\beta_n)\,.
\end{align*}
Thus, $(\tilde\alpha_n,\tilde\beta_n)$ is an even better maximising sequence so that
we may assume the maximising sequence $\alpha_n,\beta_n\in C(\Omega)$ to be uniformly bounded with $\beta_n\geq\hat\beta_n-\diam\Omega\geq\beta^*-\diam\Omega$ and $\alpha_n\leq-\beta_n$.
Since $\alpha_n,\beta_n\in\Lip(\Omega)$, the sequence is equicontinuous and converges (up to a subsequence) against some $(\alpha,\beta)\in C(\Omega)^2$.
By the upper semi-continuity of the energy, this must be a maximizer.
The argument for a suitable $\alpha^*$ is analogous.

For the converse implication assume $\measnrm{\rho_0}\,(-\h\circ\hB)'(-\beta^*)<\measnrm{\rho_1}\,(-\g)'(\beta^*)$ for all $\beta^*\in(-\infty,0]$
(the proof is analogous if the other condition is violated).
Taking $\beta^*=0$, this implies $\measnrm{\rho_0}>\measnrm{\rho_1}$.
Let $(\alpha,\beta)\in C(\Omega)^2$ be a maximizer and choose $\Delta>\max\{\hat\beta,0\}$.
We now set $\tilde\beta=\beta-\Delta$, $\tilde\alpha(x)=\hB(-\tilde\beta(x))$ (note that again $\tilde\alpha\in\Lip(\Omega)$) and obtain as before
\begin{multline*}
E_{\h,\g,B}^{\rho_0,\rho_1}(\alpha,\beta)
=\int_\Omega \h(\alpha)\,\d\rho_0+\int_\Omega \g(\beta)\,\d\rho_1\\
\leq\int_\Omega \h(\tilde\alpha)\,\d\rho_0-(-\h)'(\hB(-\check{\tilde\beta}))\Delta(-\hB)'(-\check{\tilde\beta})\measnrm{\rho_0}
+\int_\Omega \g(\tilde\beta)\,\d\rho_1-(-\g)'(\check{\tilde\beta})\Delta\measnrm{\rho_1}\\
=\int_\Omega \h(\tilde\alpha)\,\d\rho_0+\int_\Omega \g(\tilde\beta)\,\d\rho_1
+\Delta\left[(-\h\circ\gB)'(-\check{\tilde\beta})\measnrm{\rho_0}-(-\g)'(\check{\tilde\beta})\measnrm{\rho_1}\right]
<E_{\h,\g,B}^{\rho_0,\rho_1}(\tilde\alpha,\tilde\beta)\,,
\end{multline*}
contradicting the optimality of $(\alpha,\beta)$.

By repeating the arguments for the second statement under restriction to spatially constant $(\alpha,\beta) \in \Lip(\Omega)^2$, we find that the existence conditions of the second statement are equivalent to existence of maximizers for $\sup_{(\alpha,\beta)\in B}\h(\alpha)\measnrm{\rho_0}+\g(\beta)\measnrm{\rho_1}$.
\end{proof}

There may be some redundancy in the choice of $\h$, $\g$, and $B$.
In detail, it turns out that in particular cases the model can be simplified by eliminating the constraint set $B$ and the variable $\beta$.
Later, this will also allow to simplify some infimal convolution-type discrepancy measures (cf.~Corollary \ref{cor:reductionInfimalConv} and Section \ref{sec:unbalancedExamples}).

\begin{proposition}[Model reduction]
	\label{prop:modelReduction}
	Let $\gamma>0$ (possibly $\gamma = +\infty$), $\rho_0,\rho_1\in\measp(\Omega)$, and abbreviate
	\begin{equation*}
	B(a,b)=\{(\alpha,\beta)\in\R^2\,|\,\alpha+\beta\leq0\}\cap
		\left( [a,+\infty)\times [b,+\infty) \right) \,.
	\end{equation*}
	\begin{itemize}
	\item If $\gB(\alpha)=\min \{ \alpha,\gamma \}$ for all $\alpha>0$ (or equivalently $\hB(\beta)=\beta-\iota_{[-\gamma,\infty)}(\beta)$ for $\beta<0$), then $W_{\h,\g,B}=W_{\h\circ\hB,\g,B(\tildealphamin,\betamin)}$ with
	\begin{equation*}
		\tildealphamin = -\gB(-\alphamin) = \max\{\alphamin,-\gamma\}\,.
	\end{equation*}
	\item If $\hB(\beta)=\min \{ \beta,\gamma \}$ for all $\beta>0$ (or equivalently $\gB(\alpha)=\alpha-\iota_{[-\gamma,\infty)}(\alpha)$ for $\alpha<0$), then $W_{\h,\g,B}=W_{\h,\g\circ\gB,B(\alphamin,\tildebetamin)}$ with
	\begin{equation*}
		\tildebetamin = -\hB(-\betamin) = \max\{\betamin,-\gamma\}\,.
	\end{equation*}
	\item $W_{\h,\g,B(a,b)}=\sup\left\{\int_\Omega\h(\alpha)\,\d\rho_0+\int_\Omega\g(-\alpha)\,\d\rho_1\right|\left.\vphantom{\int_\Omega}\alpha\in\Lip(\Omega),\,a\leq\alpha\leq -b\right\}$.
	\end{itemize}
\end{proposition}
\begin{proof}
In the first case, notice that for any $\beta\in\Lip(\Omega)$ with $\beta \leq \gamma$ we also have $\tilde\alpha=\hB\circ(-\beta)\in\Lip(\Omega)$, since $\hB(-\cdot)$ is a contraction on $(-\infty,\gamma]$.
Note that if $\beta(x)> \gamma$ for some $x \in \Omega$, both energies are $-\infty$ for any $\alpha$.
Moreover, for $\tilde\alpha$ to be feasible, we need $\tilde\alpha(x) = \hB(-\beta(x)) \geq \alphamin$ for all $x \in \Omega$ (see \eqref{eqn:AlphaBetaMinA}-\eqref{eqn:AlphaBetaMinB}), which is equivalent to $-\beta(x) \geq -\gB(-\alphamin) = \tildealphamin$ (see \eqref{eqn:BHypograph}).
Thus,
\begin{multline*}
\sup_{\alpha\in C(\Omega)}E_{\h,\g,B}^{\rho_0,\rho_1}(\alpha,\beta)
=E_{\h,\g,B}^{\rho_0,\rho_1}(\tilde\alpha,\beta)\\
=E_{\h\circ\hB,\g,B(\tildealphamin,\betamin)}^{\rho_0,\rho_1}(-\beta,\beta)
=\sup_{\alpha\in C(\Omega)}E_{\h\circ\hB,\g,B(\tildealphamin,\betamin)}^{\rho_0,\rho_1}(\alpha,\beta)
\end{multline*}
from which the statement follows.
The second case follows analogously.
Finally,
\begin{align*}
W_{\h,\g,B(a,b)}
	& =\sup\left\{\int_\Omega\h(\alpha)\,\d\rho_0+\int_\Omega\g(\beta)\,\d\rho_1\right|\left.\vphantom{\int_\Omega}\alpha,\beta\in\Lip(\Omega),\,\alpha+\beta\leq0,\,a \leq \alpha,\,b \leq \beta\right\} \\
	& =\sup\left\{\int_\Omega\h(\alpha)\,\d\rho_0+\int_\Omega\g(-\alpha)\,\d\rho_1\right|\left.\vphantom{\int_\Omega}\alpha\in\Lip(\Omega),\,a\leq\alpha\leq -b\right\}\,. \qedhere
\end{align*}
\end{proof}

\begin{remark}[Wasserstein-1 metric]\label{rem:W1reduction}
	For standard $W_1$, where $\h = \g = \hB = \id$, one has $\alphamin = \betamin = -\infty$ and $B=B(\alphamin,\betamin) = \R^2$. Consequently, by virtue of Proposition \ref{prop:modelReduction}, one can eliminate one dual variable by setting $\alpha=-\beta$, as is common practice (cf.~Section \ref{sec:IntroOverview}).
\end{remark}

\subsection[Infimal convolution-type extensions of W1]{Infimal convolution-type extensions of $W_1$}\label{sec:infConvW}
In the literature typically a different approach is taken to achieve convex and efficient generalizations of the $W_1$ metric,
namely an infimal convolution-type combination of non-transport-type metrics with the Wasserstein metric.
To introduce a general class of such discrepancies we now fix a suitable family of local, non-transport-type discrepancies.
\begin{definition}[Local discrepancy]
	\label{def:LocalSimilarityMeasure}
	Let $c : \R \times \R \mapsto [0,\infty]$ satisfy the following assumptions,
	\begin{enumerate}
  \renewcommand{\theenumi}{\alph{enumi}}
	\item\label{enm:cConv} $c$ is convex, positively 1-homogeneous, and lower semi-continuous jointly in both arguments,
	\item\label{enm:strictPos} $c(m,m)=0$ for all $m\geq0$ and $c(m_0,m_1)>0$ if $m_0\neq m_1$,
  \item\label{enm:domain} $c(m_0,m_1) = \infty$ whenever $m_0<0$ or $m_1 < 0$.
	\end{enumerate}
	Then $c$ induces a discrepancy $D$ on $\measp(\Omega)$ (extended to the rest of $\meas(\Omega)$ by infinity) via
	\begin{align}
		D & : \meas(\Omega)^2 \rightarrow [0,\infty]\,, &
			(\rho_0,\rho_1) & \mapsto \int_\Omega c(\RadNik{\rho_0}{\rho},\RadNik{\rho_1}{\rho})\,\d\rho\,,
	\end{align}
	where $\rho$ is any measure in $\measp(\Omega)$ with $\rho_0$, $\rho_1 \ll \rho$ (for instance $\rho=|\rho_0|+|\rho_1|$ with $|\cdot|$ indicating the total variation measure).
	Note that due to the 1-homogeneity of $c$ this definition does not depend on the choice of the reference measure $\rho$ (see e.\,g.\ \cite[Thm.\,2]{GoffmanSerrin64}),
	for which reason we shall also use the shorter notation
	\begin{equation*}
	D(\rho_0,\rho_1)=\int_\Omega c(\rho_0,\rho_1)\,.
	\end{equation*}
\end{definition}

Several examples of local discrepancies will be provided in Section~\ref{sec:Examples}, including the squared Hellinger distance or the metric induced by the total variation, 
\begin{equation*}
D^{FR}(\rho_0,\rho_1)=\int_\Omega\left(\sqrt{\tfrac{\d\rho_0}{\d\rho}}-\sqrt{\tfrac{\d\rho_1}{\d\rho}}\right)^2\,\d\rho
\quad\text{and}\quad
D^{TV}(\rho_0,\rho_1)=\|\rho_0-\rho_1\|_\meas\,.
\end{equation*}
Below we discuss a few basic properties.

\begin{remark}[Metric properties]
It is straightforward to check that a local discrepancy with integrand $c$ is a metric on $\measp(\Omega)$ if $c$ is a metric on $[0,\infty)$.
\end{remark}

\begin{remark}[Seminorm properties]
Obviously, a local discrepancy $D$ is positively homogeneous and convex on $\meas(\Omega)^2$.
\end{remark}

\begin{proposition}[Lower semi-continuity]
	\label{prop:LSCLocDis}
	A local discrepancy $D$ is weakly-* (and thus sequentially weakly-*) lower semi-continuous on $\meas(\Omega)^2$.
\end{proposition}

Our proof follows \cite{BoBu90}, who have shown sequential lower semi-continuity for more general functionals on $\meas(\Omega)^2$.
In case of a finite integrand $c$, sequential weak-* lower semi-continuity also directly follows from \cite[Thm.\,3]{GoffmanSerrin64}.

\begin{proof}
Let us define the sets
\begin{align*}
B&=\{u\in\R^2\,|\,u\cdot m\leq c(m_0,m_1)\text{ for all }m = (m_0,m_1) \in\R^2\}\\
\text{and}\quad
H&=\{u\in C(\Omega,\R^2)\,|\,u(x)\in B\text{ for all }x\in\Omega\}\,.
\end{align*}
Note that the set $B$ is also characterized by $c^\ast = \iota_B$.
Now consider a net $\rho^a=(\rho_0^a,\rho_1^a)$ which converges weakly-* to $\rho=(\rho_0,\rho_1)$ in $\meas(\Omega)^2$.
For an arbitrary $u\in H$ we have
\begin{equation*}
D(\rho_0^a,\rho_1^a)
=\int_\Omega c\left(\RadNik{\rho_0^a}{|\rho^a|},\RadNik{\rho_1^a}{|\rho^a|}\right)\,\d|\rho^a|
\geq\int_\Omega u(x)\cdot\d\rho^a(x)
\to\int_\Omega u(x)\cdot\d\rho\,,
\end{equation*}
where $|\cdot|$ indicates the total variation measure.
Introducing the indicator function $\iota_H:L^\infty(\Omega,|\rho|)^2 \allowbreak \to \allowbreak \{0,\infty\}$, $\iota_H(u)=0$ if $u\in H$ and $\iota_H(u)=\infty$ else,
we now have
\begin{equation*}
\sup_{u\in H}\int_\Omega u(x)\cdot\d\rho
=\preconj{\iota_H}(\RadNik{\rho}{|\rho|})
=\int_\Omega c\left(\RadNik{\rho_0}{|\rho|},\RadNik{\rho_1}{|\rho|}\right)\,\d|\rho|
=D(\rho_0,\rho_1)\,.
\end{equation*}
Indeed, the first and last equality hold by definition of the preconjugate and $D$, while the middle one is due to \cite[Thm.\,2]{BouchitteValadier88} (taking their $J\equiv0$, in which case their $k(x,(m_0,m_1))=c(m_0,m_1)$),
where the conjugation is with respect to the dual pair $(L^1(\Omega,|\rho|)^2,L^\infty(\Omega,|\rho|)^2)$.
\end{proof}

\begin{remark}[Coercivity]\label{rem:coercivity}
The convexity and positive homogeneity of $c$ imply
\begin{equation*}
D(\rho_0,\rho_1)\geq c(\measnrm{\rho_0},\measnrm{\rho_1})
\end{equation*}
via Jensen's inequality.
Furthermore, for any $\varepsilon\in(0,1)$ we have
\begin{equation*}
D(\tilde\rho,\rho)\geq c(\varepsilon,1)\measnrm{\rho}
\quad\text{and}\quad
D(\rho,\tilde\rho)\geq c(1,\varepsilon)\measnrm{\rho}
\qquad\text{for all }\rho,\tilde\rho\text{ with }\measnrm{\rho} \geq \measnrm{\tilde\rho}/\varepsilon
\end{equation*}
due to $D(\tilde\rho,\rho)\geq\measnrm{\rho} c(\measnrm{\tilde\rho}/\measnrm{\rho},1)$
and $D(\rho,\tilde\rho)\geq\measnrm{\rho} c(1,\measnrm{\tilde\rho}/\measnrm{\rho})$.
\end{remark}

\begin{definition}[Infimal convolution of discrepancies]
Let $X$ be a space and $d_1,d_2:X\times X\to[0,\infty]$ with $d_i(x_1,x_2)=0$, $i=1,2$, if and only if $x_1=x_2$. We define a new function $d_1\diamond d_2:X\times X\to[0,\infty]$ by
\begin{equation*}
(d_1\diamond d_2)(x_1,x_2)=\inf_{x\in X}d_1(x_1,x)+d_2(x,x_2)\,.
\end{equation*}
\end{definition}

\begin{remark}[Properties of $d_1\diamond d_2$]\label{rem:propertiesInfConv}\ 
\begin{enumerate}
\item The operation $\diamond$ is associative, that is, $d_1\diamond(d_2\diamond d_3)=(d_1\diamond d_2)\diamond d_3$, but not commutative.
\item For $X$ a vector space, the operation $\diamond$ can be viewed as an infimal convolution $\square$ via
\begin{equation*}
	(d_1\diamond d_2)(x_1,x_2)
	=(d_1(x_1,\cdot) \, \square \, d_2(y-\cdot,x_2))(y)
	=(d_1(x_1,y-\cdot) \, \square \, d_2(\cdot,x_2))(y)
\end{equation*}
for an arbitrary $y\in X$.
\item We have $(d_1\diamond d_2)(x_1,x_2)\leq d_i(x_1,x_2)$ for $i=1,2$.
\item The nonnegativity of $d_1$ and $d_2$ is inherited by $d_1\diamond d_2$.
Furthermore, $(d_1\diamond d_2)(x_1,x_2)=0$ if and only if $x_1=x_2$.
\item If $X$ is a vector space and $d_1$ and $d_2$ are jointly convex and positively 1-homogeneous in both arguments, then $d_1\diamond d_2$ is so as well, as is straightforward to check.
If in addition $X$ has a topology and $d_1$ and $d_2$ are lower semi-continuous, then so is $d_1\diamond d_2$.
\item\label{enm:metricInfConv} If $d_1=d_2=d$ for a metric $d$, then $d_1\diamond d_2=d$.
\end{enumerate}
\end{remark}

We are now in a position to introduce a new class of generalizations to the Wasserstein-1 metric. 
The idea is to just combine $W_1$ with local discrepancies via the above infimal convolution-type approach.
There are multiple possibilities to do so, for instance one might append the local discrepancy before or after the Wasserstein metric,
or one might want to allow mass change in between two mere mass transports.
The below definition is held general enough such that it encompasses all these possibilities.

\begin{definition}[$W_{\Dl,\Dm,\Dr}$-discrepancy]\label{def:SandwichProblem}
	For three local discrepancies $\Dl$, $\Dm$, $\Dr$ we define the following discrepancy on $\measp(\Omega)$,
	\begin{multline}
		\label{eq:SandwichProblem}
		W_{\Dl,\Dm,\Dr}(\rho_0,\rho_1)
		=(\Dl\diamond W_1\diamond\Dm\diamond W_1\diamond \Dr)(\rho_0,\rho_1)\\
		=\inf \left\{
			\Dl(\rho_0,\rho_0') + W_1(\rho_0',\rho_0'') + \Dm(\rho_0'',\rho_1'')
				+ W_1(\rho_1'',\rho_1') + \Dr(\rho_1',\rho_1)\, \right| \\
		\left.
			(\rho_0',\rho_0'',\rho_1'',\rho_1') \in \measp(\Omega)^4
			\right\}\,.
	\end{multline}
\end{definition}
\begin{remark}[Unbalanced measures] \label{rem:UnbalancedMeasures}
The above-defined discrepancy can be used to extend $W_1$ to unbalanced measures.
Indeed, the mass can change in three places, before the first $W_1$ transport (penalized by $\Dl$), after the second $W_1$ transport (in $\Dr$), and in between (in $\Dm$).
This does not only allow to accommodate mass differences in $\rho_0$ and $\rho_1$, but will also have the effect that mass may be decreased a little before transport and then re-increased again afterwards.
By choosing the extended discrete metric
\begin{equation}\label{eqn:discreteMetric}
D^d(\rho_0,\rho_1)=\begin{cases}0&\text{if }\rho_0=\rho_1\in\measp(\Omega)\,,\\\infty&\text{else}\end{cases}
\end{equation}
for some of the $\Dl,\Dm,\Dr$, mass change can be prohibited and thus the level of `local flexibility' in \eqref{eq:SandwichProblem} be varied.
Note that the model can then be simplified due to
\begin{equation*}
W\diamond D^d=D^d\diamond W=W
\end{equation*}
for any (local or nonlocal) discrepancy $W$ on $\measp(\Omega)$.
Furthermore, 
due to Remark \ref{rem:propertiesInfConv}\eqref{enm:metricInfConv}
a further model reduction results from
\begin{equation*}
W_1\diamond W_1=W_1\,.
\end{equation*}
\end{remark}

The existence of minimizers $\rho_0',\rho_1',\rho_0'',\rho_1''$ will follow automatically from the proof of Proposition~\ref{prop:PrimalDualSandwich} below,
but can also easily be proven directly, only using the sequential weak-* lower semi-continuity and coercivity of the discrepancies.

\begin{proposition}[Existence]
Given $\rho_0,\rho_1\in\measp(\Omega)$ such that $W_{\Dl,\Dm,\Dr}(\rho_0,\rho_1)$ is finite, problem\,\eqref{eq:SandwichProblem} admits minimizers $\rho_0',\rho_1',\rho_0'',\rho_1''$.
\end{proposition}
\begin{proof}
Consider a minimising sequence $(\rho_{0,n}',\rho_{1,n}',\rho_{0,n}'',\rho_{1,n}'')$, $n=1,2,\ldots$.
By the coercivity of $\Dl$ and $\Dr$ we may assume $(\measnrm{\rho_{0,n}'},\measnrm{\rho_{1,n}'})$ and thus also $(\measnrm{\rho_{0,n}''},\measnrm{\rho_{1,n}''})$ to be uniformly bounded.
Therefore, a subsequence of $(\rho_{0,n}',\rho_{1,n}',\rho_{0,n}'',\rho_{1,n}'')$ converges weakly-* against some $(\rho_0',\rho_1',\rho_0'',\rho_1'')\in\measp(\Omega)^4$,
which must be the minimizer due to the sequential weak-* lower semi-continuity of $\Dl$, $\Dr$, $\Dm$, and $W_1$.
\end{proof}

\begin{proposition}[Sequential lower semi-continuity]
	\label{prop:seqLSCW}
	The discrepancy $W_{\Dl,\Dm,\Dr}$ is sequentially weakly-* lower semi-continuous on $\meas(\Omega)^2$.
\end{proposition}
\begin{proof}
Indeed, let $\rho_i^n\to_{n\to\infty}\rho_i$ weakly-*, $i=0,1$, and assume (potentially by restricting to a subsequence) $W_{\Dl,\Dm,\Dr}(\rho_0^n,\rho_1^n) < C$ for some $C>0$ and all $n$ (otherwise there is nothing to show).
Furthermore, let $(\rho_{0,n}',\rho_{1,n}',\rho_{0,n}'',\rho_{1,n}'')$ be the corresponding minimizers in \eqref{eq:SandwichProblem}.
The coercivity of $\Dl$, $\Dr$, and $W_1$ now implies the uniform boundedness of $(\rho_{0,n}',\rho_{1,n}',\rho_{0,n}'',\rho_{1,n}'')$ in $\measp(\Omega)^4$.
Therefore we have weak-* convergence against some $(\rho_0',\rho_1',\rho_0'',\rho_1'')$ up to taking a subsequence and thus
\begin{align*}
W_{\Dl,\Dm,\Dr}(\rho_0,\rho_1)
&\leq \Dl(\rho_0,\rho_0') + W_1(\rho_0',\rho_0'') + \Dm(\rho_0'',\rho_1'') + W_1(\rho_1'',\rho_1') + \Dr(\rho_1',\rho_1)\\
&\leq\liminf_{n\to\infty} \Dl(\rho_{0}^n,\rho_{0,n}') + W_{1}(\rho_{0,n}',\rho_{0,n}'') + \Dm(\rho_{0,n}'',\rho_{1,n}'') \\
& \qquad \qquad + W_{1}(\rho_{1,n}'',\rho_{1,n}') + \Dr(\rho_{1,n}',\rho_{1}^n)\\
&=\liminf_{n\to\infty}W_{\Dl,\Dm,\Dr}(\rho_0^n,\rho_1^n)
\end{align*}
due to the weak-* lower semi-continuity of $\Dl$, $\Dr$, $\Dm$, and $W_1$.
\end{proof}

\begin{remark}[Semimetric properties]
If $\Dl(\mu,\nu)=\Dr(\nu,\mu)$ for all $\mu,\nu\in\measp(\Omega)$ and $\Dm$ is symmetric, then also $W_{\Dl,\Dm,\Dr}$ will be symmetric.
Thus, $W_{\Dl,\Dm,\Dr}$ is a semimetric in that it satisfies all metric axioms except for possibly the triangle inequality.
\end{remark}

\subsection{Model equivalence}\label{sec:modelEquivalence}
Here we prove that the two previously introduced model families $W_{\h,\g,B}$ and $W_{\Dl,\Dm,\Dr}$ are actually equivalent.
As a byproduct we arrive at a more intuitive interpretation of the quantities from definition\,\eqref{def:GeneralizedPrimalProblem},
which so far was just derived as the most general extension of the predual $W_1$ formulation.

\begin{proposition}[Primal and predual formulation]
	\label{prop:PrimalDualSandwich}
	Let the local discrepancies $\Dl$, $\Dm$, and $\Dr$ be induced by the integrands $\Cl$, $\Cm$, and $\Cr$,
	and let $\rho_0,\rho_1\in\measp(\Omega)$ with finite mass.
	The following primal and predual formulations hold.
	\begin{align}
	W_{\Dl,\Dm,\Dr}(\rho_0,\rho_1)
	&=\inf_{\pi_1,\pi_2\in\measp(\Omega^2)}
	\int_\Omega \Cl(\rho_0,{\Proj_1}_\sharp\pi_1)
	+\int_{\Omega\times\Omega}d(x,y)\,\d\pi_1(x,y)\nonumber\\
	&\qquad+\int_\Omega \Cm({\Proj_2}_\sharp\pi_1,{\Proj_1}_\sharp\pi_2)
	+\int_{\Omega\times\Omega}d(x,y)\,\d\pi_2(x,y)
	+\int_\Omega \Cr({\Proj_2}_\sharp\pi_2,\rho_1)\,,\label{eqn:primal}\\
	W_{\Dl,\Dm,\Dr}(\rho_0,\rho_1)
	&=\sup_{\substack{\alpha,\beta\in\Lip(\Omega)\\(\alpha(x),\beta(x))\in B_{01}\cap(B_0\times B_1)\,\forall x\in\Omega}}
	\int_\Omega \h(\alpha)\,\d\rho_0+\int_\Omega \g(\beta)\,\d\rho_1\,,\label{eqn:predual}
	\end{align}
	where
	\begin{gather}
	\label{eqn:predualInducedH}
	\hB(\beta)=-[\Cm(1,\cdot)]^\ast(-\beta)\,,\qquad
	\h(\alpha)=-[\Cl(1,\cdot)]^\ast(-\alpha)\,,\qquad
	\g(\beta)=-[\Cr(\cdot,1)]^\ast(-\beta)\,,\\
	B_{01}=\{(\alpha,\beta)\in\R^2\,|\,\alpha\leq \hB(-\beta)\}\,, \nonumber \\
	B_0=\cl\{\alpha\in\R\,|\,\h(\alpha)>-\infty\}\,,\qquad
	B_1=\cl\{\beta\in\R\,|\,\g(\beta)>-\infty\}\,, \nonumber
	\end{gather}
	where $\cl$ denotes the closure.
	If it is finite, the infimum in the primal formulation is achieved.
\end{proposition}

\begin{remark}[Relation to $W_{\h,\g,B}$]
	Problem \eqref{eqn:predual} looks very similar to a $W_{\h,\g,B}$-type problem as specified in  \eqref{eq:GeneralizedPrimalProblem}, where partial conjugates of $(\Cl,\Cr,\Cm)$ take the roles of $(\h,\g,\hB)$. We will make this correspondence more precise at the end of this section.
	Note that the set $B_{01}$ can also be characterized by $\Cm^\ast = \iota_{B_{01}}$.
	However, the characterization via the function $\hB$ emphasises the symmetry of the roles of $\h$, $\g$ and $B$, as counterparts of $\Cl$, $\Cr$ and $\Cm$.
\end{remark}

\begin{remark}[Relation to entropy-transport]\label{rem:entropyTransport}
With $\Dm$ the extended discrete metric \eqref{eqn:discreteMetric} (in which case one of $\pi_1$ and $\pi_2$ can be eliminated) and the metric $d$ replaced by a more general cost $c$,
the above primal problem \eqref{eqn:primal} has for instance also been considered by Liero et al.\,\cite[(1.6)]{LieroMielkeSavare-HellingerKantorovich-2015a}.
\end{remark}

The proof of Proposition \ref{prop:PrimalDualSandwich} requires a few preparatory lemmas.
The first one is the analogue of the well-known Fenchel--Moreau theorem, only stated for functionals on the dual space.

\begin{lemma}
	\label{lem:FenchelMoreau}
	Let $X$ be a Banach space with topological dual $X^*$ and $w:X^*\to(-\infty,\infty]$ be proper convex and weakly-* lower semi-continuous. Then $w=(\preconj w)^*$.
\end{lemma}
\begin{proof}
Let $Y=X^*$ equipped with the weak-* topology, then its dual space, the space of continuous linear functionals on $Y$, is $Y^*=X$.
Now define $v:Y\to(-\infty,\infty]$ by $v(y)=w(y)$ for all $y\in Y$, then $v$ is proper convex lower semi-continuous on $Y$.
By the Fenchel--Moreau theorem, $v=\preconj(v^*)$, however, $v^*=\preconj w$ and $\preconj(v^*)=(\preconj w)^*$ by definition of the Legendre--Fenchel conjugate.
\end{proof}

Next, recall that any discrepancy $D$ acts on all of $\meas(\Omega)^2$ with $D(\rho_0,\rho_1)=\infty$ as soon as $\rho_0\notin\measp(\Omega)$ or $\rho_1\notin\measp(\Omega)$. 

\begin{lemma}
	\label{lem:conjugateD}
	Let $D(\rho_0,\rho_1)=\int_\Omega c(\rho_0,\rho_1)$ be a local discrepancy on $\measp(\Omega)^2$ and $\rho\in\measp(\Omega)$. We have
	\begin{equation*}
	\preconj[D(\rho,\cdot)](\alpha)=\int_\Omega h(\alpha)\,\d\rho+\iota_{\cl[h(\alpha)<\infty]}(\alpha)
	\qquad\text{and}\qquad
	\preconj D(\alpha,\beta)=\iota_{[\alpha\leq-h(\beta)]}(\alpha,\beta)
	\end{equation*}
	for all $\alpha,\beta\in C(\Omega)$ and $h(\alpha)=[c(1,\cdot)]^\ast(\alpha)$,
	where $\cl$ denotes the closure and $[h(\alpha)<\infty]\subset C(\Omega)$ and $[\beta\leq h(\alpha)]\subset C(\Omega)^2$ denote the sets of functions $\alpha$ and $\beta$ such that the respective conditions hold pointwise.
	Moreover we have
	\begin{equation}
		\label{eqn:CConjugateAlphaMin}
		\cl\{\alpha\in\R\,|\,h(\alpha)<\infty\}=(-\infty,c(0,1)]\,.
	\end{equation}
\end{lemma}
\begin{proof}
We have
\begin{multline*}
\preconj[D(\rho,\cdot)](\alpha)
=\sup_{\hat\rho\in\meas(\Omega)}\int_\Omega\alpha\,\d\hat\rho-\Dl(\rho,\hat\rho)\\
=\sup_{\substack{\hat\rho,\mu\in\measp(\Omega)\,:\\\rho,\hat\rho\ll\mu}}\int_\Omega\alpha\RadNik{\hat\rho}{\mu}-c\left(\RadNik{\rho}{\mu},\RadNik{\hat\rho}{\mu}\right)\,\d\mu
=\sup_{\substack{\mu\in\measp(\Omega)\,:\,\rho\ll\mu\\g\geq0\text{ measurable}}}\int_\Omega\alpha g-c\left(\RadNik{\rho}\mu,g\right)\,\d\mu\,.
\end{multline*}
Now take as a test case $\mu=\delta_{\hat x}+\rho$ for $\hat x\in\Omega$ and $g(\hat x)=n\in\N$, $g(x)=1$ else.
We obtain
\begin{align*}
\preconj[D(\rho,\cdot)](\alpha)
	& \geq\int_\Omega\alpha g-c\left(\RadNik{\rho}\mu,g\right)\,\d\mu \\
	& = (1+\rho(\{\hat x\})) \cdot \left[ n \cdot \alpha(\hat x) - n \cdot c \left(\RadNik{\rho}\mu(\hat x)/n,1\right) \right] + \int_{\Omega \setminus \{\hat x\}} \left( \alpha -c(1,1) \right)\,\d \rho
\end{align*}
which, due to $c(z/n,1)\to c(0,1)$ for any $z\geq0$, diverges to infinity as $n\to\infty$ if $\alpha(\hat x)>c(0,1)$. Thus
\begin{equation*}
\preconj[D(\rho,\cdot)](\alpha)=\infty\qquad\text{if }\alpha(x)>c(0,1)\text{ for any }x\in\Omega\,.
\end{equation*}
Furthermore, for any $\mu\in\measp(\Omega)$ let us denote the Lebesgue decomposition by $\mu=g_\mu\rho+\mu^\perp$, where $g_\mu$ is a density and $\mu^\perp$ is the singular part with respect to $\rho$.
If $\alpha(x)\leq c(0,1)$ for all $x\in\Omega$, we have
\begin{align*}
\preconj[D(\rho,\cdot)](\alpha)
&=\sup_{\substack{\mu\in\measp(\Omega),\,\rho\ll\mu\\g\geq0\text{ measurable}}}\int_\Omega\alpha g-c(0,g)\,\d\mu^\perp+\int_\Omega\left(\alpha g-c(\tfrac1{g_\mu},g)\right)g_\mu\,\d\rho\\
&=\sup_{\tilde g\text{ measurable}}\int_\Omega\alpha\tilde g-c(1,\tilde g)\,\d\rho\,.
\end{align*}
Since $c$ is a normal integrand and $\alpha\in L^\infty(\Omega,\rho)$ we have (see e.\,g.\ \cite[Thm.\,VII-7]{CaVa77})
\begin{equation*}
\preconj[D(\rho,\cdot)](\alpha)
\geq\sup_{\tilde g\in L^1(\Omega,\rho)}\int_\Omega\alpha\tilde g-c(1,\tilde g)\,\d\rho
=\int_\Omega[c(1,\cdot)]^\ast(\alpha)\,\d\rho\,,
\end{equation*}
while on the other hand also
\begin{equation*}
\preconj[D(\rho,\cdot)](\alpha)
\leq\int_\Omega\sup_{\tilde g}\left(\alpha\tilde g-c(1,\tilde g)\right)\,\d\rho
=\int_\Omega[c(1,\cdot)]^\ast(\alpha)\,\d\rho\,.
\end{equation*}
Finally, it is straightforward to check $$\cl\{\alpha\in\R\,|\,[c(1,\cdot)]^\ast(\alpha)<\infty\}=(-\infty,c(0,1)]$$
so that summarising, we arrive at $\preconj[D(\rho,\cdot)](\alpha)=\int_\Omega h(\alpha)\,\d\rho+\iota_{\cl[h(\alpha)<\infty]}(\alpha)$.

As for $\preconj D$ we can finally compute
\begin{equation*}
\preconj D(\beta,\alpha)
=\sup_{\rho\in\measp(\Omega)}\int_\Omega\beta\,\d\rho+\preconj[D(\rho,\cdot)](\alpha)
=\sup_{\rho\in\measp(\Omega)}\int_\Omega\beta+h(\alpha)\,\d\rho+\iota_{\cl[h(\alpha)<\infty]}(\alpha)\,,
\end{equation*}
where it is straightforward to identify the right-hand side with $\iota_{[\beta\leq-h(\alpha)]}(\beta,\alpha)$.
\end{proof}

\begin{lemma}[Support function of $\Lip(\Omega)$]
	\label{lem:SupLip}
	For $\rho\in\meas(\Omega)$ we have
	$$\iota_{\Lip(\Omega)}^\ast(\rho)=\inf_{\mu\in\meas(\Omega)}W_1(\mu,\rho+\mu)\,,$$
	where the infimum is achieved if it is finite.
\end{lemma}
\begin{proof}

Let $\mu\in\meas(\Omega)$ be any measure with $\mu,\rho+\mu\in\measp(\Omega)$.
By Remark \ref{rem:W1reduction} we have
\begin{equation*}
W_1(\mu,\rho+\mu)
=\sup_{\beta\in\Lip(\Omega)}\int_\Omega-\beta\,\d\mu+\int_\Omega\beta\,\d(\rho+\mu)
=\sup_{\beta\in\Lip(\Omega)}\int_\Omega\beta\,\d\rho
=\iota_{\Lip(\Omega)}^\ast(\rho)\,.
\end{equation*}
Taking the infimum over $\mu\in\meas(\Omega)$ yields the result (recall that $W_1(\mu,\rho+\mu)=\infty$ for $\mu\notin\measp(\Omega)$ or $\rho+\mu\notin\measp(\Omega)$).
Obviously, the infimum is achieved by any $\mu$ with $\mu,\rho+\mu\in\measp(\Omega)$.
\end{proof}

\begin{proof}[Proof of Proposition~\ref{prop:PrimalDualSandwich}]
To obtain the primal formulation of $W_{\Dl,\Dm,\Dr}$ it is sufficient to replace any occurrence of $W_1$ by its definition \eqref{eqn:W1primal}, so let us now consider the predual formulation.
Let us abbreviate
\begin{equation*}
K=\{(\alpha,\beta)\in C(\Omega)^2\,|\,(\alpha(x),\beta(x))\in B_{01}\,\forall x\in\Omega\}\,.
\end{equation*}
By Lemmas~\ref{lem:FenchelMoreau} and \ref{lem:conjugateD} we have
\begin{equation*}
\Dm=(\preconj{\Dm})^\ast=\iota_K^\ast
\end{equation*}
(note that Lemma~\ref{lem:FenchelMoreau} can be applied to $\Dm$ due to Proposition~\ref{prop:LSCLocDis}).
Next consider the indicator function $H:C(\Omega)^2\to\{0,\infty\}$,
\begin{equation*}
H(\alpha,\beta)=\iota_{K}(\alpha,\beta)+\iota_{\Lip(\Omega)}(\alpha)+\iota_{\Lip(\Omega)}(\beta)\,.
\end{equation*}
Denoting by $\square$ the infimal convolution and by $\cl$ the closure of functions (also known as lower semi-continuous envelope), for $\mu,\nu\in\meas(\Omega)$ we have
\begin{align*}
H^\ast(\mu,\nu)
&=\cl\left[\iota_K^\ast\square(\iota_{\Lip(\Omega)},\iota_{\Lip(\Omega)})^\ast\right](\mu,\nu)\\
&=\cl\left[\inf_{\hat\eta,\hat\theta\in\meas(\Omega)}\Dm(\mu-\hat\eta,\nu-\hat\theta)+\iota_{\Lip(\Omega)}^\ast(\hat\eta)+\iota_{\Lip(\Omega)}^\ast(\hat\theta)\right]\\
&=\cl\bigg[\inf_{\hat\eta,\hat\theta\in\meas(\Omega)}\min_{\eta,\theta\in\meas(\Omega)}\Dm(\mu-\hat\eta,\nu-\hat\theta)+W_1(\hat\eta+\eta,\eta)+W_1(\hat\theta+\theta,\theta)\bigg]
\end{align*}
by Lemma~\ref{lem:SupLip}.
Substituting $\zeta=\hat\eta+\eta$ and $\xi=\hat\theta+\theta$ we arrive at
\begin{align*}
H^\ast(\mu,\nu)
&=\cl\bigg[\inf_{\zeta,\xi\in\meas(\Omega)}\min_{\eta,\theta\in\meas(\Omega)}\Dm(\mu-\zeta+\eta,\nu-\xi+\theta)+W_1(\zeta,\eta)+W_1(\xi,\theta)\bigg]\\
&=\min_{\zeta,\xi,\eta,\theta\in\meas(\Omega)}\Dm(\mu-\zeta+\eta,\nu-\xi+\theta)+W_1(\zeta,\eta)+W_1(\xi,\theta)\,,
\end{align*}
where the closedness of the right-hand side follows as in Proposition~\ref{prop:seqLSCW}
and the existence of minimizers follows via the direct method from the sequential weak-* lower semi-continuity of $W_1$ and $\Dm$
as well as the fact that the minimization may be restricted to $\|\zeta\|_\meas=\|\eta\|_\meas\leq\|\mu\|_\meas$ and $\|\xi\|_\meas=\|\theta\|_\meas\leq\|\nu\|_\meas$.
Indeed, we may assume $\eta$ and $\zeta$ to be nonnegative (else the functional would be infinite) and to be singular, $\eta\perp\zeta$,
since otherwise we can subtract their common part from both $\eta$ and $\zeta$ without changing the functional value.
Then, however, we require $\zeta\leq\mu$, for otherwise we would have $\mu-\zeta+\eta\notin\measp(\Omega)$ and the functional would be infinite.
The analogous holds for $\theta$ and $\xi$.

Now abbreviate for $i=0,1$
\begin{equation*}
F_i:C(\Omega)\to(-\infty,\infty],\quad
F_i(\alpha)=-\int_\Omega h_i(\alpha)\,\d\rho_i+\iota_{\cl[h_i(\alpha)>-\infty]}\,.
\end{equation*}
By Lemma~\ref{lem:conjugateD} we have
\begin{equation*}
F_0^\ast(\rho)=\Dl(\rho_0,-\rho)\,,
\qquad
F_1^\ast(\rho)=\Dr(-\rho,\rho_1)\,.
\end{equation*}
Furthermore, by the conditions in Definition~\ref{def:LocalSimilarityMeasure} there exist $\alpha,\beta\in\Lip(\Omega)$ with $(\alpha(x),\beta(x))\in\mathrm{int}(B_{01}\cap(\dom \h\times\dom \g))$ for all $x\in\Omega$
(for instance take $\alpha\equiv\beta\equiv-\delta$ for $\delta>0$ small enough).
Thus we have strong Fenchel duality \cite[p.\,201, Thm.\,1]{Luenberger69}, that is,
\begin{align*}
&\sup_{\alpha,\beta\in C(\Omega)}-\left[(F_0,F_1)(\alpha,\beta)+H(\alpha,\beta)\right]\\
=&\min_{\rho_0',\rho_1'\in\meas(\Omega)}\left[(F_0^\ast,F_1^\ast)(-\rho_0',-\rho_1')+H^\ast(\rho_0',\rho_1')\right]\\
=&\min_{\rho_0',\rho_1',\eta,\theta,\zeta,\xi\in\meas(\Omega)}\Dl(\rho_0,\rho_0')+\Dr(\rho_1',\rho_1)
+\Dm(\rho_0'-\zeta+\eta,\rho_1'-\xi+\theta)+W_1(\zeta,\eta)+W_1(\xi,\theta)\\
=&\min_{\rho_0',\rho_1',\rho_0'',\rho_1'',\eta,\theta\in\measp(\Omega)}
\Dl(\rho_0,\rho_0')+W_1(\rho_0'+\eta-\rho_0'',\eta)+\Dm(\rho_0'',\rho_1'')\\
& \qquad \qquad +W_1(\theta,\rho_1'+\theta-\rho_1'')+\Dr(\rho_1',\rho_1)\\
=&W_{\Dl,\Dm,\Dr}(\rho_0,\rho_1)\,,
\end{align*}
where in the last step we used $W_1(\rho_0'+\eta-\rho_0'',\eta)=W_1(\rho_0',\rho_0'')$ and $W_1(\theta,\rho_1'+\theta-\rho_1'')=W_1(\rho_1'',\rho_1')$
due to the fact that $W_1(\mu,\nu)=W_1(\mu+\rho,\nu+\rho)$ for any $\mu,\nu\in\measp(\Omega)$ and $\rho\in\meas(\Omega)$ such that $\mu+\rho,\nu+\rho\in\measp(\Omega)$.
Note that in the above calculation we have assumed the value of the optimization problem to be finite (else the $\min$ would have to be replaced by $\inf$)
so that Fenchel duality automatically yields the existence of optimal $\rho_0',\rho_0'',\rho_1'',\rho_1'$.
\end{proof}

The predual formulation \eqref{eqn:predual} of the $W_{\Dl,\Dm,\Dr}$-discrepancy \eqref{eq:SandwichProblem} already looks very similar to a $W_{\h,\g,B}$-type formulation, \eqref{eq:GeneralizedPrimalProblem}.
For the equivalence it remains to establish that the functions $\gB$, $\h$ and $\g$ as defined in \eqref{eqn:predualInducedH} are admissible in the sense of Definition \ref{def:admissibility}
and that conversely, admissible choices $\gB$, $\h$ and $\g$ can indeed be induced via \eqref{eqn:predualInducedH} from local discrepancy integrands $\Cl$, $\Cm$ and $\Cr$ in the sense of Definition \ref{def:LocalSimilarityMeasure}.

\begin{lemma}[Conversion of $c$ and $h$]
	\label{lem:chConversion}
	Let $c$ be a local discrepancy integrand in the sense of Definition \ref{def:LocalSimilarityMeasure}. Then $$h : \alpha \mapsto -[c(1,\cdot)]^\ast(-\alpha)$$ is an admissible function in the sense of Definition \ref{def:admissibility}, and $\alphamin = -c(0,1)$ (see \eqref{eqn:AlphaBetaMinA}-\eqref{eqn:AlphaBetaMinB}).
	
	Conversely, if $h$ is admissible, then
	\begin{equation}
		\label{eqn:cInduced}
		c : (m_0,m_1) \mapsto \begin{cases}
			m_0\,(-h)^\ast(-\frac{m_1}{m_0}) & \tn{ if } m_0 > 0,\, m_1 \geq 0, \\
			m_1\,\lim_{z \to \infty} \big((-h)^\ast(-z)/z \big) & \tn{ if } m_0 = 0,\,m_1 > 0, \\
			0 & \tn{ if } m_0 = 0,\, m_1 = 0, \\
			+ \infty & \tn{ else,}
			\end{cases}
	\end{equation}
	is a local discrepancy integrand and $h = -[c(1,\cdot)]^\ast(-\cdot)$.
\end{lemma}
\begin{proof}
	For a given $c$, the induced $h$ is concave and upper semi-continuous due to the properties of the Fenchel-Legendre conjugate (conditions \ref{enm:convexity} and \ref{enm:wellposedness}).
	Furthermore, condition~\ref{enm:positivity1} is a consequence of $c\geq0$ and $c(1,1)=0$ for all local discrepancy integrands $c$,
while condition~\ref{enm:positivity} then follows from the strict positivity property~\ref{enm:strictPos} of Definition~\ref{def:LocalSimilarityMeasure} due to the conjugate subgradient theorem (see e.\,g.\ \cite[Prop.\,16.13]{BauschkeCombettes2011}).
Finally, condition~\ref{enm:negativeMeasures} from Definition~\ref{def:admissibility} is implied by property~\ref{enm:domain} of Definition~\ref{def:LocalSimilarityMeasure}. The value of $\alphamin = -c(0,1)$ is given by Lemma \ref{lem:conjugateD}.

Now let an admissible $h$ be given and consider the induced $c$. We need to show that it satisfies the properties in Definition~\ref{def:LocalSimilarityMeasure}.
It is a straightforward exercise to show that property~\ref{enm:strictPos} is implied by conditions~\ref{enm:positivity1} and \ref{enm:positivity} on $h$.
The positive one-homogeneity of $c$ follows by definition.
Convexity is implied by this one-homogeneity together with the subadditivity
\begin{multline*}
c(m_0+n_0,m_1+n_1)
=(m_0+n_0)(-h)^\ast\left(-\tfrac{m_1+n_1}{m_0+n_0}\right)\\
\leq(m_0+n_0)\left[\tfrac{m_0}{m_0+n_0}(-h)^\ast\left(-\tfrac{m_1}{m_0}\right)+\tfrac{n_0}{m_0+n_0}(-h)^\ast\left(-\tfrac{n_1}{n_0}\right)\right]
=c(m_0,n_0)+c(m_1,n_1)\,,
\end{multline*}
where we have used convexity of $(-h)^\ast(-\cdot)$.
The lower semi-continuity of $c$ on $(0,+\infty) \times [0,+\infty)$ is a direct consequence of the lower semi-continuity of $(-h)^\ast(-\cdot)$ and the continuity of $(m_0,m_1)\mapsto(m_0,\frac{m_1}{m_0})$. The function for $m_0=0$ and $m_1>0$ is just defined as the limit $m_0 \to 0$ for $m_1 > 0$, and lower semi-continuity in $(m_0,m_1) = (0,0)$ follows since $0$ is the global minimum of $c$ (this establishes property \ref{enm:cConv}).
Property~\ref{enm:domain} is satisfied by definition of $c$.
Finally, it is a simple exercise to show that the $h$ induced by $c$ is in fact the original $h$.
\end{proof}

We can finally state the equivalence relation between \eqref{eq:GeneralizedPrimalProblem} and \eqref{eq:SandwichProblem}.

\begin{corollary}[Equivalence of formulations]
	\label{cor:equivalenceStatic}
	Let the local discrepancies $\Dl$, $\Dm$, and $\Dr$ be induced by the integrands $\Cl$, $\Cm$, and $\Cr$.
	Then $W_{\Dl,\Dm,\Dr}=W_{\h,\g,B}$ for
	\begin{gather*}
	\h(\alpha)=-[\Cl(1,\cdot)]^\ast(-\alpha)\,,\qquad
	\g(\beta)=-[\Cr(\cdot,1)]^\ast(-\beta)\,,\\
	\gB(\alpha)=-[\Cm(\cdot,1)]^\ast(-\alpha)\,,\qquad
	\hB(\beta)=-[\Cm(1,\cdot)]^\ast(-\beta)\,,\\
	B_{01}=\{(\alpha,\beta)\in\R^2\,|\,\alpha\leq\hB(-\beta)\} = \{(\alpha,\beta)\in\R^2\,|\,\beta\leq\gB(-\alpha)\} \,,\\
	B_0 = [\alphamin,+\infty)\,, \qquad \alphamin = -\Cl(0,1)\,, \\
	B_1 = [\betamin,+\infty)\,, \qquad  \betamin=-\Cr(1,0)\,,\\
	B = B_{01} \cap (B_0 \times B_1)\,,
	\end{gather*}
	and the triple $(\h,\g,B)$ is admissible in the sense of Definition \ref{def:admissibility}.
	
	Conversely, given admissible $(\h,\g,B)$ and $\hB$ describing $B$ we have $W_{\h,\g,B}=W_{\Dl,\Dm,\Dr}$ with $\Cl$, $\Cr$ and $\Cm$ being induced by $\h$, $\g$ and $\hB$ as given in \eqref{eqn:cInduced} (for the conversion of $\g$ into $\Cr$ the arguments $m_0$ and $m_1$ have to be swapped). $(\Dl,\Dm,\Dr)$ are local discrepancy measures in the sense of Definition \ref{def:LocalSimilarityMeasure}.
\end{corollary}
\begin{proof}
The claim follows directly from Proposition \ref{prop:PrimalDualSandwich} and Lemma \ref{lem:chConversion}.
\end{proof}

The above identification of different formulations allows to always choose the more convenient for analytical and numerical purposes.
The following corollary makes use of this fact and proves a reduction of particular infimal convolution type discrepancies via Proposition \ref{prop:modelReduction},
which is by far not obvious from Definition~\ref{def:SandwichProblem}.
It can be stated more elegantly with an auxiliary Lemma.
\begin{lemma}[Infimal convolution of local discrepancy measures]
	\label{lem:LocalConvolution}
	Let $D_0$, $D_1$ be two local discrepancy measures in the sense of Definition \ref{def:LocalSimilarityMeasure} with integrands $c_0$, $c_1$, and let $h_i= -[c_i(1,\cdot)]^*(-\cdot)$ for $i\in\{0,1\}$.
	Further, let $c$ be the discrepancy integrand induced by $h=h_0 \circ h_1$ via \eqref{eqn:cInduced} and $D$ be the corresponding discrepancy measure.
	Then, $D=D_0 \diamond D_1$.
\end{lemma}
\begin{proof}
	It is easy to see that $h=h_0 \circ h_1$ is indeed admissible in the sense of Definition \ref{def:admissibility} and therefore, by virtue of Lemma \ref{lem:chConversion} $c$ is an admissible local discrepancy integrand. With arguments as in Lemma \ref{lem:conjugateD} one can show that $D=D_0 \diamond D_1$ is equivalent to $c=c_0 \diamond c_1$. Consider now $m_0$, $m_1 > 0$. Introducing $g_1= -[c_1(\cdot,1)]^*(-\cdot)$, we find
	\begin{align*}
		(c_0 \diamond c_1)(m_0,m_1) & = \inf_{m \in \R}
			c_0(m_0,m) + c_1(m,m_1)
		= \sup_{\alpha \in \R} 
			\underbrace{-[c_0(m_0,\cdot)]^*(-\alpha)}_{m_0 \cdot h_0(\alpha)}
			\underbrace{-[c_1(\cdot,m_1)]^*(\alpha)}_{m_1 \cdot g_1(-\alpha)} \\
		& = \sup_{\substack{(\alpha,\beta) \in \R^2 \colon\\
			\beta \leq g_1(-\alpha)}} m_0 \cdot h_0(\alpha) + m_1 \cdot \beta 
		 = \sup_{\substack{(\alpha,\beta) \in \R^2 \colon\\
			\alpha \leq h_1(-\beta)}} m_0 \cdot h_0(\alpha) + m_1 \cdot \beta \\
		& = \sup_{\substack{(\alpha,\beta) \in \R^2 \colon\\
			\alpha \leq h_0(h_1(-\beta))}} m_0 \cdot \alpha + m_1 \cdot \beta = c(m_0,m_1)\,.
	\end{align*}
	In the second equality we have used the Fenchel--Rockafellar Theorem and that $h_0(0)=g_1(0)=0$ and both functions are continuous at $0$ (see Definition \ref{def:admissibility}).
	For $m_0<0$ or $m_1<0$ it is easy to verify that $(c_0 \diamond c_1)(m_0,m_1)=\infty$, and so is $c(m_0,m_1)$ since it is admissible.
	Finally, since $c$, $c_0$, and $c_1$ (and thus also $(c_0 \diamond c_1)$ by Remark \ref{rem:propertiesInfConv}\eqref{enm:metricInfConv}) are convex and lower semi-continuous, they must all be right-continuous in $m_0=0$ or $m_1=0$.
	Since $c$ and $(c_0 \diamond c_1)$ coincide on $(0,+\infty)^2$, this implies their equality also on $\{0\} \times [0,\infty) \cup [0,\infty) \times \{0\}$.
\end{proof}

\begin{corollary}[Reduction of infimal convolution formulation]
	\label{cor:reductionInfimalConv}
	Let the local discrepancies $\Dl$, $\Dm$, and $\Dr$ be induced by the integrands $\Cl$, $\Cm$, and $\Cr$, and let $\h$, $\g$, $\hB$, and $\gB$ be as in Corollary \ref{cor:equivalenceStatic}.
	Furthermore let $z\in(0,\infty]$, and let the extended discrete metric be defined by \eqref{eqn:discreteMetric}.
	\begin{itemize}
	\item If $\Cm(m_0,m_1)=z|m_0-m_1|$ for all $m_1>m_0 \geq 0$, then
		\begin{equation*}
		W_{\Dl,\Dm,\Dr}
		=W_{\Dl \diamond \Dm,D^d,\Dr}
		=\Dl \diamond \Dm\diamond W_1\diamond\Dr\,.
		\end{equation*}
	\item If $\Cm(m_0,m_1)=z|m_0-m_1|$ for all $0 \leq m_1<m_0$, then
		\begin{equation*}
		W_{\Dl,\Dm,\Dr}
		=W_{\Dl,D^d,\Dm \diamond \Dr}
		=\Dl\diamond W_1\diamond \Dm \diamond \Dr\,.
		\end{equation*}
	\end{itemize}
\end{corollary}
\begin{proof}
Consider the first statement (the second follows analogously).
The second equality follows from $W_1\diamond D^d\diamond W_1=W_1\diamond W_1=W_1$ (Remark \ref{rem:UnbalancedMeasures}).
As for the first equality, by Corollary \ref{cor:equivalenceStatic} and Proposition \ref{prop:modelReduction} we have
\begin{equation*}
W_{\Dl,\Dm,\Dr}
=W_{\h,\g,B}
=W_{\h\circ\hB,\g,B(\tildealphamin,\betamin)}\,,
\end{equation*}
where we have used the notation from Corollary \ref{cor:equivalenceStatic} and Proposition \ref{prop:modelReduction}
as well as the fact $\gB(\alpha)=\min\{\alpha,z\}$ for all $\alpha\geq0$.
Now it is straightforward to check via Lemma \ref{lem:LocalConvolution} and Corollary \ref{cor:equivalenceStatic} that the equivalent formulation of $W_{\Dl \diamond \Dm,D^d,\Dr}$ is also $W_{\h\circ\hB,\g,B(\tildealphamin,\betamin)}$.
\end{proof}

Let us close by giving yet another, flow-based formulation.

\begin{remark}[Flow formulation]\label{rem:flowFormulation}
In Euclidean space $\R^n$, the following relation is known as Beckmann's problem \cite[Thm.\,4.6]{Santambrogio15},
\begin{equation*}
W_1(\rho_0,\rho_1)=\min_{\phi\in\meas(\Omega)^n,\,\div \phi=\rho_1-\rho_0}\||\phi|\|_{\meas}\,,
\end{equation*}
where the divergence is taken in any open superset of $\Omega$ (cf.\ also \cite{LellmannKantorovichRubinstein2014}) in the distributional sense.
As a consequence we have the alternative flow formulation
\begin{multline*}
W_{\Dl,\Dm,\Dr}(\rho_0,\rho_1)=
\min_{\rho_0'',\rho_1''\in\measp(\Omega),\,\phi,\psi\in\meas(\Omega)^n}\Dl(\rho_0,\rho_0''-\div \phi)+\||\phi|\|_\meas\\+\Dm(\rho_0'',\rho_1'')+\||\psi|\|_\meas+\Dr(\rho_1''-\div \psi,\rho_1)\,,
\end{multline*}
where $\phi$ and $\psi$ have the interpretation of a mass flow field and can thus be used in applications to extract flow information.
The divergence essentially arises via dualization of the Lipschitz constraint in \eqref{eqn:W1predual}, interpreted as a local constraint $|\nabla\alpha|\leq1$ almost everywhere
(see Section~\ref{sec:variableSplitting}).
\end{remark}

\subsection{The inhomogeneous case}
\label{sec:Inhomogeneous}
For ease of exposition we restricted ourselves to the spatially homogeneous case in the previous sections.
For the sake of completeness we shall here briefly comment on what changes if all integrands become space-dependent.
In more detail, $\h$, $\g$, and $B$ in Definition~\ref{def:GeneralizedPrimalProblem} are now also allowed to depend on $x\in\Omega$, and the energy becomes
\begin{equation*}
E_{\h,\g,B}^{\rho_0,\rho_1}(\alpha,\beta) = \begin{cases}
	\int_\Omega \h(\alpha(x),x)\,\d\rho_0(x)\\ \;+ \int_\Omega \g(\beta(x),x)\,\d\rho_1(x)&
	\text{if }\alpha, \beta \in \Lip(\Omega)\text{ with }(\alpha(x),\beta(x)) \in B(x) \forall\, x \in\Omega\,,\\-\infty&\text{else.}\end{cases}
\end{equation*}
The spatial dependence demands two additional admissibility conditions in Definition~\ref{def:admissibility}:
$\h$, $\g$, and $B$ must depend measurably on $x$ so that $E_{\h,\g,B}^{\rho_0,\rho_1}$ is well-defined, where $B(x)=B_{01}(x)\cap(B_0(x)\times B_1(x))$ for
\begin{gather*}
B_{01}(x)=\left\{(\alpha,\beta)\in\R^2\,\right|\left.\alpha\leq\hB(-\beta,x)\right\}=\left\{(\alpha,\beta)\in\R^2\,\right|\left.\beta\leq\gB(-\alpha,x)\right\}\,,\\
B_0(x)=\cl\{\alpha\in\R\,|\,\h(\alpha,x)>-\infty\}\,,\qquad
B_1(x)=\cl\{\beta\in\R\,|\,\g(\beta,x)>-\infty\}\,.
\end{gather*}
Furthermore,
\begin{equation*}
\tilde{B}_0(x)=\{(\alpha,\beta)\in\R^2\,|\,\alpha\leq\h(-\beta,x)\}\,,\quad
\tilde{B}_1(x)=\{(\alpha,\beta)\in\R^2\,|\,\beta\leq\g(-\alpha,x)\}\,,\text{ and }B(x)
\end{equation*}
have to be lower semi-continuous in $x$
(a multifunction $x\mapsto B(x)$ is said to be lower semi-continuous if the set $\{x\in\Omega\,|\,U\cap B(x)\neq\emptyset\}$ is open for any open set $U$)
to ensure sequential \mbox{weak-*} lower semi-continuity of $W_{\h,\g,B}$.
At first glance one might expect the stronger requirement of $\h$ and $\g$ being upper semi-continuous in $x$, however,
due to the restriction of $\alpha$ and $\beta$ to the domains of $\h$ and $\g$ we essentially only need upper semi-continuity where $\h$ and $\g$ are finite,
resulting in lower semi-continuity of $\tilde{B}_0$ and $\tilde{B}_1$.
That this heuristic intuition is correct will turn out from the model equivalence below.

Likewise, the local discrepancy integrand in Definition~\ref{def:LocalSimilarityMeasure} may depend on $x$ as well, inducing a discrepancy
\begin{equation*}
D(\rho_0,\rho_1) = \int_\Omega c(\RadNik{\rho_0}{\rho}(x),\RadNik{\rho_1}{\rho}(x),x)\,d\rho(x)\,.
\end{equation*}
Of course, in addition to the properties from Definition~\ref{def:LocalSimilarityMeasure} $c$ is required to be measurable also in $x$.
Furthermore it needs to be jointly lower semi-continuous in all its arguments,
which can easily be seen to be a necessary condition for sequential weak-* lower semi-continuity of $D$ (its sufficiency is shown below).

Under the above natural conditions, the equivalence between both models stays valid with slight proof variations.
In particular, the proof of weak-* lower semi-continuity, Proposition~\ref{prop:LSCLocDis}, becomes more complicated as more work would be required to apply \cite[Thm.\,2]{BouchitteValadier88} in the last step.
Instead we can make use of \cite[Lemma\,3.7]{BoBu90} (which performs that additional work inside its proof).
Before, we require a little lemma which takes the role of \cite[Lemma\,3.6]{BoBu90}.

\begin{lemma}
$F:\meas(\Omega)^2\to(-\infty,\infty]$ is weakly-* lower semi-continuous
if and only if $F_\varepsilon(\rho)=F(\rho)+\varepsilon\int_\Omega
	\binom{1}{1}
	\cdot\d\rho$ is so for any $\varepsilon>0$.
\end{lemma}
\begin{proof}
This follows directly from the fact that $\varepsilon \int_\Omega \binom{1}{1} \cdot\d\rho$ is a weak-* continuous perturbation.
\end{proof}

\begin{proposition}[Lower semi-continuity]
$D(\rho_0,\rho_1)= \int_\Omega c(\RadNik{\rho_0}{\rho}(x),\RadNik{\rho_1}{\rho}(x),x)\,d\rho(x)$ is weakly-* lower semi-continuous on $\meas(\Omega)^2$.
\end{proposition}
\begin{proof}
\newcommand{\auxSet}{A}
Due to the previous lemma it is sufficient to show weak-* lower semi-continuity of
$D_\varepsilon(\rho_0,\rho_1) \allowbreak=\allowbreak \int_\Omega c_\varepsilon(\RadNik{\rho_0}{\rho}(x),\RadNik{\rho_1}{\rho}(x),x)\,d\rho(x)$
with $c_\varepsilon(m_0,m_1,x)=c(m_0,m_1,x)+\varepsilon(m_0+m_1)$.
Define the sets
\begin{align*}
\auxSet(x) & = \{u\in\R^2\,|\,u\cdot m\leq c_\varepsilon(m_0,m_1,x)\text{ for all }m=(m_0,m_1)\in\R^2\}\,, \\
\tn{and} \quad H & =\{u\in C(\Omega)\,|\,u(x)\in\auxSet(x)\text{ for all }x\in\Omega\}
\end{align*}
and consider a net $\rho^a=(\rho_0^a,\rho_1^a)$ converging weakly-* against $\rho=(\rho_0,\rho_1)$ in $\meas(\Omega)^2$.
For an arbitrary $u\in H$ we have
\begin{equation*}
D_\varepsilon(\rho_0^a,\rho_1^a)
=\int_\Omega c_\varepsilon\left(\RadNik{\rho_0^a}{|\rho^a|},\RadNik{\rho_1^a}{|\rho^a|},x\right)\,\d|\rho^a|
\geq\int_\Omega u(x)\cdot\d\rho^a(x)
\to\int_\Omega u(x)\cdot\d\rho\,.
\end{equation*}
Now note that $\auxSet(x)$ is lower semi-continuous in $x$ due to the lower semi-continuity of $c_\varepsilon$ and $c_\varepsilon(m_0,m_1,x)=\sup_{u\in\auxSet(x)}u_0m_0+u_1m_1$ \cite[Thm.\,17]{BoVa89}
and that $\{u\in\R^2\,|\,|u|\leq\varepsilon\}\subset\auxSet(x)$ for all $x\in\Omega$.
Thus we may apply \cite[Lemma\,3.7]{BoBu90} with $\mu=0$, $\lambda=\rho$, and $h=c_\varepsilon$, yielding
\begin{equation*}
\sup_{u\in H}\int_\Omega u(x)\cdot\d\rho
=\int_\Omega c_\varepsilon\left(\RadNik{\rho_0}{|\rho|},\RadNik{\rho_1}{|\rho|},x\right)\,\d|\rho|
=D_\varepsilon(\rho_0,\rho_1)
\end{equation*}
which concludes the proof.
\end{proof}

Another change is required in Remark~\ref{rem:coercivity}, where the coercivity estimate now turns into $D(\rho_0,\rho_1)\geq\tilde c(\measnrm{\rho_0},\measnrm{\rho_1})$
with $\tilde c(m_0,m_1)$ the convex envelope of $\min_{x\in\Omega}c(m_0,m_1,x)$.
Note that $\tilde c(m_0,m_1)$ is well-defined (due to the lower semi-continuity of $c$ in $x$ and the compactness of $\Omega$) and strictly positive for $m_0\neq m_1$.

The final addition concerns the proof of Corollary \ref{cor:equivalenceStatic},
where we need to show that the lower semi-continuity of $\tilde{B}_0$, $\tilde{B}_1$, and $B$ in $x$ is equivalent to the lower semi-continuity of $\Cl$, $\Cm$, and $\Cr$ in all their arguments.
This is true by \cite[Thm.\,17]{BoVa89}.

%% file: 03-examples.tex
\section{Examples}
\label{sec:Examples}
In this section we provide a few examples of local discrepancies and their combinations with $W_1$.
For the reader's convenience we furthermore list all the functionals required
to move from formulation\,\eqref{eq:SandwichProblem} to formulation\,\eqref{eq:GeneralizedPrimalProblem}.

\subsection{Different discrepancy measures}
In Table \ref{tab:discrepancies} and Figures \ref{fig:cost} and \ref{fig:costConj} we list several commonly used discrepancies, the corresponding integrands $c$ from Definition \ref{def:LocalSimilarityMeasure} and their convex conjugates $[c(1,\cdot)]^*$ used in the equivalence relation of Corollary\,\ref{cor:equivalenceStatic}.
The special case of the extended discrete metric $\SimDisc$, which does not allow any mass change, has already been introduced in \eqref{eqn:discreteMetric}.
The standard case of the total variation metric $\SimTV$ is just the norm of the difference; we will see that it has implicitly been used in \cite{LellmannKantorovichRubinstein2014}.
In Section \ref{sec:unbalancedExamples} we discuss the slightly more general variant $\SimTV_{a,b}$ with different weights for increasing and reducing mass.
The examples of squared Hellinger--Kakutani distance $\SimFR$, Jensen--Shannon divergence $\SimJS$, and symmetric $\chi^2$ measure $\SimChi$
have also been considered as examples in \cite[(1.14)-(1.17)]{LieroMielkeSavare-HellingerKantorovich-2015a}.
Further examples are given by the entropy functions $\SimEp$ for $p\in\R$ (see e.\,g.\ \cite[(1.4)]{LieroMielkeSavare-HellingerKantorovich-2015a}).

\begin{table}
\def\arraystretch{1.5}%
\centering
\begin{tabular}{c|c|c}
&$c(m_0,m_1)$
&$[c(1,\cdot)]^\ast(\alpha)$\\\hline
$\SimDisc$
&$\begin{cases}0&\text{if }m_0=m_1\\\infty&\text{else}\end{cases}$
&$\alpha$\\\hline
$\SimTV$
&$|m_1-m_0|$
&$\max(\alpha,-1)+\iota_{(-\infty,1]}(\alpha)$
\\\hline
$\SimTV_{a,b}$
&$\begin{cases}
		a|m_1-m_0|&\text{if }m_1\geq m_0,\\
		b|m_1-m_0|&\text{else.}
		\end{cases}$
&$\max(\alpha,-b)+\iota_{(-\infty,a]}(\alpha)$
\\\hline
$\SimFR$
&$(\sqrt{m_1}-\sqrt{m_0})^2$
&$\frac{\alpha}{1-\alpha}+\iota_{(-\infty,1)}(\alpha)$
\\\hline
$\SimJS$
&$m_0\log_2\frac{2m_0}{m_0+m_1}+m_1\log_2\frac{2m_1}{m_0+m_1}$
&$\log_2\tfrac1{2-2^\alpha}+\iota_{(-\infty,1)}(\alpha)$
\\\hline
$\SimChi$
&$\tfrac{(m_1-m_0)^2}{m_1+m_0}$
&$\begin{cases}-1&\text{if }\alpha<-3\\4-\alpha-4\sqrt{1-\alpha}&\text{if }\alpha\in[-3,1]\\\infty&\text{else}\end{cases}$
\\\hline
$\SimE{0}$
&$m_1-m_0-m_0\log\tfrac{m_1}{m_0}$
&$-\log(1-\alpha)+\iota_{(-\infty,1)}(\alpha)$
\\\hline
$\SimE{1}$
&$m_1\log\tfrac{m_1}{m_0}-m_1+m_0$
&$\exp\alpha-1$
\\\hline
$\SimEp$
&$m_0\tfrac1{p(p-1)}((\tfrac{m_1}{m_0})^p-p(\tfrac{m_1}{m_0}-1)-1)$
&$\frac{(1-(1-p)\alpha)^{p/(p-1)}-1}p+\iota_{(-\infty,1]}((1-p)\alpha)$
\end{tabular}
\caption{Formulas of different local discrepancy integrands and their conjugates. Formulas for $c(m_0,m_1)$ are only valid for $m_0$, $m_1 \geq 0$, otherwise set $c(m_0,m_1)=\infty$.
Note that due to symmetry we have $[c(1,\cdot)]^\ast=[c(\cdot,1)]^\ast$ for all discrepancies except $\SimTV_{a,b}$, for which $[c^{\tn{TV}}_{a,b}(\cdot,1)]^\ast=[c^{\tn{TV}}_{b,a}(1,\cdot)]^\ast$,
and $\SimEp$,
for which $[c^{E,p}(\cdot,1)]^\ast=[c^{E,1-p}(1,\cdot)]^\ast$.}
\label{tab:discrepancies}
\end{table}

\begin{figure}
\setlength{\unitlength}{\linewidth}
\begin{picture}(1,.35)
	\put(.06,0){\includegraphics[height=.37\unitlength]{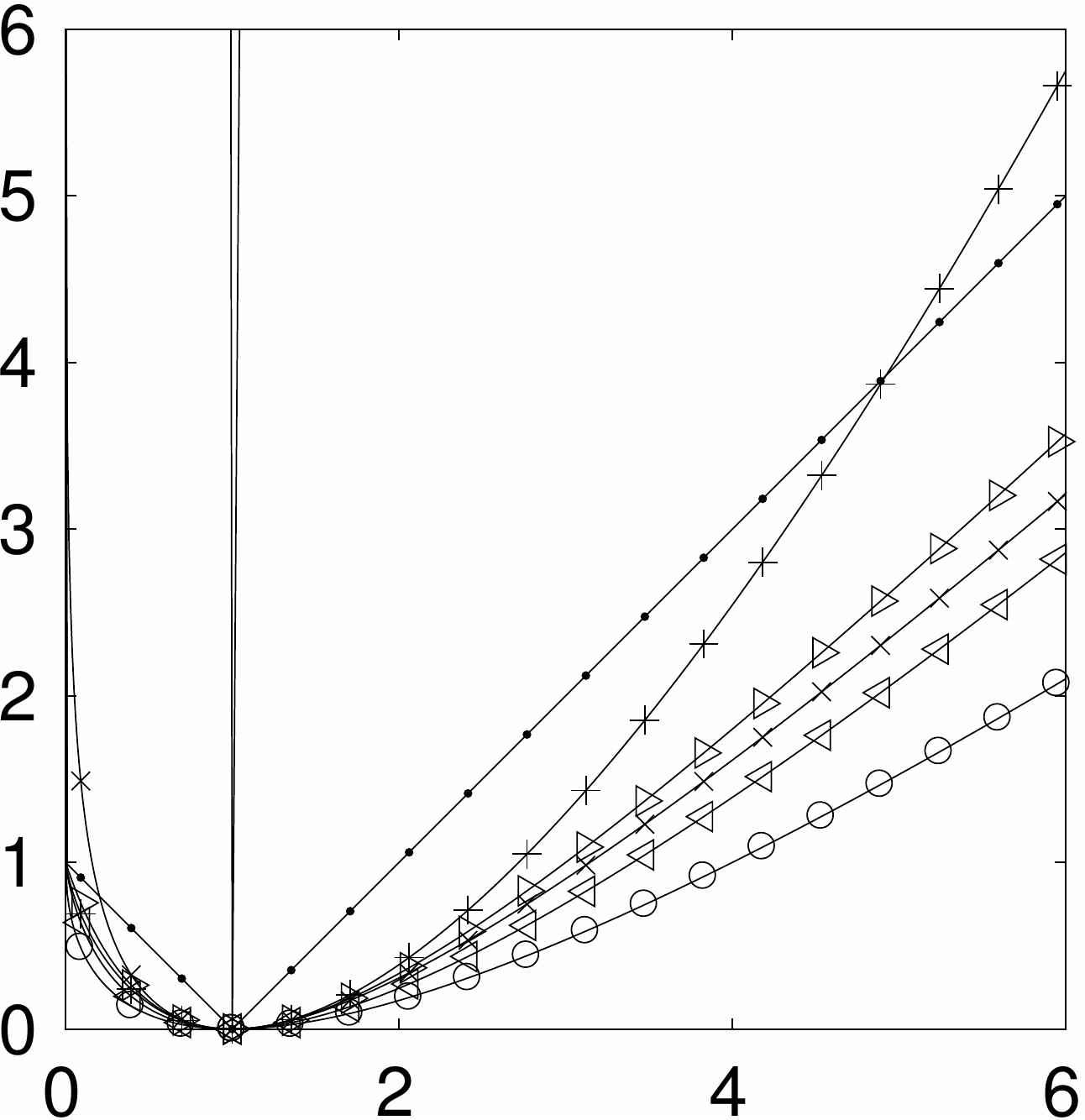}}
	\put(.5,.1){\includegraphics[width=.45\unitlength]{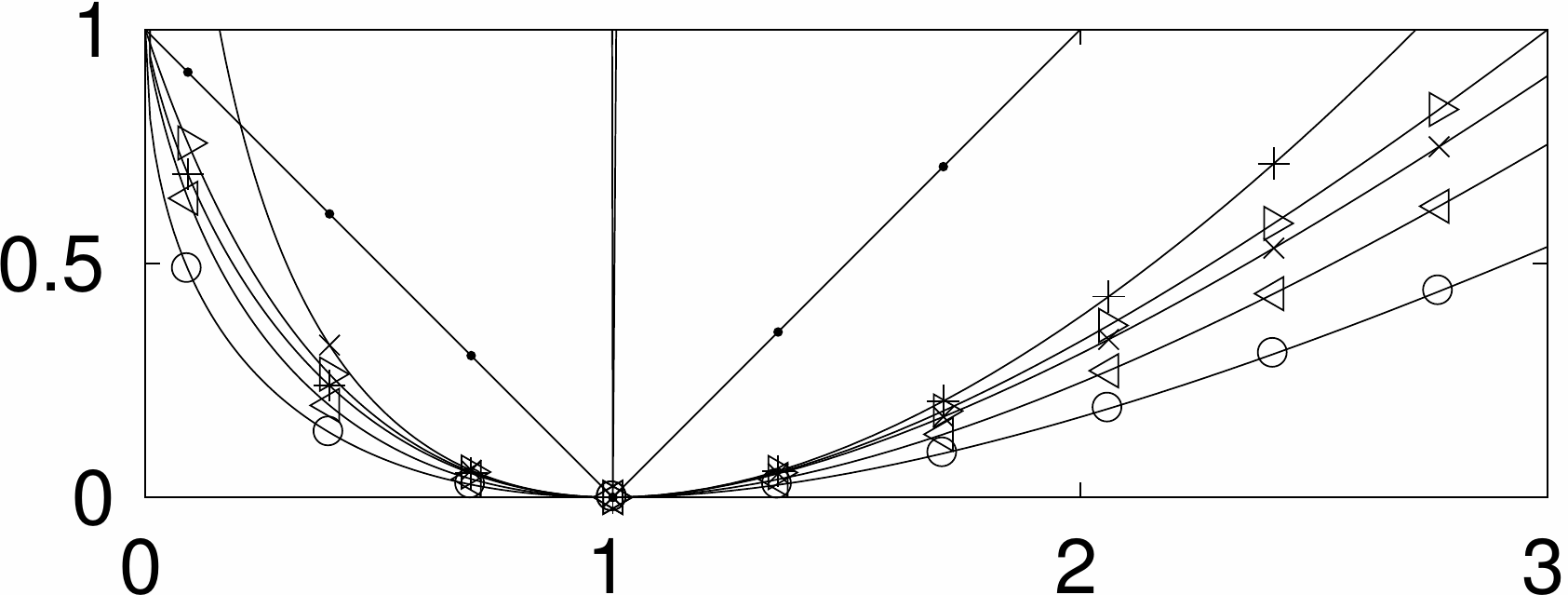}}
	\put(.175,.25){\def\arraystretch{.7}\setlength{\tabcolsep}{0pt}
		\begin{tabular}{cl}
		\small$-$&\small$\SimDisc$\\
		\small$\cdot$&\small$\SimTV$\\
		\small$\circ$&\small$\SimFR$\\
		\small$\triangleleft$&\small$\SimJS$\\
		\small$\triangleright$&\small$\SimChi$\\
		\small$\times$&\small$\SimE{0}$\\
		\small$+$&\small$\SimE{1}$
		\end{tabular}}
\end{picture}
\caption{Graphs of local discrepancy integrands $c$ from Table\,\ref{tab:discrepancies}.
We only plot $c(1,m)$ as a function of $m$, which is sufficient due to the 1-homogeneity of $c$.
The right graph shows a zoom.}
\label{fig:cost}
\end{figure}

\begin{figure}
\setlength{\unitlength}{\linewidth}
\begin{picture}(1,.35)
	\put(.06,0){\includegraphics[height=.37\unitlength]{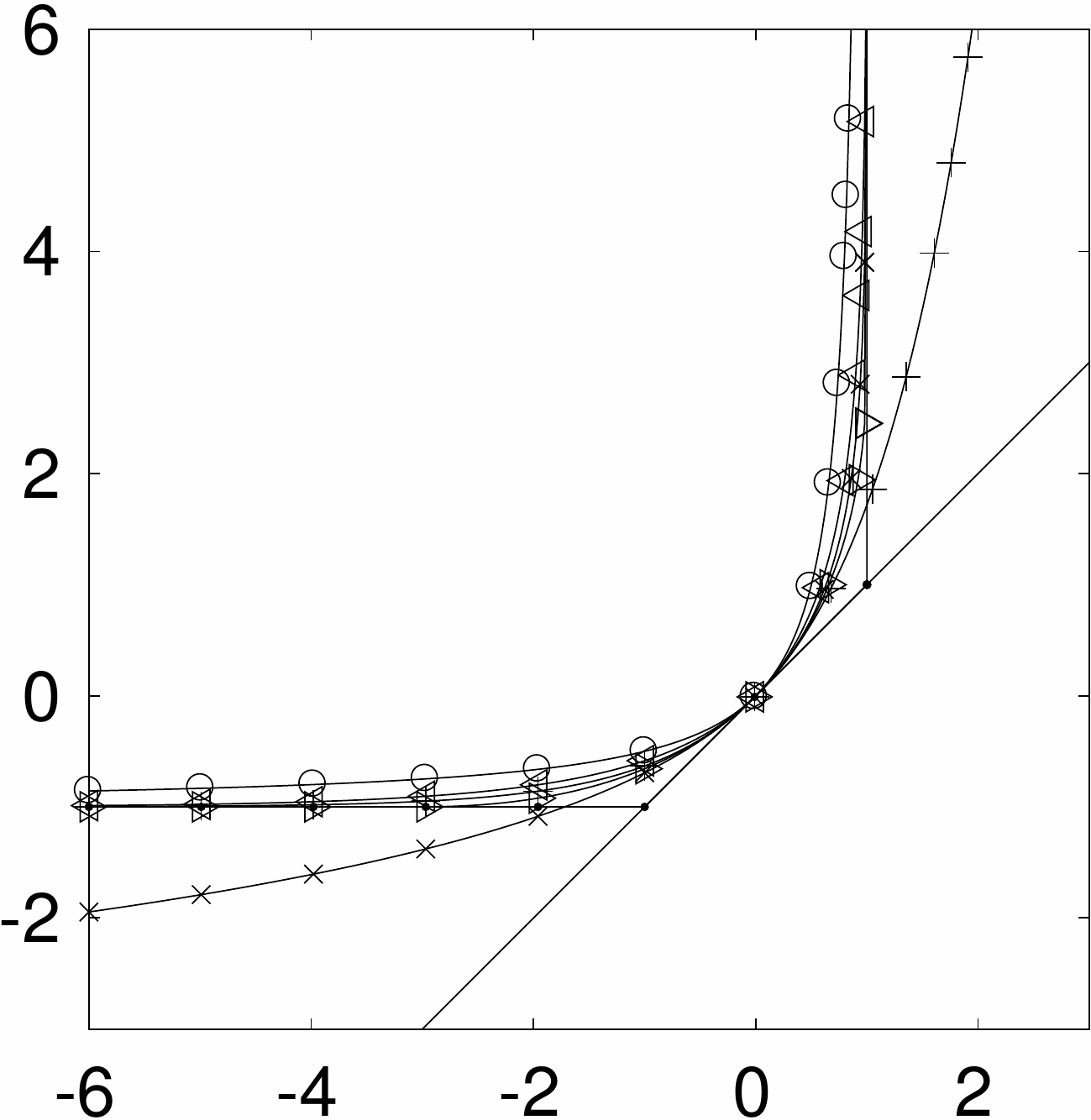}}
	\put(.56,0){\includegraphics[height=.37\unitlength]{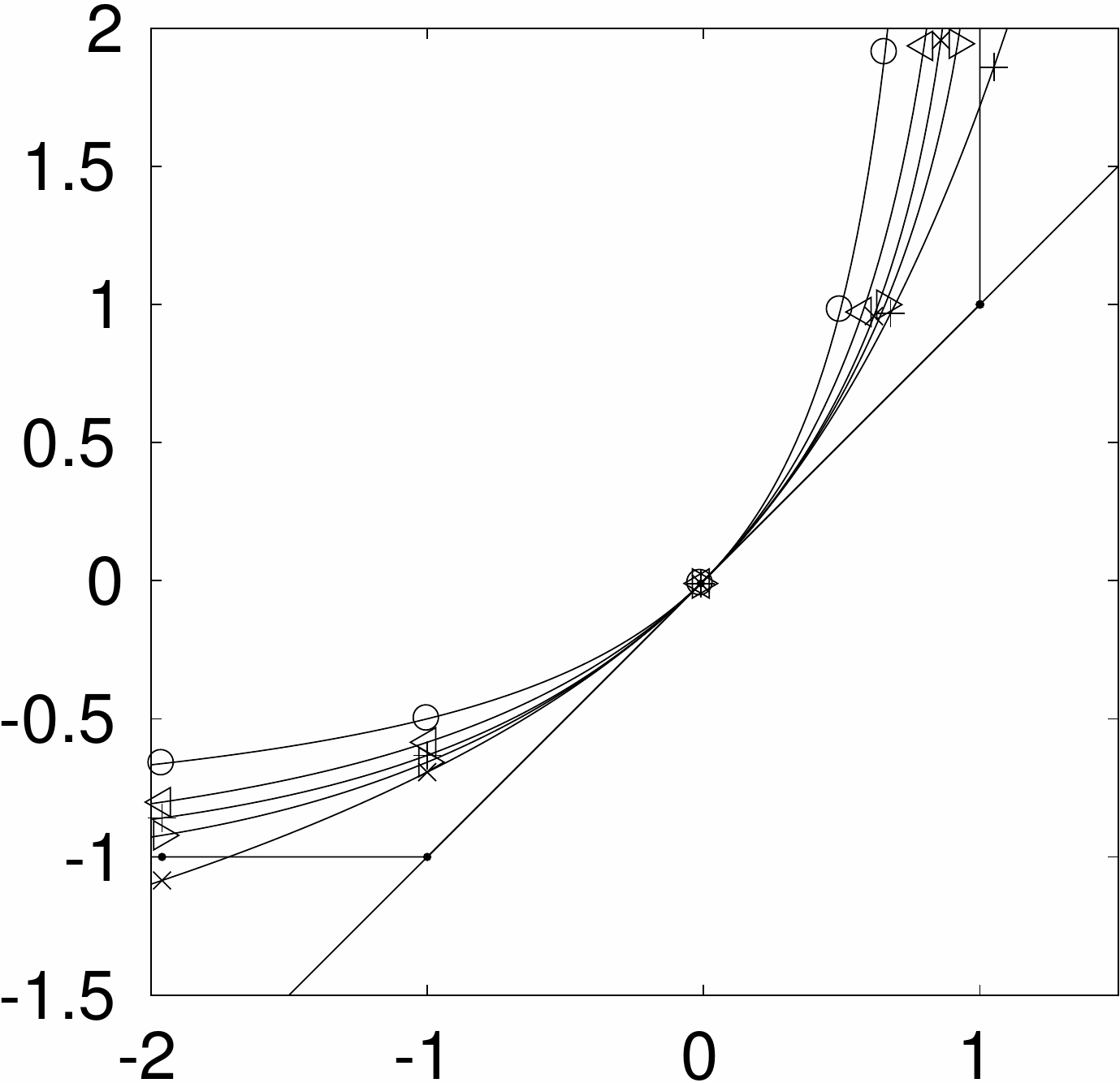}}
	\put(.175,.25){
		\def\arraystretch{.7}\setlength{\tabcolsep}{0pt}
		\begin{tabular}{cl}
		\small$-$&\small$\SimDisc$\\
		\small$\cdot$&\small$\SimTV$\\
		\small$\circ$&\small$\SimFR$\\
		\small$\triangleleft$&\small$\SimJS$\\
		\small$\triangleright$&\small$\SimChi$\\
		\small$\times$&\small$\SimE{0}$\\
		\small$+$&\small$\SimE{1}$
		\end{tabular}}
\end{picture}
\caption{Legendre--Fenchel conjugates $[c(1,\cdot)]^\ast$ of the local discrepancy integrands from Table\,\ref{tab:discrepancies}.
The right graph shows a zoom.}
\label{fig:costConj}
\end{figure}

\subsection{Unbalanced optimal transport}\label{sec:unbalancedExamples}
Of course, the above discrepancies can be combined in Definition\,\ref{def:SandwichProblem} in various possible ways.
We shall have a look at some exemplary special cases.

\paragraph{Combination with total variation.}
Let us first consider cases with $\SimTV$ or a generalization thereof.
For $a,b\in(0,\infty]$ let $\SimTV_{a,b}$ be induced by
\begin{equation*}
	c^{TV}_{a,b}(m_0,m_1)=\begin{cases}
		a|m_1-m_0|&\text{if }m_1\geq m_0 \geq 0,\\
		b|m_1-m_0|&\text{if }m_0 \geq m_1 \geq 0,\\
		+\infty & \text{else,}
		\end{cases}
\end{equation*}
where $a$ is the cost for increasing and $b$ for decreasing the mass by one mass unit.
Note $\SimTV=\SimTV_{1,1}$ and $\SimDisc=\SimTV_{\infty,\infty}$.
By Corollary\,\ref{cor:reductionInfimalConv} we have
\begin{equation*}
\SimTV_{a,b}\diamond W_1\diamond \SimTV_{c,d}\diamond W_1\diamond \SimTV_{e,f}
=\SimTV_{a\wedge c,b\wedge d}\diamond W_1\diamond \SimTV_{e,f}
=\SimTV_{a,b}\diamond W_1\diamond \SimTV_{c\wedge e,d\wedge f}\,,
\end{equation*}
where $\wedge$ indicates the minimum.
Taking $a=b=e=f=\infty$ this in particular implies
\begin{equation*}
W_1\diamond \SimTV_{c,d}\diamond W_1
=\SimTV_{c,d}\diamond W_1
=W_1\diamond \SimTV_{c,d}
=\SimTV_{c,d} \diamond W_1\diamond \SimTV_{c,d}\,,
\end{equation*}
that is, $W_1$ and the generalized total variation metric commute.
This can now be used to imply the fully reduced form
\begin{equation*}
\SimTV_{a,b}\diamond W_1\diamond \SimTV_{c,d}\diamond W_1\diamond \SimTV_{e,f}
=\SimTV_{a\wedge c\wedge e,b\wedge d\wedge f}\diamond W_1
=W_1\diamond \SimTV_{a\wedge c\wedge e,b\wedge d\wedge f}\,.
\end{equation*}
Furthermore, using Corollary \ref{cor:equivalenceStatic} and Proposition\,\ref{prop:modelReduction} it is straightforward to show that a corresponding predual formulation is given by
\begin{equation*}
W_{\SimDisc,\SimTV_{c,d},\SimDisc}(\rho_0,\rho_1)
=\sup_{\substack{\alpha,\beta\in\Lip(\Omega):\\\alpha+\beta\leq0,\,-\alpha\leq c,\,-\beta\leq d}}\int_\Omega\alpha\,\d\rho_0+\int_\Omega\beta\,\d\rho_1
=\sup_{\substack{\alpha\in\Lip(\Omega):\\-c \leq \alpha \leq d}}\int_\Omega\alpha\,\d(\rho_0-\rho_1)\,.
\end{equation*}
According to Proposition \ref{prop:existencePrimal} optimizers exist as long as $c<\infty$ (or $\measnrm{\rho_0}\geq\measnrm{\rho_1}$) and $d<\infty$ (or $\measnrm{\rho_0}\leq\measnrm{\rho_1}$).
Applications of optimal transport with `reservoir bins' to account for superfluous mass \cite{RubnerEMD-IJCV2000,PeleECCV2008}, with truncated cost functions  \cite{Pele2009}, as well as the optimal partial transport problem \cite{Caffarelli-McCann-FreeBoundariesOT-2010} are (or can be interpreted as) combinations of optimal transport with total variation.

\newcommand{\WKR}{W^{\tn{KR}}}
\paragraph{Kantorovich--Rubinstein norm.}
In \cite{LellmannKantorovichRubinstein2014} the authors propose to use a variation of the Kan\-toro\-vich--Rubinstein norm as image discrepancy.
For $\lambda_1,\lambda_2\geq0$ and two given images (in our case measures) $\rho_0,\rho_1$ they define
\begin{equation*}
\WKR(\rho_0,\rho_1)=\lambda_2\sup_{\alpha\in\Lip(\Omega),\,|\alpha|\leq\lambda_1/\lambda_2}\int_\Omega\alpha\,\d(\rho_0-\rho_1)\,.
\end{equation*}
Comparing with the above, we obtain
\begin{align*}
\WKR(\rho_0,\rho_1)/\lambda_2
=\SimTV_{\frac{\lambda_1}{\lambda_2},\frac{\lambda_1}{\lambda_2}} \diamond W_1
=\tfrac{\lambda_1}{\lambda_2}\SimTV\diamond W_1\,.
\end{align*}
which is a special case of the combination with total variation.
Note that the authors provide several primal and dual formulations (however, none of the simple explicit form \eqref{eq:SandwichProblem} or of the form \eqref{eq:GeneralizedPrimalProblem}).

\paragraph{Hellinger--Kantorovich distance.}
Liero et al.\ \cite{LieroMielkeSavare-HellingerKantorovich-2015a} propose a class of optimal entropy-trans\-port problems, for instance
\begin{equation*}
	\min_{\pi\in\measp(\Omega\times\Omega)} \SimE{1}(\rho_0,{\Proj_0}_\sharp \pi) + \SimE{1}(\rho_1,{\Proj_1}_\sharp \pi) + \int_{\Omega\times\Omega}c(x_1,x_2)\,\d\pi(x_1,x_2)\,.
\end{equation*}
This obviously has the form $W(\rho_0,\rho_1)=\SimE{1}\diamond W_c\diamond \SimE{0}$ for $W_c$ the Wasserstein distance with cost $c$
(note that $\SimE{0}$ is the symmetric discrepancy to $\SimE{1}$, obtained by swapping both arguments).
For $c(x,y)=|x-y|^2$ this yields the so called (squared) Gaussian Hellinger--Kantorovich distance, for $c(x,y) = -\log(\cos^2(|x-y|))$ if $|x-y|<\pi/2$ and $+\infty$ else one obtains the (squared) Hellinger--Kantorovich distance or Wasserstein--Fisher--Rao distance, as also introduced in \cite{KMV-OTFisherRao-2015,ChizatDynamicStatic2018}. It can be interpreted as a Riemannian infimal convolution between the Wasserstein-2 distance and the Hellinger/Fisher--Rao distance.

In our setting we choose the cost $c(x,y)=|x-y|$ for computational efficiency, yielding the special case $W_{\SimE{1},\SimDisc,\SimE{0}}$ of our class of functionals.
The corresponding predual formulation reads
\begin{equation*}
W_{\SimE{1},\SimDisc,\SimE{0}}(\rho_0,\rho_1)
	=\sup_{\alpha \in\Lip(\Omega)}
	\int_\Omega (1-\exp(-\alpha))\,\d\rho_0+\int_\Omega (1-\exp(\alpha))\,\d\rho_1\,.
\end{equation*}

\paragraph{Combination with squared Hellinger distance.}
In Section \ref{sec:experiments} we also consider a combination of $W_1$ and the squared Hellinger distance, namely $W_1 \diamond \lambda\,\SimFR \diamond W_1$ for positive weights $\lambda$.
The $\SimFR$ term is well-suited for balancing gradual local fluctuations of intensity, since the corresponding integrand $c(1,\cdot)$ is flat around 1 (cf.~Figure \ref{fig:cost}).
Moreover, when $\SimFR$ is put between two $W_1$ terms, it allows to shrink mass that is transported in the second $W_1$ term as well as to increase mass that is transported in the first $W_1$ term. Thus it can symmetrically account for discrepancies in both directions (similar properties hold for $\SimJS$ and $\SimChi$).

%% file: 04-algorithm.tex
\section{Discretization and Optimization}\label{sec:discretization}

In all examples we shall consider, the corresponding formulation \eqref{eq:GeneralizedPrimalProblem} is well-posed and admits optimizers
(as can easily be checked via Proposition \ref{prop:existencePrimal}).
In fact, nonexistence of optimizers usually only happens if no mass growth or shrinkage is allowed despite $\measnrm{\rho_0}\neq\measnrm{\rho_1}$ for the two input measures
or if one of the measures is zero.
Thus we will base our implementation on formulation \eqref{eq:GeneralizedPrimalProblem}, exploiting the locality of the constraints on $\alpha$ and $\beta$.
In fact, a numerical implementation of \eqref{eq:SandwichProblem} instead would involve solving two $W_1$ problems
which each have complexity similar to \eqref{eq:GeneralizedPrimalProblem} (using the standard formulation \eqref{eqn:W1predual} or Beckmann's formulation from Remark\,\ref{rem:flowFormulation}).
For ease of exposition and with regard to imaging applications we shall here restrict ourselves to the domain $\Omega=[0,1]^2$.
The discretization in higher dimensions or in non-Euclidean spaces can be performed in a similar way, but might require the use of unstructured grids.

\subsection{Variable splitting}\label{sec:variableSplitting}
To achieve a simple proximal point type algorithm, we aim to separate the different constraints on $\alpha$ and $\beta$ from each other.
First note that $\alpha\in\Lip(\Omega)$ is equivalent to $\alpha\in W^{1,\infty}(\mathrm{int}\Omega)$ with $|\nabla\alpha|\leq1$ almost everywhere.
Thus, the indicator function of the set $\Lip(\Omega)$ is given by
\begin{equation*}
\iota_{\Lip(\Omega)}(\alpha)
=\int_\Omega\sup_{\phi\in\R^2}\nabla\alpha\cdot\phi-|\phi|\,\d x
=\sup_{\phi\in L^1(\mathrm{int}\Omega)^2}\int_\Omega\nabla\alpha\cdot\phi-|\phi|\,\d x\,,
\end{equation*}
where the second equality follows by \cite[Thm.\,VII-7]{CaVa77}.
Thus, \eqref{eq:GeneralizedPrimalProblem} can be transformed into the saddle point problem
\begin{multline*}
W_{\h,\g,B}(\rho_0,\rho_1)
=\sup_{\alpha,\beta\in C(\Omega)}\inf_{\phi,\psi\in L^1(\mathrm{int}\Omega)^2}\int_\Omega\h(\alpha)\,\d\rho_0+\int_\Omega\g(\beta)\,\d\rho_1-\iota_{B}(\alpha,\beta)\\
+\int_\Omega|\phi|+|\psi|\,\d x-\int_\Omega\nabla\alpha\cdot\phi+\nabla\beta\cdot\psi\,\d x\,.
\end{multline*}
The vector fields $\phi$ and $\psi$ indicate the material flow and may themselves be of interest for applications, see Remark \ref{rem:flowFormulation}.
Now splitting the variables $\alpha$ and $\beta$ up into two representatives allows to separate the Lipschitz constraint from the set $B$,
\begin{multline}\label{eqn:saddlePointProblem}
W_{\h,\g,B}(\rho_0,\rho_1)
=\sup_{\alpha,\beta,\eta,\theta\in C(\Omega)}\inf_{\substack{\phi,\psi\in L^1(\mathrm{int}\Omega)^2\\\rho_0'',\rho_1''\in\meas(\Omega)}}
\int_\Omega\h(\alpha)\,\d\rho_0+\int_\Omega\g(\beta)\,\d\rho_1-\iota_{B}(\eta,\theta)\\
+\int_\Omega\eta-\alpha\,\d\rho_0''+\int_\Omega\theta-\beta\,\d\rho_1''
+\int_\Omega|\phi|+|\psi|\,\d x-\int_\Omega\nabla\alpha\cdot\phi+\nabla\beta\cdot\psi\,\d x\,.
\end{multline}

\subsection{Discretization}
We choose a pixel-based regular grid of width $\Delta x=\frac1N$.
Discrete quantities will be denoted by a hat.
The functions $\alpha$, $\beta$, $\eta$, and $\theta$ will be represented by their values at the positions $x_{ij}=(i\Delta x,j\Delta x)$, $i,j=1,\ldots,N$,
that is, $$\hat\alpha=(\alpha_{ij})_{i,j=1,\ldots,N}\quad\text{with}\quad\alpha_{ij}=\alpha(x_{ij})$$ and analogously for the other variables.
The gradient operator is usually implemented via forward finite differences,
\begin{equation}\label{eqn:forwardDifference}
(\hat\nabla\hat\alpha)_{ij}=\tfrac1{\Delta x}(\alpha_{i+1,j}-\alpha_{ij},\alpha_{i,j+1}-\alpha_{ij})^T\,.
\end{equation}
To reduce anisotropic effects due to the discretization (forward differences behave differently along both diagonals, see Section\,\ref{sec:flow}), we alternatively also consider a discrete representation as the collection of two gradient approximations,
the first one via forward differences, the second one via a mixture of forward and backward differences.
\begin{equation}\label{eqn:mixedDifference}
(\hat\nabla\hat\alpha)_{ij}=\tfrac1{\Delta x}\left[(\alpha_{i+1,j}-\alpha_{ij},\alpha_{i,j+1}-\alpha_{ij})^T,(\alpha_{i+1,j}-\alpha_{ij},\alpha_{ij}-\alpha_{i,j-1})^T\right]\,.
\end{equation}
Just for comparison purposes we introduce a third variant, where the gradient information is represented as a collection of directional derivatives, for instance with eight directions,
\begin{multline}\label{eqn:directionalDifference}
(\hat\nabla\hat\alpha)_{ij}
=\tfrac1{\Delta x}\big(
\alpha_{i+1,j}-\alpha_{ij},
\tfrac{\alpha_{i+2,j+1}-\alpha_{ij}}{\sqrt5},
\tfrac{\alpha_{i+1,j+1}-\alpha_{ij}}{\sqrt2},
\tfrac{\alpha_{i+1,j+2}-\alpha_{ij}}{\sqrt5},\\
\alpha_{i,j+1}-\alpha_{ij},
\tfrac{\alpha_{i-1,j+2}-\alpha_{ij}}{\sqrt5},
\tfrac{\alpha_{i-1,j+1}-\alpha_{ij}}{\sqrt2},
\tfrac{\alpha_{i-2,j+1}-\alpha_{ij}}{\sqrt5}\big)\,.
\end{multline}
Note that here we used a horizontal vector so that in all discretization cases the (directional) gradients are contained in the columns of $(\hat\nabla\hat\alpha)_{ij}$.
The discretization \eqref{eqn:directionalDifference} corresponds to a network flow problem where the plane is approximated by a discrete grid graph in which each vertex is connected to its 8-neighbourhood (see e.\,g.~\cite{TreeEMD2007}).
The vector fields $\phi$ and $\psi$ are discretely represented as the duals to the discrete gradient vectors, that is,
$\hat\phi=(\phi_{ij})_{i,j=1,\ldots,N}$ and $\hat\psi=(\psi_{ij})_{i,j=1,\ldots,N}$, where $\phi_{ij}$ and $\psi_{ij}$ have the same shape and dimension as $(\hat\nabla\hat\alpha)_{ij}$.
Finally, the measures $\rho_0$, $\rho_1$, $\rho_0''$, and $\rho_1''$ shall be discretized as a sum of Dirac measures, that is,
$$\rho_k=\sum_{i,j=1}^N(\rho_k)_{ij}\delta_{x_{ij}}\,,
\qquad
\hat\rho_k=((\rho_k)_{ij})_{i,j=1,\ldots,N}\,,$$
and analogously for the other measures.
The discrete version of saddle point problem \eqref{eqn:saddlePointProblem} is then given as
\begin{equation}\label{eqn:saddlePointDiscrete}
\max_{\hat\alpha,\hat\beta,\hat\eta,\hat\theta}\min_{\hat\phi,\hat\psi,\hat\rho_0'',\hat\rho_1''}
\left\langle\hat\eta-\hat\alpha,\hat\rho_0''\right\rangle
+\left\langle\hat\theta-\hat\beta,\hat\rho_1''\right\rangle
-\left\langle\hat\nabla\hat\alpha,\hat\phi\right\rangle
-\left\langle\hat\nabla\hat\beta,\hat\psi\right\rangle
-F^{\hat\rho_1}[\hat\alpha,\hat\beta,\hat\eta,\hat\theta]
+G[\hat\phi,\hat\psi,\hat\rho_0'',\hat\rho_1'']
\end{equation}
\begin{align*}
\text{for}\quad
F^{\hat\rho_1}[\hat\alpha,\hat\beta,\hat\eta,\hat\theta]
&=\textstyle-\Delta x^2\sum_{i,j=1}^N\left(\h(\alpha_{ij})(\rho_0)_{ij}+\g(\beta_{ij})(\rho_1)_{ij}-\iota_B(\eta_{ij},\theta_{ij})\right)\,,\\
G[\hat\phi,\hat\psi,\hat\rho_0'',\hat\rho_1'']
&=\textstyle\Delta x^2\sum_{i,j=1}^N\sum_{k=1}^M\left(|\phi_{ij}^k|+|\psi_{ij}^k|\right)\,,\\
\left\langle\hat a,\hat b\right\rangle
&=\textstyle\Delta x^2\sum_{i,j=1}^Nb_{ij}\cdot a_{ij}\,,
\end{align*}
where $\phi_{ij}^k$ and $\psi_{ij}^k$ for $k=1,\ldots,M$ are the columns of $\phi_{ij}$ and $\psi_{ij}$, respectively.

\subsection{Proximal operators}
To apply a standard primal-dual method we require the corresponding proximal operators.
In detail, for a convex function $f$ and a stepsize $\tau>0$ we set $\prox_{\tau f}(x)=\argmin_{\tilde x}\frac12\|\tilde x-x\|_2^2+\tau f(\tilde x)$,
where in the discrete vector-valued case $\|\hat x\|_2^2=\langle\hat x,\hat x\rangle$.

The proximal operator of $G$ is straightforward to compute,
\begin{multline*}
\big(\prox_{\tau G}(\hat\phi,\hat\psi,\hat\rho_0'',\hat\rho_1'')\big)_{ij}\\
=\big([s_\tau(|\phi_{ij}^1|)\phi_{ij}^1,\ldots,s_\tau(|\phi_{ij}^M|)\phi_{ij}^M],[s_\tau(|\psi_{ij}^1|)\psi_{ij}^1,\ldots,s_\tau(|\psi_{ij}^M|)\psi_{ij}^M],(\rho_0'')_{ij},(\rho_1'')_{ij}\big)\\
\text{with }s_\tau(t)=\max(1-\tfrac\tau t,0)\,.
\end{multline*}

The proximal map for $F^{\hat\rho_1}$ can in general not be given by an explicit formula, hence we will usually solve it via a Newton iteration.
Note that the proximal map can be computed separately for each component,
\begin{equation*}
\big(\prox_{\sigma F^{\hat\rho_1}}(\hat\alpha,\hat\beta,\hat\eta,\hat\theta)\big)_{ij}
=\big(\prox_{(\sigma\Delta x^2(\rho_0)_{ij})(-\h)}(\alpha_{ij}),\prox_{(\sigma\Delta x^2(\rho_1)_{ij})(-\g)}(\beta_{ij}),\prox_{\sigma\Delta x^2\iota_B}(\eta_{ij},\theta_{ij})\big)\,.
\end{equation*}
To find $\prox_{\sigma(-\h)}(\tilde\alpha)$ (and analogously for $\g$) we minimize $f(\alpha):=(\alpha-\tilde\alpha)^2-2\sigma\h(\alpha)$ for $\alpha$ via Newton's method, that is, we iterate
\begin{equation*}
\alpha_{k+1}=\alpha_k-\frac{f'(\alpha_k)}{f''(\alpha_k)}=\alpha_k-\frac{\alpha_k-\tilde\alpha-\sigma\h'(\alpha_k)}{1-\sigma\h''(\alpha_k)}\,,\qquad \alpha_0=\tilde\alpha\,.
\end{equation*}

\begin{lemma}
Let $c$ be the integrand of any local discrepancy from Table\,\ref{tab:discrepancies} except $\SimTV$, $\SimTV_{a,b}$, or $\SimEp$ with $p\notin[0,1]$, and let $\h(\alpha)=-[c(1,\cdot)]^\ast(-\alpha)$.
Given $\tilde\alpha\in\R$ with $\h(\tilde\alpha)>-\infty$, the function $f$ has a unique minimizer $\alpha^*\geq\tilde\alpha$,
and $\alpha_k\to \alpha^*$ monotonically.
\end{lemma}
\begin{proof}
The uniqueness of the minimizer follows from the strict convexity and coercivity of $f$.
Furthermore, $\alpha^*\geq\tilde\alpha$ since $\h$ is increasing.
Finally, in all considered cases, $\h'$ turns out to be convex below $\alpha^*$,
thus $f'$ is concave on $[\tilde\alpha,\alpha^*]=[\alpha_0,\alpha^*]$ and the Newton iteration converges monotonically.
\end{proof}

\noindent For the $\h$ deriving from $\SimTV_{a,b}$ we can instead find the explicit formula
\begin{equation*}
\prox_{\sigma(-\h)}(\tilde\alpha)=\min(\max(-b,\tilde\alpha+\sigma),\max(a,\tilde\alpha))\,.
\end{equation*}
Finally, as for $\prox_{\sigma\iota_B}$, we exploit that $B=\{(\alpha,\beta)\in\R^2\,|\,\beta\leq q(\alpha)\}$ is the hypograph of a concave, monotonically decreasing function $q=\gB(-\cdot)$ with $q(0)=0$ so that
\begin{equation*}
\prox_{\sigma\iota_B}(\tilde\alpha,\tilde\beta)=\begin{cases}
(\tilde\alpha,\tilde\beta)&\text{if }\tilde\beta\leq q(\tilde\alpha),\\
(\alpha^*,q(\alpha^*))&\text{else, with }
\alpha^*=\argmin_{\alpha}f(\alpha)\;\text{for }f(\alpha)=\tfrac12(\alpha-\tilde\alpha)^2+\tfrac12(q(\alpha)-\tilde\beta)^2.
\end{cases}
\end{equation*}
Here again, we perform a Newton iteration to find the minimizer $\alpha^*$ of $f$.
To this end we set
\begin{equation*}
\alpha_{k+1}=\alpha_k-\frac{f'(\alpha_k)}{f''(\alpha_k)}=\alpha_k-\frac{\alpha_k-\tilde\alpha+(q(\alpha_k)-\tilde\beta)q'(\alpha_k)}{1+(q'(\alpha_k))^2+(q(\alpha_k)-\tilde\beta)q''(\alpha_k)}\,,
\qquad
\alpha_0=\tilde\alpha\,.
\end{equation*}

\begin{lemma}
	\label{lem:NewtonConv2}
	Let $c$ be the integrand of any local discrepancy from Table\,\ref{tab:discrepancies} except $\SimTV$, $\SimTV_{a,b}$, or $\SimEp$ with $p\notin[0,1]$,
	and let $q(\alpha)=\gB(-\alpha)=-[c(\cdot,1)]^\ast(\alpha)$.
	Given $\tilde\alpha,\tilde\beta\in\R$ with $\tilde\beta>q(\tilde\alpha)>-\infty$, the function $f$ has a unique minimizer $\alpha^*\leq\tilde\alpha$,
	and $\alpha_k\to \alpha^*$ monotonically.
\end{lemma}
\begin{proof}
The uniqueness of the minimizer follows from the convexity and closedness of the set $B$.
Furthermore, $\alpha^*\leq\tilde\alpha$ since $q$ is decreasing and $\tilde\beta>q(\tilde\alpha)$.
Finally, in all considered cases, $f'$ or equivalently $(\tilde\beta-q)q'$ is concave on $[\alpha^*,\tilde\alpha]=[\alpha^*,\alpha_0]$
so that the Newton iteration converges monotonically.
Indeed, in those cases $f'''(\alpha)=3q'(\alpha)q''(\alpha)+(q(\alpha)-\tilde\beta)q'''(\alpha)\geq0$, since all factors are nonpositive for $\alpha\geq\alpha^*$.
\end{proof}

\noindent For the $q$ deriving from $\SimTV_{a,b}$ we can instead find the explicit formula
\begin{equation*}
\prox_{\sigma\iota_B}(\tilde\alpha,\tilde\beta)=\begin{cases}
(\tilde\alpha,\tilde\beta)&\text{if }\tilde\alpha\leq b,\tilde\beta\leq a,\tilde\alpha+\tilde\beta\leq0,\\
(\tilde\alpha,1)&\text{else if }\tilde\alpha<-a,\\
(1,\tilde\beta)&\text{else if }\tilde\beta<-b,\\
\max(-a,\min(b,\frac{\tilde\alpha-\tilde\beta}2))\cdot(1,-1)&\text{else.}
\end{cases}
\end{equation*}
Note that the initial iterate $\alpha_0$ of Newton's method can be improved in various ways;
in our simulations we actually use $\alpha_0=\min\{0,\tilde\alpha\}$ whenever $\tilde\beta\geq\tilde\alpha$
(otherwise we simply swap the roles of $\alpha$ and $\beta$ so that we iterate in $\beta$ instead).
Lemma \ref{lem:NewtonConv2} stays valid with this choice.

\subsection{Primal-dual iteration}
A standard primal-dual ascent-descent method can now be applied to solve \eqref{eqn:saddlePointDiscrete},
for instance the algorithm by Chambolle and Pock \cite{ChPo11},
\begin{align*}
(\hat\alpha,\hat\beta,\hat\eta,\hat\theta)^{n+1}&=\prox_{\sigma F^{\hat\rho_1}}((\hat\alpha,\hat\beta,\hat\eta,\hat\theta)^n+\sigma K(\hat\phi,\hat\psi,\hat\rho_0'',\hat\rho_1'')^{n+\frac12})\\
(\hat\phi,\hat\psi,\hat\rho_0'',\hat\rho_1'')^{n+1}&=\prox_{\tau G}((\hat\phi,\hat\psi,\hat\rho_0'',\hat\rho_1'')^n-\tau K^T(\hat\alpha,\hat\beta,\hat\eta,\hat\theta)^{n+1})\\
(\hat\phi,\hat\psi,\hat\rho_0'',\hat\rho_1'')^{n+\frac32}&=(\hat\phi,\hat\psi,\hat\rho_0'',\hat\rho_1'')^{n+1}+\vartheta((\hat\phi,\hat\psi,\hat\rho_0'',\hat\rho_1'')^{n+1}-(\hat\phi,\hat\psi,\hat\rho_0'',\hat\rho_1'')^{n})
\end{align*}
with overrelaxation parameter $\vartheta=1$, time steps $\sigma,\tau>0$ satisfying $\sigma\tau\|K\|^2<1$, and linear operator
\begin{equation*}
K=\left(\begin{array}{cccc}-\hat\nabla^T&&-I\\&-\hat\nabla^T&&-I\\&&I\\&&&I\end{array}\right)\,.
\end{equation*}

\subsection{Including further terms}
The discrepancy $W_{\Dl,\Dm,\Dr}$ can also be used as a fidelity term in inverse or image processing problems.
Here, typically $\rho_0$ represents a given measure, and we aim to solve a minimization problem of the form
\begin{equation*}
\min_{\rho\in\measp(\Omega)}W_{\Dl,\Dm,\Dr}(\rho_0,\rho)+H(\rho)\,,
\end{equation*}
for instance with $H$ being the total variation seminorm on images.
This will lead to saddle point problems of the form
\begin{multline}
	\label{eqn:FullDiscreteProblem}
\max_{\hat\alpha,\hat\beta,\hat\eta,\hat\theta,\hat\zeta}\min_{\hat\phi,\hat\psi,\hat\rho_0'',\hat\rho_1'',\hat\rho,\hat\xi}
\left\langle\hat\eta-\hat\alpha,\hat\rho_0''\right\rangle
+\left\langle\hat\theta-\hat\beta,\hat\rho_1''\right\rangle
-\left\langle\hat\nabla\hat\alpha,\hat\phi\right\rangle
-\left\langle\hat\nabla\hat\beta,\hat\psi\right\rangle
-F^{\hat\rho}[\hat\alpha,\hat\beta,\hat\eta,\hat\theta]
+G[\hat\phi,\hat\psi,\hat\rho_0'',\hat\rho_1'']\\
+\left\langle\hat\zeta,K_2\hat\rho+K_3\hat\xi\right\rangle
-F_2[\hat\zeta]
+G_2[\hat\rho,\hat\xi]
\end{multline}
(details on $F_2$, $G_2$, $K_2$, and $K_3$ will be given in the next section).
Note that this formulation contains a term $F^{\hat\rho}$ in which the primal and dual variables $\hat\rho$ and $\hat\beta$ are coupled in a nonlinear manner
(if desired, this can be avoided by replacing $\g(\beta)=\inf_mm\beta+\Cr(m,1)$, thereby introducing another variable $m$).
In that case, the primal-dual iteration can simply be complemented by the additional steps
\begin{align*}
\hat\zeta^{n+1}&=\prox_{\sigma F_2}(\hat\zeta^n+\sigma(K_2\hat\rho^{n+\frac12}+K_3\hat\xi^{n+\frac12}))\\
(\hat\rho,\hat\xi)^{n+1}&=\prox_{\tau G_2}((\hat\rho,\hat\xi)^n-\tau(K_2^T\hat\zeta^{n+1}+\g(\hat\beta^{n+1}),K_3^T\hat\zeta^{n+1}))\\
(\hat\rho,\hat\xi)^{n+\frac32}&=(\hat\rho,\hat\xi)^{n+1}+\vartheta((\hat\rho,\hat\xi)^{n+1}-(\hat\rho,\hat\xi)^{n})\,.
\end{align*}

%% file: 05-numerics.tex
\section{Numerical Experiments}\label{sec:experiments}
In this section we provide some numerical examples, comparing different extensions of $W_1$.

\subsection{Discretization}
Figure\ \ref{fig:flowDiscretization} shows results for different discretizations of the gradient in a simple test example.
Here, $\rho_0$ consists of Dirac masses equally distributed on a circle, while $\rho_1$ has all mass concentrated at the centre.
The exact optimal flow is a line measure connecting all Dirac masses to the centre, however, such line measures are difficult to represent on the discrete level.
The employed discretization of the gradient determines how concentrated the flow can be represented.
One can clearly see that simple forward differences\,\eqref{eqn:forwardDifference} produce rather diffuse flows along the south-east diagonal, which is remedied by the mixed differences version\,\eqref{eqn:mixedDifference}.
The discretization\,\eqref{eqn:directionalDifference} via directional derivatives only has eight possible flow directions at its disposal and results in perfectly concentrated flows along those directions, however, it exhibits stronger diffusion in other directions.
Using $W_{\h,\g,B}=W_{\SimDisc,a\SimTV,\SimDisc}$, the mass will be changed rather than transported as soon as the transport distance exceeds $2a$.
However, the numerical flow diffusion induces slightly increased effective transport distances so that the mass change already occurs earlier.
For all three discretizations, the transition from mass transport to mass change happens within the range $a\in[17.5\Delta x,18.5\Delta x]$ (the circle radius is roughly $36$ pixels),
so the effective transport distance is accurate up to the size of a pixel.

Here and in the following, we compute $\phi,\psi:\Omega\to\R^2$ by solving \eqref{eqn:saddlePointProblem} and then use $\psi-\phi$ to visualize the material flow (cf.\ Remark \ref{rem:flowFormulation} and Section~\ref{sec:variableSplitting}).
Moreover, all further weights are given in units of $\Delta x$, or equivalently $\Delta x=1$.

\begin{figure}
\centering
\setlength{\unitlength}{.22\linewidth}%
\setlength{\tabcolsep}{2pt}%
\begin{tabular}{cccc}
\includegraphics[width=\unitlength]{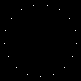}\ \ \ \ &
\includegraphics[width=\unitlength]{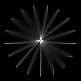}&
\includegraphics[width=\unitlength]{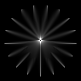}&
\includegraphics[width=\unitlength]{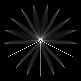}\\
\includegraphics[width=\unitlength]{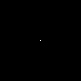}\ \ \ \ &
\includegraphics[width=\unitlength]{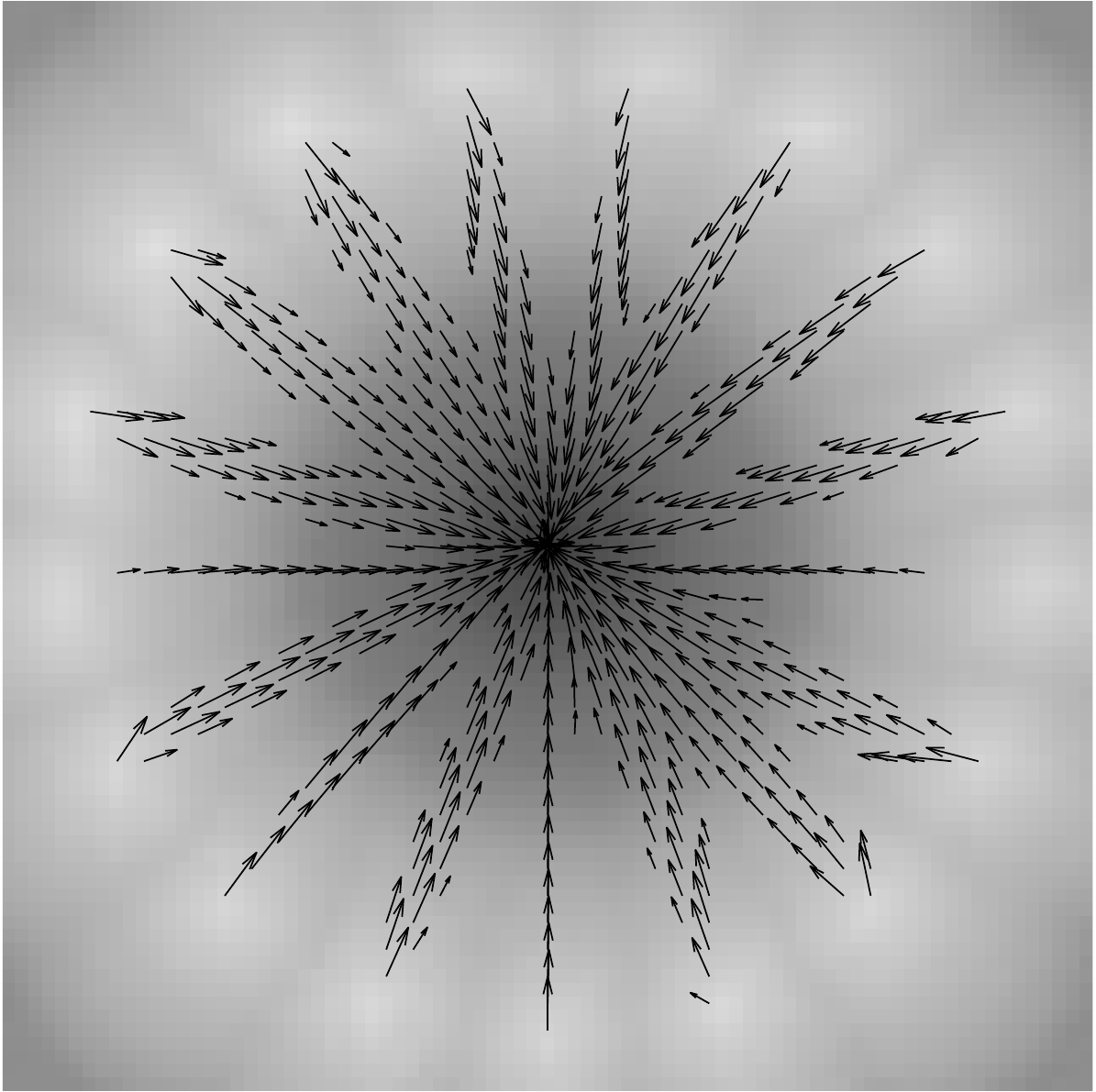}&
\includegraphics[width=\unitlength]{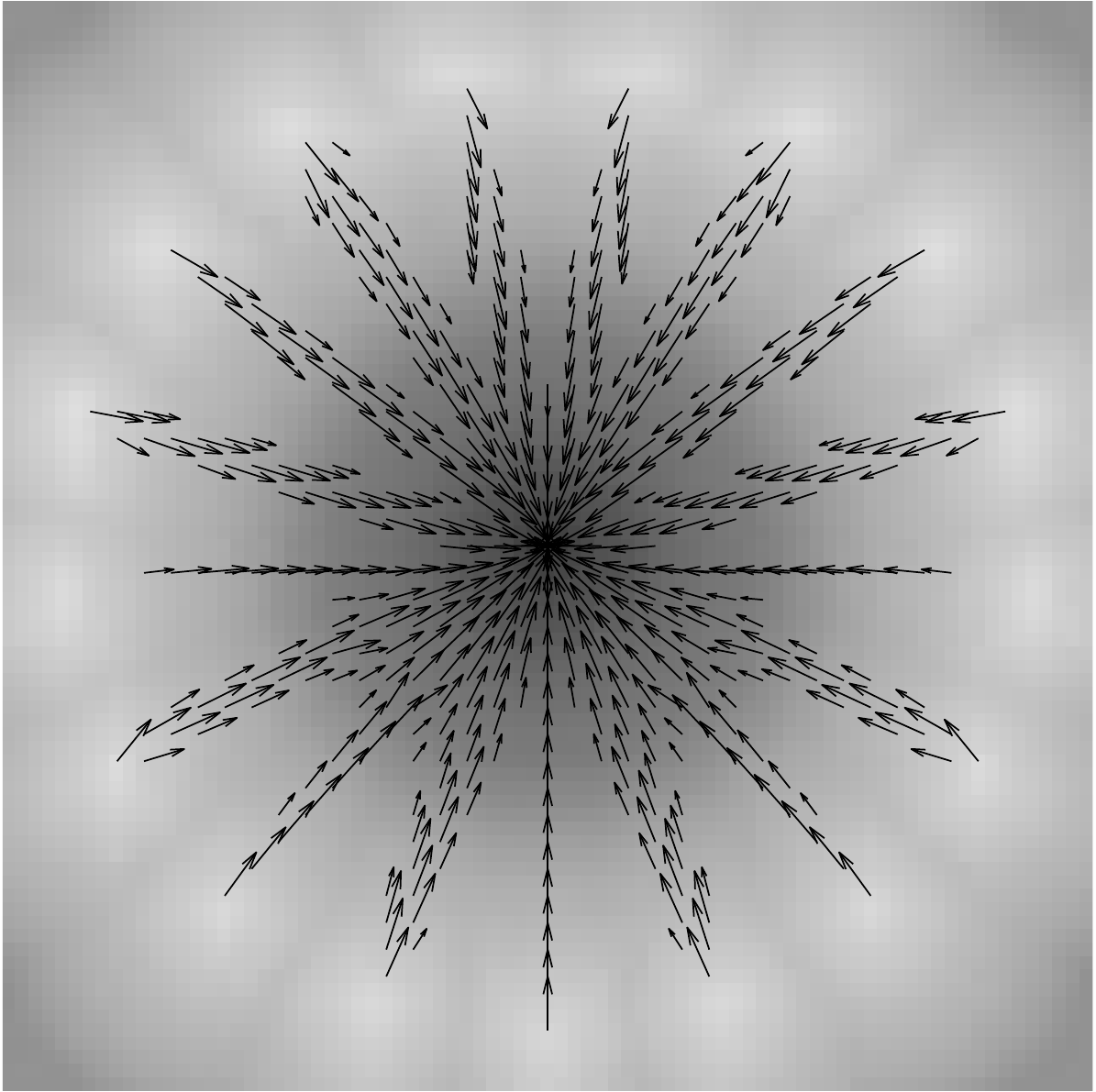}&
\includegraphics[width=\unitlength]{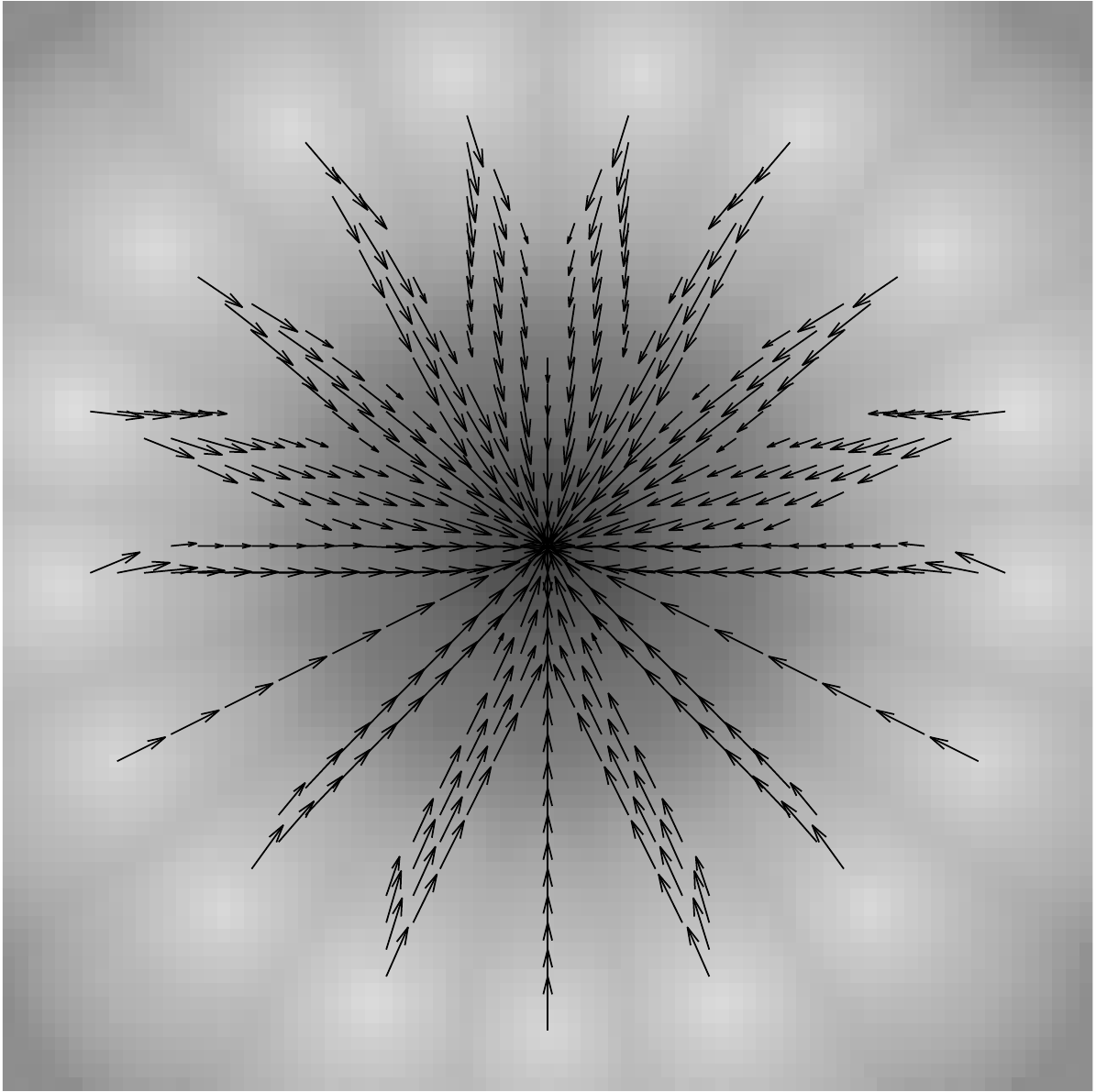}\\
$\rho_0$ \& $\rho_1$\ \ &\eqref{eqn:forwardDifference}&\eqref{eqn:mixedDifference}&\eqref{eqn:directionalDifference}
\end{tabular}
\caption{Computed mass flows $\psi-\phi$ from solving \eqref{eqn:saddlePointProblem}, using different gradient discretizations and $W_{\h,\g,B}=W_{\SimDisc,a\SimTV,\SimDisc}$ with $a$ slightly larger than the circle radius.
The top row shows the flow magnitude, the bottom row the corresponding vector field overlaid over $\alpha$.}
\label{fig:flowDiscretization}
\end{figure}

\subsection{Robustness of unbalanced transport and material flow from videos}\label{sec:flow}

\begin{figure}
	\centering
	\includegraphics{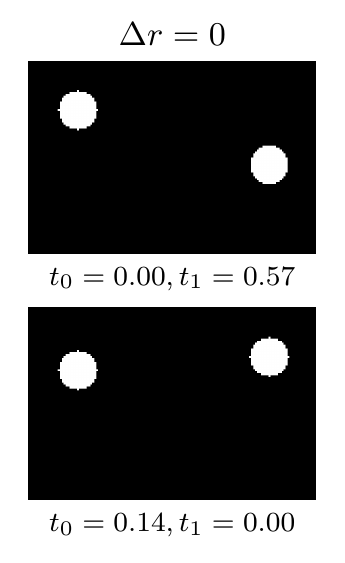}
	\includegraphics{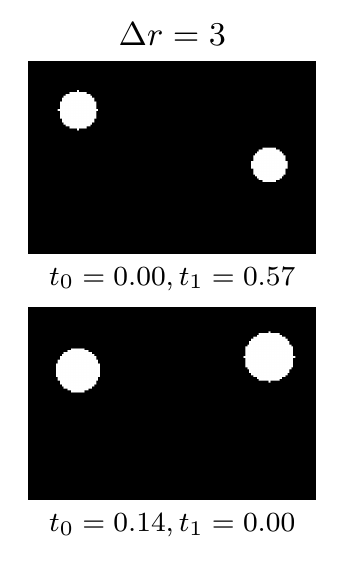}
	\includegraphics{fig_FlowCircles_Flows.pdf}
	\caption{Robust flow reconstruction with unbalanced transport. \textit{First and second column:} two samples from the image families $I_{\Delta r}$. For $\Delta r=0$ all radii are identical, for $\Delta r =3$ the radii change with position and the disks have different sizes. %
	\textit{Third column:} mass flows for transporting $I_{\Delta r=3}(0,0.57)$ onto $I_{\Delta r=3}(0.14,0)$ with $W_1$ (between normalized images) and $W_{\SimDisc,20 \SimFR,\SimDisc}$ (combined flow for both $W_1$ terms is shown). %
	\textit{Fourth column, bottom:} Mass change in the $\SimFR$ term (bright means increase) with two characteristic effects: mass is generated at the larger disks to compensate for the difference, and mass that is to be transported very far, is reduced before transport.%
	}
	\label{fig:FlowCircles}
\end{figure}

Analysing fluorescence microscopy videos to track the spatio-temporal distribution of a certain fluorescent molecule is a common application in biology.
The overall intensity in such videos typically changes over time due to various effects such as photobleaching, changing focus, or material in- and out-flow.
Therefore, standard optimal transport, applied to normalized images, may reconstruct faulty flows and notions of similarity. This can be remedied with unbalanced transport models.

First, we illustrate the advantage of unbalanced models in a synthetic example.
For $[t_0,t_1] \in [0,1]^2$ we create a two parameter family of images $I_0(t_0,t_1)$. $I_0(t_0,t_1)$ is an image of two white disks with radius $r=10$ on black background, one on the left and one on the right side of the image. The parameters $t_0$, $t_1$ control the vertical positions of the two discs.
More generally, we introduce families $I_{\Delta r}$ where the radii of the disks oscillate around $r=10$ by $\Delta r$ while moving. This is a toy model for local fluctuations of intensity.
Some sample images for $\Delta r = 0$ and $\Delta r =  3$ with corresponding flows for $W_1$ and the unbalanced extension $W_{\SimDisc,20 \SimFR,\SimDisc}$ are shown in Figure \ref{fig:FlowCircles}.
For $\Delta r>0$, to compensate for the mass asymmetry between left and right disks, $W_1$ must send mass from left to right, leading to an `unnatural' flow. The unbalanced model can compensate for the difference with the $\SimFR$ term and robustly extracts a flow that represents the movement of the disks.

\newcommand{\TSample}{T_{\tn{Sample}}}

\begin{figure}
	\centering
	\includegraphics[width=0.45\textwidth]{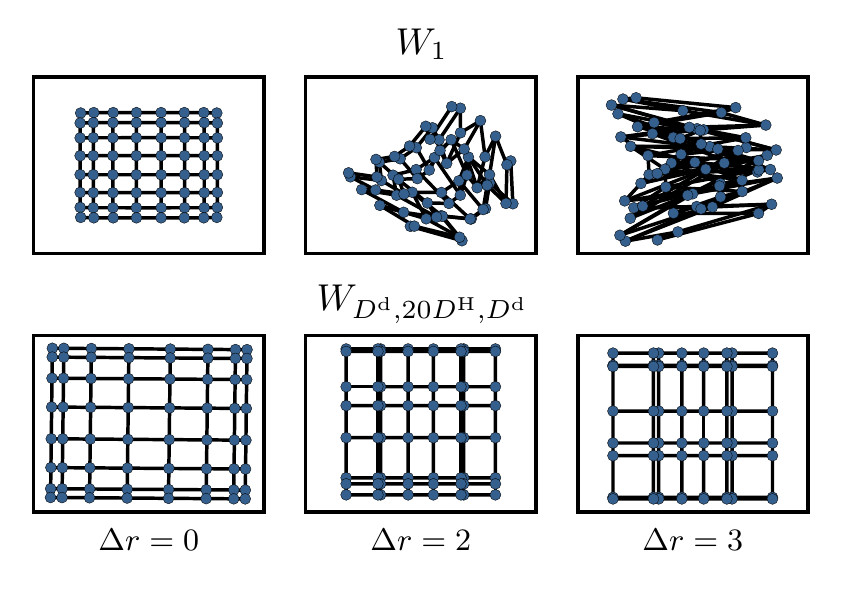}
	\includegraphics[width=0.45\textwidth]{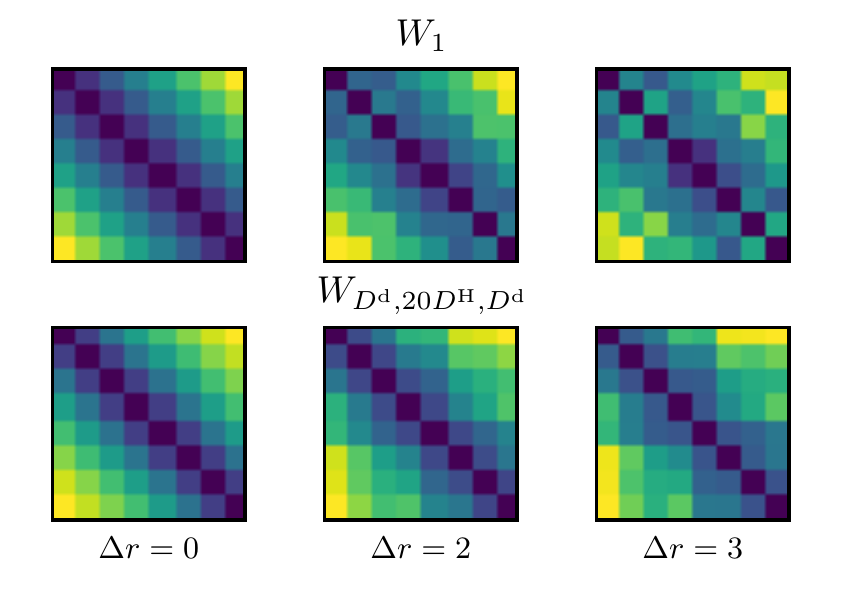}
	\caption{ %
	\textit{Left:} Two-dimensional embedding of images $I_{\Delta r}(t_0,t_1)$ for $(t_0,t_1) \in \TSample^2$ via the graph Laplacian \cite{BelkinNiyogi03} based on the $W_1$ metric matrix and the $W_{\SimDisc,20 \SimFR,\SimDisc}$ discrepancy matrix. Lines represent the Cartesian grid structure of $\TSample^2$. %
	\textit{Right:} Visualization of metric matrices for $t_0=t_1 \in \TSample$, with standard $W_1$ (after normalization to unit mass) and with $W_{\SimDisc, 20 \SimFR, \SimDisc}$.%
	}
	\label{fig:MetricEmbedding}
\end{figure}

Further, let $\TSample=\{0, 0.14, 0.29, 0.43, 0.57, 0.71, 0.86, 1\}$ be eight equidistant samples from $[0,1]$. For the sample images $I_{\Delta r}(t_0,t_1)$ with $(t_0,t_1) \in \TSample^2$ we compute the metric matrix for $W_1$ and the discrepancy matrix with respect to $W_{\SimDisc,20 \SimFR,\SimDisc}$.
We then use the dimensionality reduction scheme of \cite{BelkinNiyogi03} to embed the samples into $\R^2$ (Figure \ref{fig:MetricEmbedding}, left). For $\Delta r = 0$ both discrepancy measures correctly extract the 2-d grid structure of $(t_0,t_1) \in \TSample^2$. As $\Delta r$ increases, the embedding based on $W_1$ becomes increasingly distorted, while the embedding based on the unbalanced model preserves the structure more clearly.
The embeddings with respect to $W_{\SimDisc,20 \SimFR,\SimDisc}$ were also more robust with respect to the choice of the time-scale parameter required in \cite{BelkinNiyogi03}. In Figure~\ref{fig:MetricEmbedding} it was set to five times the nearest neighbour distance.

The metric/discrepancy matrices for the `diagonal' subset $t_0=t_1 \in \TSample$ are shown in Figure \ref{fig:MetricEmbedding}, right. Again, for $\Delta r=0$, the 1-d chain structure is well pronounced in both matrices. As $\Delta r$ increases, the $W_1$ matrices become substantially distorted, while the unbalanced model remains more robust.

An example on real data was illustrated in Figure \ref{fig:transportBio}.
It shows the computed flow of fluorescent molecules between two consecutive frames from a microscopy video of an endothelial cell contact.
With standard $W_1$ transport between normalized images, mass for the expanding structure at the center has to be gathered from throughout the frame, drowning out any local movements.
The flow extracted with $W_{\SimDisc,40 \SimFR, \SimDisc}$ is localized around the growing structure (and other growth/shrinkage events) and clearly indicates the expansion.
It should be noted that part of its mass is still collected from outside the structure, which in this specific example is known to be unbiological. This could only be remedied by a more complex model incorporating additional biological prior knowledge. Nevertheless, unbalanced transport gives a much clearer interpretation than standard $W_1$ transport.

\subsection{Cartoon-texture-noise decomposition}

\begin{figure}
	\centering
	\includegraphics{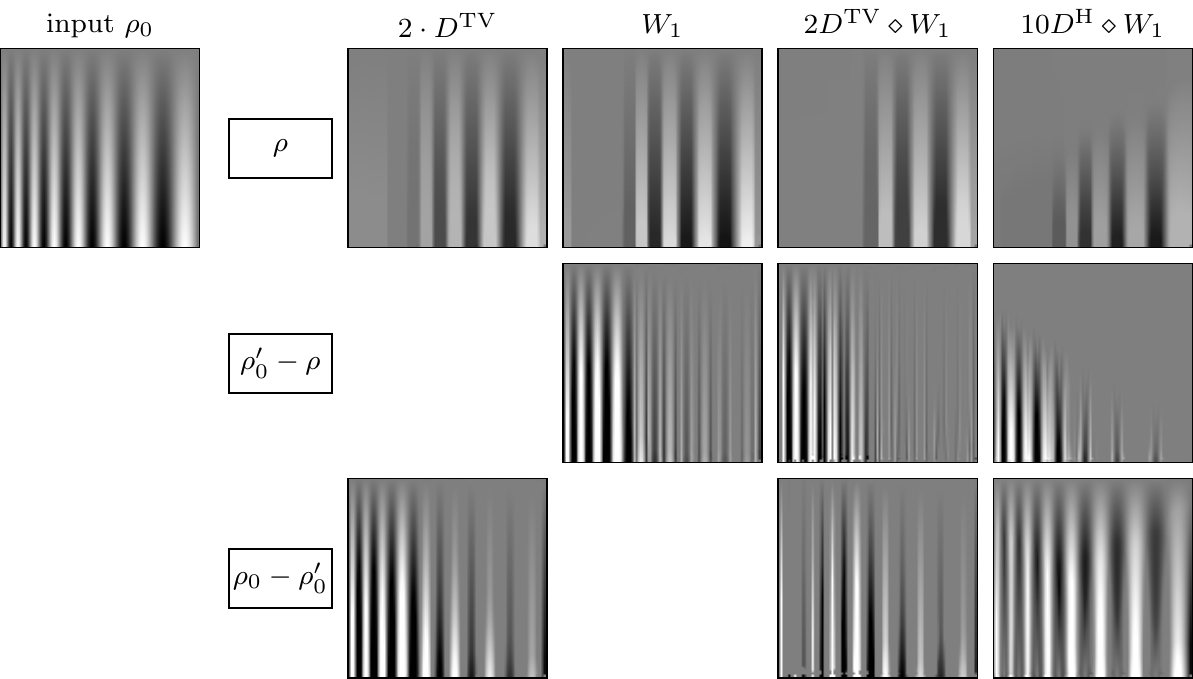}
	\caption{Cartoon, texture and noise parts $(\rho,\rho_0'-\rho,\rho_0'-\rho_0)$ extracted from $\rho_0$ (left) via \eqref{eqn:cartoonTextureDecomp}, using $\lambda_1=5$ and the indicated data fidelity measures. %
	The weights were chosen so as to smooth out roughly the same length scales. %
	For $2 \cdot \SimTV$ and $W_1$ there is no intermediate $\rho_0'$ and we interpret the difference $\rho_0-\rho$ as local (noise) or transport (texture) component respectively.}
	\label{fig:CartTextNoise1}
\end{figure}

Lellmann et al.\ \cite{LellmannKantorovichRubinstein2014} have employed the Kantorovich--Rubinstein discrepency for cartoon-texture decomposition in images.
Expressed in our notation, they considered the problem
\begin{equation}
	\label{eqn:cartoonTextureDecomp}
	\min_{\rho \in \measp(\Omega)} [\Dl \diamond W_1](\rho_0,\rho)+\lambda_1 |\rho|_{\TV}
\end{equation}
for $\Dl=\lambda_0 \SimTV$ (see Section \ref{sec:unbalancedExamples}), in which $\lambda_0$, $\lambda_1 \geq 0$ are two weighting parameters, $\rho_0$ denotes a given image and $|\cdot|_{\TV}$ is the total variation semi-norm.
Here we minimize the same energy, exploring different choices of $\Dl$.
To this end we express the discretized total variation semi-norm and the nonnegativity constraint for $\rho$ as
\begin{equation*}
\max_{\hat\zeta\in\R^{N\times N\times 2}}\langle\hat\zeta,\hat\nabla\hat\rho\rangle-\iota_{|\cdot|\leq\lambda_1}(\hat\zeta)+\iota_{\geq0}(\hat\rho)\,,
\end{equation*}
where $\iota_{\geq0}$ is a more intuitive notation for the indicator function to the set $[0,\infty)^{N\times N}$
(thus $\iota_{\geq0}(\hat\rho)=\infty$ unless all entries of $\hat\rho$ are nonnegative)
and $\iota_{|\cdot|\leq\lambda_1}$ stands for the indicator function to the set $\{\hat\zeta\in\R^{N\times N\times 2}\,|\,|(\zeta_{ij1},\zeta_{ij2})|\leq\lambda_1\,\forall i,j\}$.
Consequently, in the full problem \eqref{eqn:FullDiscreteProblem} we set
\begin{equation*}
F_2[\hat\zeta]=\iota_{|\cdot|\leq\lambda_1}(\hat\zeta)\,,\qquad
G_2[\hat\rho,\hat\xi]=\iota_{\geq0}(\hat\rho)\,,\qquad
K_2=\hat\nabla\,,\qquad
K_3=0\,.
\end{equation*}

\begin{figure}
	\centering
	\includegraphics{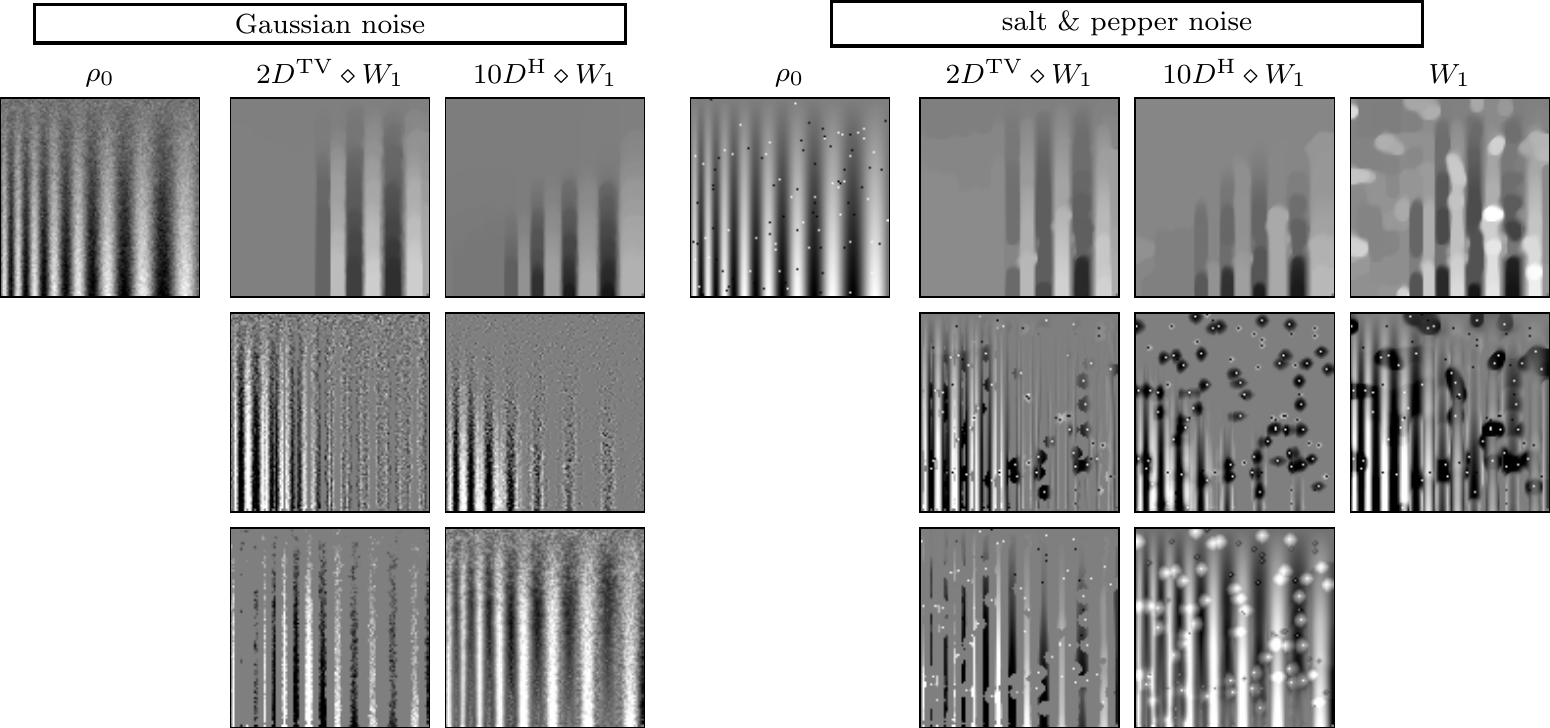}
	\caption{Cartoon, texture and noise parts $(\rho,\rho_0'-\rho,\rho_0'-\rho_0)$ extracted from noisy $\rho_0$ via \eqref{eqn:cartoonTextureDecomp}, using $\lambda_1=5$ and the indicated data fidelity measures. %
	Columns are arranged as in Figure \ref{fig:CartTextNoise1}. %
	\textit{Left:} Original image $\rho_0$ from Figure \ref{fig:CartTextNoise1} with added Gaussian noise with $\sigma=0.05$ (largest amplitude at bottom of image is 1). %
	\textit{Right:} Salt \& pepper noise with corruption probability $0.5\%$ and magnitude 20. %
	For $W_1$ only $\rho$ and the total difference $\rho_0-\rho$ are shown.}
	\label{fig:CartTextNoise2}
\end{figure}

Denote by $\rho_0'$ the implicit intermediate measure of the infimal convolution $\Dl \diamond W_1$ in \eqref{eqn:cartoonTextureDecomp} (cf.~\eqref{eq:SandwichProblem}). Then the effects of $\Dl$ and $W_1$ on the image can be clearly separated. 
$\Dl$ can be thought of as `fixing' single corrupted pixels, hence we may interpret $\rho_0 - \rho_0'$ as the (local) noise component of the image.
$W_1$ can equalize high frequency oscillations with little cost, therefore we interpret $\rho_0'-\rho$ as the texture component.
The remaining smooth image $\rho$ will be called the cartoon part.
The measure $\rho_0'$ can for instance be obtained from the variables in \eqref{eqn:FullDiscreteProblem} via $\rho_0'=\rho_0''-\div \phi$ (see Remark \ref{rem:flowFormulation}).

As an illustrative example we consider the input image $\rho_0$ as shown in Figure~\ref{fig:CartTextNoise1} left, which contains grey value oscillations of different length scales and amplitudes.
If instead of $[\Dl \diamond W_1]$ one uses $\SimTV$ as data fidelity in \eqref{eqn:cartoonTextureDecomp} one obtains the well known $L^1$-TV regularization \cite{TVL12005}, where the $L^1$ term cuts off extreme values and high frequencies to reduce the total variation.
Standard $W_1$ as fidelity in \eqref{eqn:cartoonTextureDecomp} was also studied in \cite{LellmannKantorovichRubinstein2014} where the weight of the $\SimTV$ term was often set so high as to essentially act like $\SimDisc$, that is, there was no separate noise component. High frequencies can be eliminated by $W_1$ with little transport cost. Due to conservation of mass, extreme values are not cut off, but `pressed flat', that is, the mass is distributed more evenly around the extreme value (cf.~the alternating lines in the low frequency area of the texture component).
Using $\lambda_0 \SimTV \diamond W_1$ (with sufficiently small $\lambda_0$), both effects are combined: high frequencies are absorbed by $W_1$, extreme values are truncated by $\SimTV$.
Overall, the effect of regularization is relatively independent of the oscillation amplitudes.

With $\lambda_0 \SimFR \diamond W_1$ the situation is qualitatively different: since the integrand $c(1,\cdot)$ of $\SimFR$ is flat around 1 (cf.~Figure \ref{fig:cost}), low amplitudes can be equalized by $\SimFR$ with little cost and thus appear in the noise component.
For this reason, $\SimJS$, $\SimChi$ and $\SimEp$ produce visually similar results to $\SimFR$.

Examples with added noise are illustrated in Figure \ref{fig:CartTextNoise2}.
Gaussian noise is more consistently separated by $\SimFR$, whereas $\SimTV$ still only filters extreme values. Conversely, for high peaked salt \& pepper noise, $\SimTV$ captures more of the spikes than $\SimFR$. Additionally, due to the non-linearity of $\SimFR$ it is very costly to fully remove a concentrated high spike. Therefore, in this case the peaks appear `smeared out' in the noise component, which is subsequently corrected for by $W_1$.
More generally, image features with high total variation, but with a local zero mean against a constant background can easily be flattened by the $W_1$ term, hence the distinction of noise and texture is not always straightforward.
Despite this ambiguity, both unbalanced models handle the salt \& pepper noise clearly better than the standard $W_1$ fidelity.
The behaviour of $\SimTV$ and $\SimFR$ on Gaussian and salt \& pepper noise is consistent with previous observations in image denoising.
Again, the other discrepancy measures of Table \ref{tab:discrepancies} where $c(1,\cdot)$ is smooth around 1 yield results comparable to $\SimFR$.

\subsection{Enhancement of thin structures}
Thin elongated structures are difficult to extract from noisy images since they are frequently interrupted and do not provide a strong signal.
Here a curvature regularization may help to improve the continuation of elongated structures.
One particular choice is the $\TVX_0$ functional \cite{BrPoWi12}, which penalizes the number of direction changes by lifting the gradient $\nabla u$ of an image $u$ into roto-translation space $\Omega\times S^1$ and applying the regularization there.
For two weights $\lambda_0$, $\lambda_1 \geq 0$, the $\TVX_0^{\lambda_0,\lambda_1}$ functional of an image $u \in L^1(\Omega)$ is given by
\newcommand{\liftedmeas}{\mc{G}_u}
\begin{equation}	
	\label{eqn:TVX0}
	\TVX_0^{\lambda_0,\lambda_1}(u) = \inf_{
			\xi \in\liftedmeas
			}
		\left( \lambda_0 \cdot \|\xi\|_{\meas}
			+ \frac{\lambda_1}{2} \sup_{\substack{
					\chi \in C^\infty_c(\Omega \times S^1):\\
					\|\chi\|_\infty \leq 1}}
				\int_{\Omega \times S^1} \nabla_x \chi(x,\vartheta) \cdot \vartheta \, \d\xi(x,\vartheta)
			\right)
\end{equation}
where
\begin{multline}
	\label{eqn:TXV0Set}
	\liftedmeas = \left\{\vphantom{\int_\Omega}
		\xi\in \measp(\Omega \times S^1) \right|\left.
		\int_\Omega u\,\div \gamma \, \d x + \int_{\Omega \times S^1} \gamma(x) \cdot \vartheta^\perp \d \mu(x,\vartheta) = 0 \right. \\
		\left. \vphantom{\int_\Omega} \tn{for all } \gamma \in C_c^\infty(\Omega,\R^2) \right\}
\end{multline}
and $\vartheta^\perp$ denotes the rotation of orientation $\vartheta \in S^1$ by $\pi/2$.
Intuitively, the set $\liftedmeas$ consists of measures $\xi$, where $\xi$ is the roto-translational lifting of the (distributional) gradient of $u$, rotated pointwise by $\pi/2$.
The $\lambda_0$ term in \eqref{eqn:TVX0} yields the total variation semi-norm of the image $u\in L^1(\Omega)$ (which equals the total variation of its distributional gradient). When $u$ is the indicator function of a polygon, the $\lambda_1$ term penalizes each vertex by $\lambda_1$.
For more details we refer to \cite{BrPoWi12}.

Based on this, for an observed image $\rho_0 \in \measp(\Omega)$ we consider the minimization problem
\begin{equation}\label{eqn:TVX0OptProblem}
	\min_{u \in L^1_+(\Omega)}
		W_{h,g,B}(\rho_0,u) + \TVX_0^{\lambda_0,\lambda_1}(u)\,,
\end{equation}
where in the first summand we interpret $u$ as a measure.

For discretization we consider 16 directions $s_1,\ldots,s_{16}$ in $S^1$ (more directions are possible in a similar way), namely the directions from a pixel to all the 16 second nearest neighbour pixels,
counting counter-clockwise from $s_1=(1,0)$.
Let $\hat s_k=\frac{\Delta x}{|s_k|_\infty}s_k$ with $|\cdot|_\infty$ the supremum norm in $\R^2$, and define for $k \in \{1,\ldots,16\}$
\begin{equation*}
x_{ijk}=x_{ij}+\tfrac{\Delta x}{2}
\begin{cases}
(1,(i \bmod 2)+\frac12)&\text{if }k\in\{2,8,10,16\},\\
((j \bmod 2)+\frac12,1)&\text{if }k\in\{4,6,12,14\},\\
(1,1)&\text{else.}
\end{cases}
\end{equation*}
Now $\xi$ is discretized as
\begin{equation*}
\xi=\sum_{i,j=1}^N\sum_{k=1}^{16}\xi_{ijk}\hat l_{ijk}\,,
\qquad
\hat\xi=(\xi_{ijk})_{i,j=1,\ldots,N}^{k=1,\ldots,16}\,,
\end{equation*}
where $\hat l_{ijk}$ is the unit line measure starting at $(x_{ijk}-\frac{\hat s_k}2,s_k)$ and ending at $(x_{ijk}+\frac{\hat s_k}2,s_k)$.
With this discretization we obtain the identity
\begin{align*}
\int_{\Omega\times S^1}\nabla_x\chi(x,\vartheta)\cdot\vartheta\,\d\xi(x,\vartheta)
&=\sum_{i,j=1}^N\sum_{k=1}^{16}\xi_{ijk}\left[\chi(x_{ijk}+\tfrac{\hat s_k}2,s_k)-\chi(x_{ijk}-\tfrac{\hat s_k}2,s_k)\right]\,.
\end{align*}
Hence, we discretize $\chi$ as the collection of its values at all points $(x_{ijk}\pm\frac{\hat s_k}2,s_k)$.
Moreover, we introduce a variable $\gamma$ which acts as Lagrange multiplier for constraint \eqref{eqn:TXV0Set}. For simplicity, in this application we choose the forward differences discretization \eqref{eqn:forwardDifference}. For the discrete saddle point problem \eqref{eqn:FullDiscreteProblem} we then find
\begin{align*}
\hat\zeta&=(\hat\chi,\hat\gamma)\\
F_2(\hat\zeta)&=\iota_{|\cdot|\leq\lambda_1/2}(\hat\chi)\\
G_2(\hat\rho,\hat\xi)&=\textstyle\iota_{\geq0}(\hat\rho)+\iota_{\geq0}(\hat\xi)+\lambda_0\sum_{i,j=1}^N\sum_{k=1}^{16}|\hat s_k|\xi_{ijk}\\
K_2&=-\hat \nabla\\
\langle \hat \zeta, K_3 \, \hat\xi \rangle &= \langle \hat \chi, B_1 \, \hat \xi \rangle
	+ \langle \hat \gamma, B_2 \, \hat \xi \rangle
\end{align*}
for the operators $B_1$ and $B_2$ given by
\begin{align*}
(B_1^T\hat\chi)_{ijk}&=\chi(x_{ijk}+\tfrac{\hat s_k}2,s_k)-\chi(x_{ijk}-\tfrac{\hat s_k}2,s_k)\,,\\
(B_2\,\hat\xi)_{ij}&=\textstyle\sum_{i=1}^{16} \hat s_k \, \xi_{ijk}\,.
\end{align*}

An example where a vertical line has to be reconstructed from partial observations is shown in Figure \ref{fig:TVX0_1}. The discrepancy measure $\SimTV$ has no notion of `moving mass a little'. Consequently, when using $\SimTV$ as data fidelity term in \eqref{eqn:TVX0OptProblem}, it fails to assemble coherent structures. For low weight the whole signal is removed, for intermediate weight the salt and pepper noise is filtered, for high weight the image remains unchanged.
When using $W_1$ as fidelity measure, the chunks at the image center can be rearranged with little cost to form a straight line. The salt and pepper noise is spread out to reduce the regularization penalty, but remains as blur in the background.
Finally, with $W_{2\SimTV,\SimDisc,\SimDisc}$, the $\SimTV$-term can cope with the salt and pepper noise, and the transport term(s) can form the line, combining the advantages of both fidelity terms.
\begin{figure}
	\centering
	\includegraphics{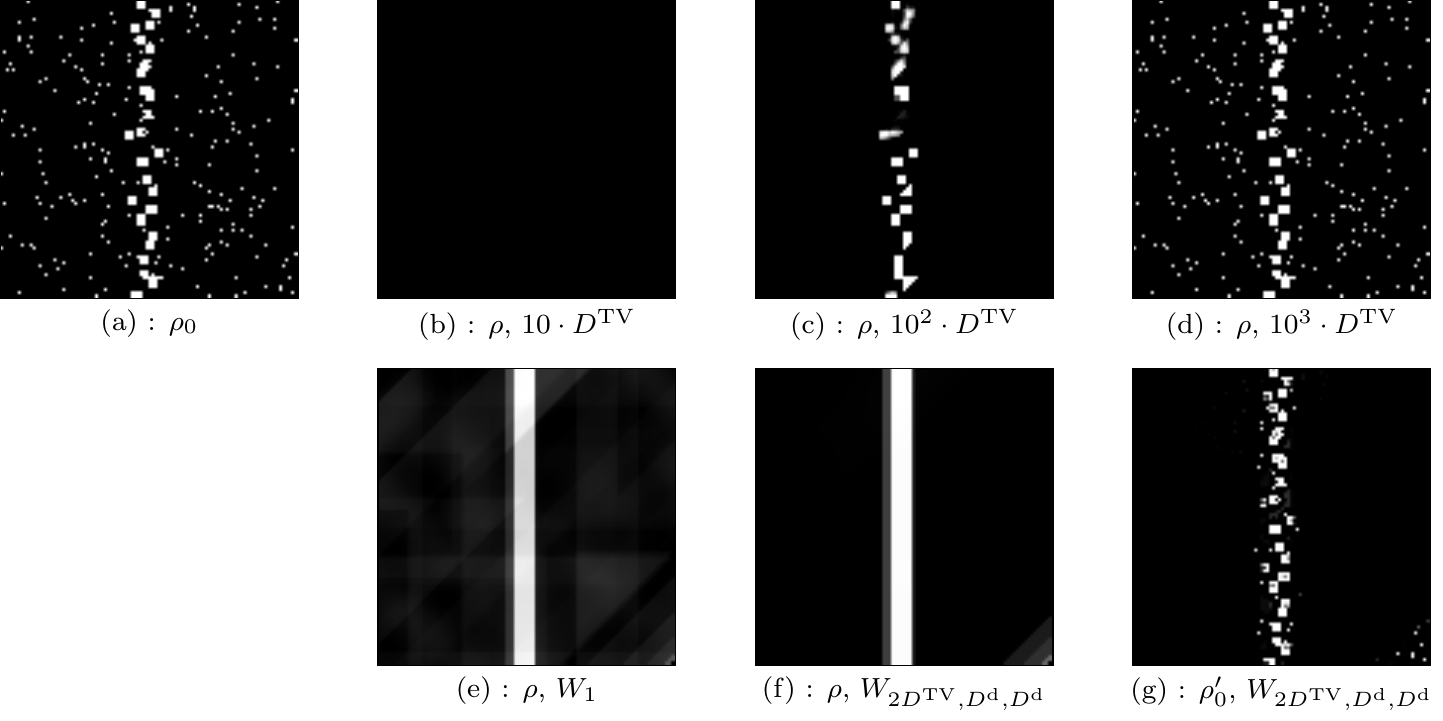}
	\caption{Reconstructing a line from a distorted observation with $\TVX_0^{\lambda_0,\lambda_1}$ regularization. \textit{(a):} input image $\rho_0$, %
	\textit{(b)-(d):} reconstruction $\rho$ with $\SimTV$ as data fidelity in \eqref{eqn:TVX0OptProblem} with different weights. %
	\textit{(e):} reconstruction with $W_1$ as fidelity term. %
	\textit{(f):} reconstruction with $W_{2 \SimTV,\SimDisc,\SimDisc}$. %
	\textit{(g):} intermediate image $\rho_0'$ in infimal convolution \eqref{eq:SandwichProblem}, illustrating the impact of the $\SimTV$ term.}
	\label{fig:TVX0_1}
\end{figure}

\begin{figure}
	\centering
	\includegraphics{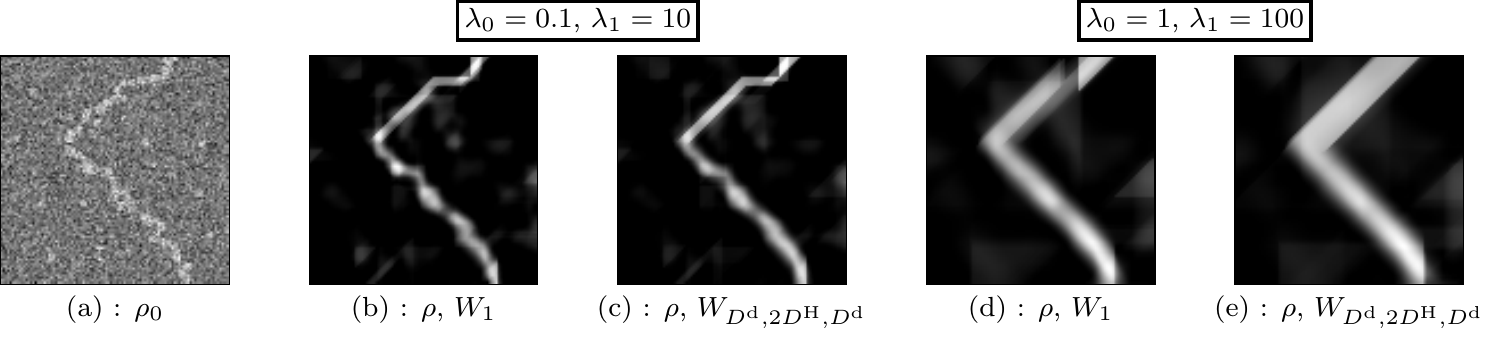}
	\caption{Reconstruction $\rho$ from noisy observation $\rho_0$ with $\TVX_0^{\lambda_0,\lambda_1}$ regularization via \eqref{eqn:TVX0OptProblem} for different transport fidelities and regularization strengths.}
	\label{fig:TVX0_2}
\end{figure}

The difference between standard $W_1$ and unbalanced transport as fidelity term and the impact of the regularization parameters $(\lambda_0,\lambda_1)$ in \eqref{eqn:TVX0OptProblem} are further illustrated in Figure \ref{fig:TVX0_2}.
For both, low and high regularization, $W_{\SimDisc,2\SimFR,\SimDisc}$ has less clutter in the background and reconstructs a more homogeneous line.
As before (Section \ref{sec:flow}), choosing $\SimFR$ in the middle allows for smooth and symmetric balancing of local mass fluctuations.
By increasing the regularization parameters $(\lambda_0,\lambda_1)$ a more schematic line is obtained. The broadening of the line for strong regularization is a typical effect of combining transport fidelity and $\TV$-type regularization and has already been observed in \cite{LellmannKantorovichRubinstein2014}.

%% file: 06-conclusion.tex
\section{Conclusion}\label{sec:conclusion}

Determining the Wasserstein-1 transport distance can be rewritten as a compact minimal cost flow problem which can be solved more efficiently than general optimal transport problems.
However, the restriction to balanced measures is impractical in many applications.
In this article we proposed two ways to generalize $W_1$ transport to unbalanced measures: First, via the efficient dual formulation, containing only local constraints. Second, by an intuitive infimal convolution-type combination with purely local discrepancy measures. We showed that (under suitable assumptions) both classes are actually equivalent and instances can be transformed into each other.
This allows to combine intuitive modelling of unbalanced transport-type discrepancy measures with efficient numerical methods.
Several examples, a discretization scheme and a simple primal-dual algorithm are discussed and numerical experiments demonstrate the usefulness of unbalanced $W_1$-type discrepancy measures.

It would be interesting to obtain a little more intuition on how a particular local discrepancy measure $D$ influences the unbalanced transport depending on its position in the infimal convolution with $W_1$.
As a first step in this direction we will in future work investigate a dynamic reformulation for the proposed models or a subclass thereof, which might separate mass change and mass transport effects in time.
Another interesting question is concerned with the (appropriately rescaled) limit of infinitely many alternating convolution-type combinations of local discrepancy measures and $W_1$.

Bernhard Schmitzer has been supported by the European Research Council (ERC project SIGMA-Vision). 